\documentclass[reqno, 11pt]{amsart}

\usepackage[T1]{fontenc} 
\usepackage[latin1]{inputenc} 
\usepackage{mathptmx} 
\usepackage{lmodern} 
\usepackage{amsthm} 
\usepackage{amsmath} 
\usepackage{amsfonts} 
\usepackage{amssymb} 
\usepackage{mathrsfs}
\usepackage{bbm}
\usepackage[left=2.1cm, right=2.1cm, top=2.1cm, bottom=2.1cm]{geometry} 
\usepackage[ngerman,german,english]{babel} 
\usepackage{paralist} 
\usepackage{graphicx} 
\usepackage[usenames,dvipsnames]{color}
\usepackage{etoolbox}
\usepackage{caption}
\usepackage{subcaption}

\usepackage[colorlinks=true, pdfstartview=FitV, linkcolor=BrickRed,citecolor=black, urlcolor=black]{hyperref}
\hypersetup{linktocpage}

\numberwithin{equation}{section}
\theoremstyle{plain} 
\newtheorem{thm}{Theorem}[section] 
\newtheorem{cor}[thm]{Corollary}
\newtheorem{lem}[thm]{Lemma}
\newtheorem{pro}[thm]{Proposition}

\theoremstyle{definition}
\newtheorem{defi}[thm]{Definition}
\newtheorem{rem}[thm]{Remark}

\theoremstyle{remark}

\newcommand{\C}{\mathbb{C}}
\newcommand{\N}{\mathbb{N}}

\newcommand{\R}{\mathbb{R}}

\renewcommand{\P}{\mathcal{P}}

\renewcommand{\S}{\mathcal{S}}
\newcommand{\K}{\mathcal{K}}

\newcommand{\M}{\mathcal{M}}
\newcommand{\Nmom}{\mathcal{N}}
\newcommand{\CO}{\mathbb{K}}

\newcommand\restr[2]{{
		\left.\kern-\nulldelimiterspace 
		#1 
		\vphantom{\big|} 
		\right|_{#2} 
}}

\newcommand{\norm}[2][]{\left\|#2\right\|_{#1}}

\newcommand{\dualbra}[2]{\left\langle #1,#2 \right\rangle}

\newcommand{\Lpnorm}[2][]{\ifthenelse{\equal{#1}{}}{\norm{#2}_{L^p}}{\norm{#2}_{L^p(#1)}}}
\newcommand{\Hknorm}[2][]{\ifthenelse{\equal{#1}{}}{\norm{#2}_{H^k}}{\norm{#2}_{H^k(#1)}}}

\newcommand{\QV}[2][]{\ifthenelse{\equal{#1}{}}{\langle #1 \rangle}{\langle #1,#2 \rangle}}

\newcommand{\ind}{\mathbbm{1}}

\newcommand{\B}{\mathscr{B}}
\newcommand{\Meas}{\mathscr{M}}

\newcommand{\Rem}{\mathscr{R}}
\newcommand{\Err}{\mathscr{E}}
\newcommand{\Lop}{\mathcal{L}}

\newcommand{\loc}{\operatorname{loc}}
\newcommand{\sym}{\operatorname{sym}}
\newcommand{\Lip}{\operatorname{Lip}}

\title[Long-time behaviour of rouleau formation models]{Long-time behaviour of rouleau formation models}

\author[E. Franco]{Eugenia Franco}
\address{Institute for Applied Mathematics, University of Bonn}
\email{franco@iam.uni-bonn.de}

\author[B. Kepka]{Bernhard Kepka}
\address{Institute of Mathematics, University of Z\"urich}
\email{bernhard.kepka@math.uzh.ch}

\subjclass[2020]{92-10, 35R09, 82C21}

\keywords{Coagulation equation, self-similar asymptotics, localization, rouleaux model}

\begin{document}

\begin{abstract}
In this paper we study a two-component coagulation equation that models the aggregation of rouleaux in blood. We consider product kernels that have homogeneity $2$ and we characterize the initial data that lead to gelation. 
We prove that, when gelation occurs, the solution to the two-component coagulation equation localizes along a direction of the space of cluster as $ t $ approaches the gelation time $0 < T_* < \infty $. 
The localization direction is determined by the initial datum.
We also prove that the solution converges to a self-similar solution along the direction of localization. 
\end{abstract}
\maketitle

\tableofcontents

\section{Introduction}\label{sec:Intro}
\subsection{State of the art}
The classical Smoluchowski's coagulation equation 
\[ 
\partial_t f(t,x)  =\frac{1}{2} \int_0^x   K(x, y)  f(t,  y ) f(t,x-y ) dy   - f(t, x)  \int_0^\infty  K(x,y) f(t, y ) dy , \quad f\mid_{t=0}=f_0,
\]
is a mean-field model that was derived in order to describe the evolution in time of the size distribution $ f(t,x)$ of a system of spherical particles that coalesce upon binary collision. 
The coagulation kernel $ K(x,y) $ describes the rate at which a particle of size $ x>0 $  merges with a particle of size $ y>0 $. The microscopic mechanisms behind coagulation events are summarized by the coagulation kernel, which is assumed to depend only on the size of the coalescing particles. 

Smoluchowski's coagulation equation (and variations of that) has been used in order to model different phenomena. For instance to model the coagulation of aerosols in the atmosphere \cite{friendlander2000smoke}, the grouping of animals \cite{gueron1995dynamics}, hemagglutination \cite{PerelsonSamsel1982} and polymerization processes \cite{blatz1945note}.

A generalization of the classical Smoluchowski's coagulation equation that has been studied by several authors is the multicomponent coagulation equation (see \cite{bertoin2009two,cristian2024coagulation,delacour2022mathematical,ferreira2021localization,ferreira2024asymptotic,ferreira2025global,krapivsky1996aggregation,norris2000cluster}). 
The solution $f(t, z) $ to a $n$-component coagulation equation represents the concentration of particles of type $z \in  \R_+^n $ for $n \geq 1 $ at a certain time. 
The variable $z\in  \R_+^n    $ could represent different features. For instance we can assume that particles are aggregates of  different types of molecules. In this case the vector $z =(z_i)_{i=1}^n  $ represents the chemical compositions of clusters. 
Another possibility is that each particle is characterized by its size and its geometry. We could then assume that each particle is described by the vector $ z =(z_1, z_2 ) $, where $z_1 $ is the volume and $z_2 $ is a parameter that describes the geometry of the particle.  

The existence of time dependent solutions to a general class of  multicomponent coagulation equations is studied in \cite{ferreira2025global,norris2000cluster}. 
The existence results proven in \cite{ferreira2025global,norris2000cluster} include kernels that exhibit gelation, i.e. loss of mass conservation in finite time due to the formation of a particle of infinite  size. 

A relevant feature of multicomponent coagulation equations is the so-called \textit{localization} phenomena. 
In particular, localization occurs when the solution to the multicomponent coagulation equation localizes asymptotically (as time tends to infinity) along a line that depends on the initial condition.  
Localization results have been proven for a class of multicomponent coagulation equations that do not exhibit gelation in \cite{ferreira2021localization,ferreira2024asymptotic}. 
A localization result has been proven in \cite{hoogendijk2024gelation} for a class of kernels that give rise to gelation and very specific initial data (monodisperse initial data). 

A common assumption concerning the coagulation kernel is the homogeneity of the kenerl, i.e. 
\[ 
K(ax, a y ) = a^\gamma K (x,y) , \quad  a,\,  x,\,  y \geq 0 
\]
for some $\gamma \in \R $. An important class of solutions to the classical coagulation equations with homogeneous kernels are the so-called \textit{self-similar} solutions of the form
\[ 
f(t, x) = s(t)^\alpha  \Phi \left( \frac{x}{s(t)} \right) 
\]
where $ \alpha \in \R $ and $ s(t) $ is a suitable function of time. 
The precise form of $s(t) $ and $\alpha $ can be obtained by a dimensional analysis. Such self-similar solutions are expected to describe the long-time behaviour of solutions to the coagulation equation. However, the domain of attraction of self-similar solutions  has been characterized only for the solvable kernels, i.e. for $ K(x,y)=1 $, $K (x,y) = x+y$, $K(x,y) = x y $, using Laplace transform methods (see \cite{menon2004approach}).

In \cite{ferreira2024asymptotic} the long time asymptotics of the solution to a class of multicomponent coagulation equations with a specific choice of coagulation kernels has been analysed. 
In particular it is proven in \cite{ferreira2024asymptotic} that, for kernels that are constant along each direction (in the state space) the solution localizes along a line as time tends to infinity and it converges to a self-similar solution along that line.

Finally, a multicomponent coagulation equation describing the coagulation of a system of particles with ramifications has been formulated and studied by Bertoin (see \cite{bertoin2009two}). 
In that paper explicit solutions are obtained under specific assumptions on the initial condition and the occurrence of gelation is  analysed in detail.

\subsection{Discrete and continuous coagulation equation for rouleaux}
In this paper we are interested in modelling the coagulation of rouleaux in blood. A rouleau is a stack of red blood cells, also called erythrocytes. The erythrocytes have a specific shape; they are biconcave thin discs. 
The aggregates that erythrocytes form appear to have very ramified structures due to various possible coagulation events. For instance, two erythrocytes can adhere on their faces, forming a cylinder. Another possibility is that an erythrocyte adheres to the wall of a rouleau. The goal of this paper is to formulate a multicomponent coagulation equation that allows to study the evolution in time of the geometry and of the size of rouleaux. 

A model describing the coagulation of rouleaux in blood has been proposed and analysed by Perelson and Samsel in \cite{PerelsonSamsel1982}.
In this paper, a system of ordinary differential equations describing the concentration of erythrocytes and the concentration of (possibly ramified) rouleaux is formulated and studied numerically. In \cite{PerelsonSamsel1984} the model studied in \cite{PerelsonSamsel1982} was generalized to describe a system of rouleaux that coalesce and fragment. Specifically the stability of steady states is analysed in \cite{PerelsonSamsel1984} under the assumption that the fragmentation and coagulation rates satisfy the detailed balance assumption.

The goal of this paper is to formulate and analyse a class of multicomponent coagulation equations that describe the dynamics of rouleaux taking into account their shape. More precisely, we assume that each particle, that is rouleau, is characterized by three variables $ (c,a,s) $ where $ a \geq 0 $, $ c \geq 2 $ and $ s\geq1 $. To this end, we associate to each rouleau a tree with vertices with degree $ d \in \{1,2,3\} $. The variable $s$ represents the size of the rouleau, i.e. the number of edges in the tree. The variable $c $ is the number of faces of the rouleau and it is identified with the number of vertices with degree $1$, i.e. the leafs of the tree. Finally, $a \geq 0 $ is the number of sites on the wall of a rouleau, to which another rouleau, or an erythrocyte, can adhere. Alternatively, we can identify $a$ with the number of edges of degree $ 2 $ in the tree associated with the rouleau. A monomer (i.e. an erythrocyte) is characterized by the triple $(2,0 ,1) $.

Based on this, we model three types of coagulation events.
\begin{enumerate}[(1)]
	\item Two rouleaux can adhere on their faces:  
	\[
	R_{1} : (c_1, a_1, s_1 ) + (c_2 , a_2, s_2  )   \mapsto ( c_1+ c_2 - 2 , a_1+ a_2 + 1,  s_1 + s_2 ), \quad c_1 \geq 2,\,  c_2 \geq 2,\,  a_1 \geq 0,\,  a_2 \geq 0. 
	\]
	We assume that when two rouleaux adhere on a face a site for the attachment of a rouleau is formed at the junction. If we identify the rouleau with a tree, then this coagulation consists in the attachment of one leave of one tree to a leaf of the other tree. The vertex that is formed due to the coagulation of these two leafs is a vertex of degree $2 $ (i.e. an arm).  We refer to Figure \ref{fig1} for an example of coagulation of type $ R_1 $. 
	
	\begin{figure}
		\centering
		\includegraphics[width=0.4\linewidth]{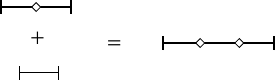}
		\caption{Coagulation of the first type. Here the vertices of degree one are denoted with "$|$", the ones of degree two with an empty diamond $\lozenge$.
		}\label{fig1}
	\end{figure}
	
	\item The second type of coagulation event that we model is that the face of a rouleau adheres on a free site of another rouleau:  
	\[
	R_{2}:  (c_1, a_1, s_1 ) + (c_2, a_2,s_2)   \mapsto ( c_1+ c_2 - 1, a_1+ a_2 - 1  , s_1 + s_2 ), \quad c_1 \geq 2,\, c_2 \geq 2,\, a_1 \geq 1,\, a_2 \geq 1. 
	\]
	In this coagulation event a leaf of a tree is attached to a vertex of degree two of another tree. Notice that during this type of coagulation a vertex of degree three is formed due to the coalescence of a vertex of degree one and a vertex of degree two. We refer to Figure \ref{fig2} for an example of coagulation of type $R_2$. 
	
	\begin{figure}
		\centering
		\includegraphics[width=0.4\linewidth]{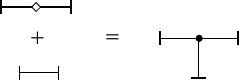}
		\caption{ Coagulation of the second type. Here the vertices of degree one are denoted with "$|$", the one of degree two with an empty diamond $\lozenge$. The vertices with degree three with a $\bullet$.
		}\label{fig2}
	\end{figure}

	\item Finally, we assume that free sites of two rouleaux can adhere: 
	\[
	R_{3} : (c_1, a_1, s_1 ) + (c_2 , a_2, s_2  )   \mapsto ( c_1+ c_2  , a_1+ a_2-2,  s_1 + s_2 + 1 ), \quad c_1 \geq 2,\, c_2 \geq 2,\, a_1 \geq 2,\, a_2 \geq 2. 
	\]
	This coagulation consists in the attachment of a vertex of degree two of a tree to a vertex of degree two of another tree. During this type of coagulation we assume that two vertices of degree three are formed and an edge is added to the tree. We refer to Figure \ref{fig3} for an example. 
	\begin{figure}
		\centering
		\includegraphics[width=0.4\linewidth]{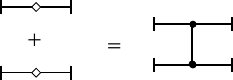}
		\caption{Coagulation of the third type. Here the vertices of degree one are denoted with "$|$", the one of degree two with an empty diamond $\lozenge$. The vertices with degree three with a $\bullet$.
		}\label{fig3}
	\end{figure}
\end{enumerate}
As mentioned before, rouleaux are stacks of erythrocytes that are formed by a sequence of reactions of the form $ R_{1}$, $R_{2} $, $R_{3} $. 
The triple $(2, 0 , 1 ) $ that characterize the erythrocytes satisfy the condition 
\begin{equation} \label{intro:cons}
a + 2 c = s + 3. 
\end{equation}
On the other hand, all the reactions introduced above conserve the law \eqref{intro:cons}. This conservation law allows us to assume that particles are characterized only by the variables $c $ and $a$. 

We now describe our choice of coagulation rates, yielding the form of the coagulation kernel. We assume that the coagulation event $R_{1}$ takes place at rate that is proportional to the product of the number of faces of the two rouleaux, more precisely we assume that 
\[
	\K_{1} ((c_1, a_1 ) , (c_2, a_2 ) ) = c_1 c_2.
\] 
Similarly, we assume that the rate at which the reaction $R_{ca } $ takes place is given by 
\[
	\K_{1} ((c_1, a_1 ) , (c_2, a_2 ) ) = \dfrac{1}{2}\left( c_1 a_2 + c_2 a_1 \right).
\] 
 and the rate at which the reaction $R_{3} $ takes place is 
\[
	\K_{3} ((c_1, a_1 ) , (c_2, a_2 ) ) = a_1 a_2.
	\] 
We stress that the choice of the coagulation kernels that we make in this paper is in agreement with the choices made in \cite{bertoin2009two,PerelsonSamsel1982}. However, it would be relevant to derive the coagulation kernels starting from a precise analysis of the mechanisms behind the coagulation of rouleaux in blood.  

The multicomponent coagulation equations that we study in this paper is the following 
\begin{align} \label{intro:dis_strong_eq}
	\begin{split}
		\partial_t f (t,v ) &= \frac{1}{2}  \sum_{ i =1 }^3  \sum_{  v' \in B_i^{(d) } (v)   } \alpha_i f(t, v- v' + \xi_i ) f(t,  v' ) \K_i ( v-v' + \xi_i, v' ) \\
		&- f(t, v )  \sum_{  v' \in \S^{(d) }  }  f( t, v' ) \sum_{ i =1 }^3 \alpha_i \K_i (v, v' )
	\end{split}
\end{align}
where we are using the following notations: $\alpha \in \R^3_+ $ is such that $\alpha \neq 0 $,  $v=(c,a)^\top \in \N_{\geq 2 } \times \N_{\geq 2 } $, 
\begin{align}\label{eq:sec1:DefXi}
    \xi_{1} = \left( \begin{matrix}
        -2 \\ 1
    \end{matrix} \right) , \quad  \xi_{2} = \left( \begin{matrix}
        -1 \\ - 1
    \end{matrix} \right) , \quad  \xi_{3} = \left( \begin{matrix}
        0 \\ -2 
    \end{matrix} \right), 
\end{align}
and 
\begin{align*}
	\S^{(d)}= \N_{\geq 2 }\times \N_{\geq 2 }, \quad B_i^{(d) } (v) := \{ v' \in S^{(d)} : v-v' + \xi_i \in \S^{(d)}  \}. 
\end{align*}
The solution $f(t, v) $ to equation \eqref{intro:dis_strong_eq} describes the evolution in time of a system of rouleaux that coalesce according to the coagulation mechanisms $ \{ R_{i} \}_{ i =1 }^3 $. The vector $ \alpha\in \R^3_+ $ allows to give different weights to the coagulation events $ \{ R_{i} \}_{ i =1 }^3  $. In particular, setting $ \alpha_j=0 $ allows to to switch off individual coagulation mechanisms.

Notice that we assumed the state space to be $  \N_{ \geq 2 } \times  \N_{ \geq 2 } $ even if coagulation events of type $ i \in \{ 1, 2 \} $ can occur between clusters $(c, a) $ with $ a \in \{ 0 , 1 \}$. This is an assumption that we make in order to simplify our analysis. However, we expect to have similar results if we assume that the state space is $ \N_{\geq 2}\times \N $.

As is standard in the theory of coagulation models, we can extend the previous model to a continuous coagulation model, that is, the discrete state space $ \S^{(d)} $ is replaced by 
\begin{align*}
 \S = [2, \infty)\times [2, \infty).
\end{align*}
The continuous version of the discrete coagulation equation \eqref{intro:dis_strong_eq} is then given by
\begin{align}\label{sec1:CoagulationEq}
	\partial_t f = \CO(f,f), \quad f\mid_{t=0}=f_0,
\end{align}
where for $ z=(x,y) \in \S $
\begin{align}\label{sec1:CoagulationOp}
    \begin{split}
        \CO(f,f)[z] &= \frac{1}{2} \sum_{ i=1 }^3 \alpha_i\int_{ B_z^i }  \K_i( z', z- z' - \xi_i  ) f(t,  z' ) f(t, z- z' - \xi_i  ) dz'  - f(t, z)  \sum_{ i=1 }^3 \alpha_i\int_{\S} \K_i(z,z') f(t,z') dz'.
    \end{split}
\end{align}
Here, we are using the abbreviation in \eqref{eq:sec1:DefXi} and 
\[ 
B^i_z := \{ z'=(x', y' ) \in \S : x' \leq \xi_{i, 1} + x ,\  y' \leq \xi_{i, 2} + y   \}, 
\]
as well as
\begin{align}
	\K_{1}(z,z') &= z^\top K_{1}z'= x x', \quad K_{1}= \left( \begin{matrix}
		1 & 0 \\ 0 & 0
	\end{matrix} \right), \label{kernel_cc}
	\\
	\K_{2}(z,z') &= z^\top K_{2}z'= \dfrac{1}{2}\left(x y'+x'y\right), \quad K_{2}= \dfrac{1}{2} \left( \begin{matrix}
		0 & 1 \\ 1 & 0
	\end{matrix} \right), \label{kernel_ca}
	\\
	\K_{3}(z,z') &= z^\top K_{3}z'= y y', \quad K_{3}= \left( \begin{matrix}
		0 & 0 \\ 0 & 1
	\end{matrix} \right). \label{kernel_aa}
\end{align}
In this paper we study measure-valued solutions to \eqref{sec1:CoagulationEq}. In particular, choosing measures supported on $ \S^{(d)} $ allows to embed the dynamics of the discrete model \eqref{intro:dis_strong_eq} into the continuous one.

We conclude this section discussing the conservation laws associated with equation \eqref{sec1:CoagulationEq}. 
Notice that if $ \alpha_i>0 $, for all $ i=1,2,3 $, then the system does not have any conserved quantity. This is due to the fact that the coagulation events $ \{ R_{i} \}_{ i =1 }^3 $ do not have a common conserved quantity. Note that such a conserved quantity is necessarily linear in the variables $ (c,a,s) $.

It is important to mention that the coagulation kernels that we consider have homogeneity $ \gamma =2 $. Therefore, we expect gelation to occur. For the classical coagulation equation an indicator for this phenomenon is the violation of mass conservation after the so-called gelation time $ T_* \in [0, \infty) $. This is due to the formation of particles of infinite size at time $T_* $ and so the flux of mass towards infinity is positive at time $ t= T_* $. As we will show, this is in fact the case for our rouleau formation models with the exception of specific initial conditions. A similar type of gelation phenomena, for systems that do not have conserved quantities, occurs for coagulation equation with source (see for instance \cite{ferreira2021stationary}). 

\subsection{Main results: localization and asymptotic self-similarity}
In this paper we study the existence and the uniqueness of a time dependent solution to equation \eqref{sec1:CoagulationEq}. 
Notice that the coagulation kernels \eqref{kernel_cc}-\eqref{kernel_ca}-\eqref{kernel_aa} that we consider  in this paper have homogeneity $\gamma =2 $ and could produce gelation. In Proposition \ref{pro:BlowUp} we determine conditions on $\alpha $ and on the initial datum that induce gelation. We also identify a particular case in which gelation does not occur (see Remark \ref{rem:no gel}). 

One of the  results that we prove in this paper is the localization of the time dependent solution that takes place as $t $ approaches the gelation time $T_*< \infty $. We study the localization phenomena only under the assumption that gelation occurs in finite time.  
In order to explain the localization result that we prove it is convenient to rescale the time dependent solution $f (t,  z ) $ according to the self-similar change of variables 
\begin{equation} \label{intro:selfsim}
	F(\tau , \eta ) := (T_* - t )^{ - 7 } f( t , z) , \quad t= t(\tau)= T_* (1- e^{- \tau }) , \ z =  ( T_* - t )^2 \eta. 
\end{equation}
In Subsection \ref{subsec:scaling} we motivate the choice of the self-similar scaling using dimensional analysis arguments. 
We prove that there exists a direction $\omega_\theta = \theta/|\theta |\in  \R^2_+ $, that depends only on the initial condition $f_0 $ in \eqref{sec1:CoagulationEq} and on the parameter $\alpha $, along which the function  $|\eta |^2 F(\tau, \eta ) $ localizes. The localization takes place as $\tau \to \infty $, hence when $ |\eta | $ is large.
More precisely we prove that 
\[ 
\lim_{\tau \to \infty } \int_{ \R^2 } |\eta |^2 \norm{ \dfrac{\eta }{|\eta |} - \omega_\theta }^2 F(\tau, \eta ) d \eta =0 . 
\]

Moreover we also prove that under natural regularity assumptions on the initial datum $f_0 $ (i.e. we assume that $f_0$  has finite fourth moment), the function $ |\eta |^2 F ( \tau , \eta ) $ converges to a self-similar solution along the localization line $\omega_\theta $. More precisely, our result can be summarized as follows 
\[ 
 F ( \tau , \eta ) \rightarrow \dfrac{1}{|\eta|}\delta_{\omega_\theta } \left( \dfrac{ \eta  }{|\eta|} \right)  F_s (|\eta| ) \text{ as } \tau \to \infty.  
\]
The function $F_s  $ is given by 
\begin{equation} \label{McLeod}
	F_s( r):= \dfrac{1}{ \sqrt{2 \pi K_0} }  r^{ - 5/2} e^{- r/ (2K_0)} 
\end{equation}
where $K_0>0$ is a parameter that depends on the localization line $\omega_\theta $ and hence is determined by the initial datum and by the parameter $\alpha $. The precise statements are summarized in Theorem \ref{thm:IntrodLocaliSelfSim} below.

We recall that $F_s$ is the self-similar profile that characterizes the self-similar solution to the coagulation equation with product kernel and initial datum with finite moments (see \cite{menon2004approach}). This is expected as the localization result allows us to reduce the dynamics of the two-dimensional coagulation equation to a one-dimensional equation along the localization line.

In order to prove our result we study in detail the time evolution of the second moments associated with the solution to equation \eqref{sec1:CoagulationEq} written in self-similar variables, i.e. associated with $F$.
Since we are working with a two-dimensional multicomponent equation the second moment $\M^2$ is a matrix that satisfies a matrix Riccati equation. Analysing the equation in detail it is possible to see that the second moment tensorizes as  $t \to T_* $, i.e. 
\[ 
\M^2[F] (t) \rightarrow \theta \otimes \theta, \text{ as } t \to T_*. 
\]
Here $\theta \in  \R_+^2$ depends on the initial datum and on $\alpha $. 
The fact that the second moment is tensorized as $t \to T_* $ allows us to prove our localization result. 

For the convergence towards the self-similar solution along the localization line we adapt the approach used by Menon and Pego in \cite{menon2004approach} for the one-dimensional coagulation equation with product kernel. In particular, we use Laplace transform methods.

\bigskip 

\paragraph{\bf Plan of the paper.}
The paper is organized as follows.
In Section \ref{sec:main results} we present the main results that we prove in this paper.
In Section \ref{subsec:scaling} we present some heuristic arguments on the localization and on the choice of the self-similar variables. 
In Section \ref{sec:WellPosed} we prove the existence and uniqueness of time dependent solutions to equation \eqref{sec1:CoagulationEq}.  
In Section \ref{subsec:StudyMoments} we study in detailed the first four moments of the solution to equation \eqref{sec1:CoagulationEq}. 
In Section \ref{sec:localization} we infer the localization of the solution to equation \eqref{sec1:CoagulationEq} in self-similar variables. In Section \ref{subsec:SelfSimilarConvergence} we give a proof of the convergence to the self-similar profile.

\subsection{Notation} \label{sec:noation}
Here, we summarize some notations used in this paper. We will write $ \R_+=[0,\infty) $ and $R_*=(0,\infty)$. For $ z\in \R^2 $ we use two different norms. We write $ |z|=|z_1|+|z_2| $ for the $ 1 $-norm and $ \norm{z}=\sqrt{|z_1|^2+|z_2|^2} $ for the euclidean norm. For a vector $ z\in \R^2 $ we also use the notation $ z^\perp = (-z_2,z_1)^\top $. 

A notational convenience for the study of higher moments is the introduction of tensors. To this end, let $ V $ be a finite dimensional vector space. We mostly consider the case $ V=\R^2 $. For $ n\in \N $ denote by $ V^{\otimes n} $ the vector space of $ n $-tensors. They are defined as $ n $-linear forms $ V^n\to \R $. Given a basis $ (b_i)_{i=1}^d $ of $ V $ a basis of $ V^{\otimes n} $ is given by the set of pure tensors
\begin{align*}
	b_{i_1}\otimes \cdots \otimes b_{i_n}, \quad i_1,\ldots, i_n\in \{1,\ldots, d\}.
\end{align*}  
In particular, $ \dim V^{\otimes n} = d^n $. Given a basis of $ V^{\otimes n} $ a tensor $ T\in V^{\otimes n} $ can be identified with a family of real numbers $ T_{i_1,\ldots, i_n} $, $ i_1,\ldots, i_n\in \{1,\ldots, d\} $. Note that for $ n=2 $ one can identify $ 2 $-tensors with matrices. Via the coordinate representation we can also define a norm via 
\begin{align*}
	\norm{T} = \left( \sum_{i_1,\ldots, i_n\in \{1,\ldots, d\}} T_{i_1,\ldots, i_n}^2\right)^{1/2}.
\end{align*}
It is convenient to define for a vector $ v\in V $ the shorthand notation
\begin{align*}
	v^{\otimes n} = v\otimes \cdots \otimes v \quad \text{($ n $ times)}.
\end{align*}
Furthermore, we define the subspace of symmetric tensors $ V^{\otimes n}_{\sym} $, that is $ n $-linear forms symmetric in all $ n $ components. On the level of the coordinate representation a symmetric tensor $ T\in V^{\otimes n}_{\sym} $ satisfies for any permutation $ \sigma \in \S(n) $
\begin{align*}
	T_{i_1,\ldots, i_n} = T_{\sigma(i_1), \ldots, \sigma(i_n)}, \quad \forall i_1,\ldots, i_n\in \{1,\ldots, d\}.
\end{align*}
Note that $ \dim V^{\otimes n}_{\sym}=\binom{n+d-1}{n} $. Furthermore, we can define the projection $ \P_n: V^{\otimes n}\to V^{\otimes n}_{\sym} $ via
\begin{align*}
	\P_n(T)_{i_1,\ldots, i_n} = \dfrac{1}{n!}\sum_{\sigma \in \S(n)} T_{\sigma(i_1), \ldots, \sigma(i_n)}, \quad \forall i_1,\ldots, i_n\in \{1,\ldots, d\}.
\end{align*}
In fact, this definition is independent of the coordinate representation. 

Moreover, for $ \S\subset \R^2 $ we denote by $ \Meas(\S) $ the space of finite Borel measures on $ \S $. The space of non-negative measures is denoted by $ \Meas_+(\S) $. For $ p\geq0 $ we write $ \Meas_{p, + }(\S) $ for measures in $ \Meas_+(\S) $ having finite moments of order $ p\geq0 $. We equip the space $ \Meas(\S) $ with the weak topology of measures, i.e. in duality with bounded continuous functions $ C_b(\S) $. On the other hand, for $ \mu\in \Meas(\S) $ we denote by $ \norm{\mu} $ the total variation norm and $ \norm[p]{\mu}= \norm{(1+|\cdot|^p)\mu} $ for $ p> 0 $.

For some time interval $ I\subset \R_+ $, open or closed, we denote by $ C(I;\Meas(\S)) $ the space of functions continuous into $ \Meas(\S) $ on $ I $. Similarly, also for $ C(I;\Meas_+(\S)), \, C([0,T];\Meas_{p,+}(\S)) $. In addition, we write $ L^\infty(I;\Meas_+(\S)) $ for functions bounded into $ \Meas_+(\S) $. Let us mention that for $ \mu\in L^\infty(I;\Meas_+(\S)) $ this is equivalent to $ \sup_{t\in I}\norm{\mu(t)} $ due to $ \mu(t)\geq0 $. Moreover, we write $ L_{\loc}^\infty $ instead of $ L^\infty $ when boundedness holds on compact subsets of $ I $.

In the following we define the Laplace transform of measures. As is necessary for our study we introduce a vector-valued, desingularized variant of it. More precisely, for $ \zeta\in \R_+^2 $ and a measure $ \mu\in \Meas_{1,+}(\R^2_+) $ we define
\begin{align*}
	\hat{\mu}: \R_+^2\to \R^2 : \; \hat{\mu}(\zeta) = \int_{\R^2_+} z\left( 1-e^{-z\cdot \zeta} \right) \, \mu(dz).
\end{align*}
The desingularized variant allows for singularities at the origin $ z=0 $. Let us mention that by the Weierstrass approximation one can show that if $ \hat{\mu}_1(\zeta)=\hat{\mu}_2(\zeta) $ for all $ \zeta\in \R_{+}^2 $ and $ \mu_1({0})=\mu_2({0})=0 $, then $ \mu_1=\mu_2 $. Hence, the desingularized Laplace transform identifies the measure outside of the origin. 

We will make also use of the desingularized Laplace transform defined on measure on $ \S $. In this case, the Laplace transform identifies the measure completely, since $ 0\not\in \S $.

\section{Main results} \label{sec:main results}

In this paper we will work with measure-valued solutions to equation \eqref{sec1:CoagulationEq}. To this end, we will mostly work with the weak formulation of equation \eqref{sec1:CoagulationEq}. 
The corresponding weak formulation of the coagulation equation \eqref{sec1:CoagulationEq} is given by
\begin{align}\label{eq:sec1:ContWeakForm}
	\int_{\S}\varphi(z) \, f(t,dz) =\int_{\S}\varphi(z) \, f_0(dz) +  \dfrac{1}{2}  \int_0^t\int_{\S}\int_{\S} \sum_{i=1}^3 \alpha_i  \K_i(z,z') \, \Delta_i\varphi(z,z') \, f(s,dz)f(s,dz')\, ds, \quad t\geq 0,
\end{align}
for any $ \varphi\in C_b(\S) $. 
Accordingly, we define,  for all $ \varphi\in C_b(\S) $
\begin{align*}
	\Delta_{i}\varphi(z,z') &= \varphi\left( z+z'+ \xi_i   \right) - \varphi(z)-\varphi(z'), \quad (z,z')\in \S^2, \quad i \in \{ 1,2,3\}. 
\end{align*}
Observe that since we assume that $ (z,z') \in \S^2 $ the function $ \Delta_i\varphi $ is well-defined.

Our main results concerns the well-posedness of the rouleau models as well as the long-time asymptotics. As is typical in kinetic theory, a first step towards the behaviour of solutions is the study of moment equations. In the case of the rouleau models first, second and third order moments provide sufficient information:
\begin{align*}
	\M^1_j(t) = \int_{\S} z_j \, f(t,dz), \quad \M^2_{jk}(t) = \int_{\S} z_jz_k \, f(t,dz), \quad \M^2_{jk\ell}(t) = \int_{\S} z_jz_kz_\ell \, f(t,dz), \quad j, \, k,\, \ell \in \{1,2\}.
\end{align*}
In fact, the behaviour is lead by the asymptotics of the second and third order moments. Due to the quadratic nature of the coagulation kernel the matrix $\M^2(t)$ of second moments satisfy a matrix Riccati differential equation of the form
\begin{align}\label{eq:sec1:RiccatiEq}
	\dfrac{d}{dt}\M^2 = \M^2K\M^2 + \M^2A+A^\top\M^2+B, \quad K=\sum_{i =1}^3\alpha_iK_i,
\end{align}
for some matrices $ A=A(t),\,  B=B(t)\in \R^{2\times2} $. In fact, $ A,\, B $ depend on $ \M^1(t) $. As is well-known, the equation \eqref{eq:sec1:RiccatiEq} can lead to finite time blow-ups. This is in accordance with the occurrence of gelation for coagulation models with quadratic kernels, as the rouleau models studied here. Let us therefore define the blow-up time 
\begin{align}\label{eq:sec1:BlowUpTime}
	T_{\alpha,*}(f_0) := \sup \left\lbrace T\geq0 \, : \, \sup_{t\in[0,T]} |\M^2(t)|<\infty \right\rbrace \in [0,\infty].
\end{align}

Our first main result concerns the well-posedness of the rouleau models. 
\begin{thm}\label{thm:IntrodWellPos}
	Let $ \alpha \in \R_+^3 $, $\alpha \neq 0$, $ p\in \N $, $ p\geq3 $. Consider $ f_0\in \Meas_{p,+}(\S) $ and $ T_*=T_{\alpha,*}(f_0) $ given in \eqref{eq:sec1:BlowUpTime}. Then, there exists a unique weak solution $ f\in C^1([0,T_*); \Meas_{+}(\S))\cap L^\infty_{\loc}([0,T_*);\Meas_{p,+}(\S))) $ to \eqref{eq:sec1:ContWeakForm}.
\end{thm}
\begin{rem}
	The existence of a time dependent solution to equation \eqref{eq:sec1:ContWeakForm} is proven for  $ t < T_{\alpha,*}(f_0) $. We do not assert that this time is in fact the maximal time of existence and we will not discuss a possible extension beyond $ T_{\alpha,*} $. The above theorem just ensures that the solution exists as long as the second moments do not blow-up. Nevertheless, as is well-known in the theory of coagulation models, this time corresponds to the case of gelation, i.e.\ formation of coagulants of infinity size in finite time. Let us mention that for certain  coagulation models the extension beyond the gelation time was proved (see \cite{leyvraz1981singularities,norris2000cluster}). 
\end{rem}
Our next result concerns the long-time asymptotics of solutions at the gelation time. Due to the quadratic nature of the coagulation kernel gelation occurs for most initial conditions. We therefore focus in the situation of a finite time blow-up, that is $ T_{*}(f_0)<\infty $. We then identify the asymptotic behaviour of the solution close to $ T_{*}(f_0) $. However there are specific situations in which $T_* =\infty $, as in the following remark. 

\begin{rem} \label{rem:no gel}
    Assume that $ f_0 (z ) = \delta_{ y = 2} (z)  \mu (dx ) $.
    In this case blow-up does not occur for the coagulation equation \eqref{sec1:CoagulationEq} when $ \alpha =( 0 , 0, \alpha_3 ) $ with $\alpha_3 >0 $.
    Indeed, due to the form of the coagulation reaction $R_{3} $ we will have that $ f(t, z) = \delta_{ y = 2 }(z) \mu_t ( dx ) $ for every $ t\geq 0 $. 
    Then we deduce that $\mu_t $ satisfies the classical Smoluchowski's coagulation equation with constant kernel $ K = 4$, i.e. the following equation
	 \[
 		\partial_t \mu_t (y)  = 2 \int_0^y   \mu_t(t,  y' ) \mu_t (y-y' ) dy   - 4 \mu_t (y)   \int_0^\infty   \mu_t (y)  dy. 
 	\]   
 	It is known that gelation in this case does not take place. In fact, as we will see later on, see Proposition \ref{pro:BlowUp}, the above situation is the only instance for which gelation does not occur.
\end{rem}

To study rigorously the asymptotic behaviour of the solution close to $ T_{*}(f_0) $ a detailed analysis of the asymptotic behaviour of the third moment is necessary 
\begin{align*}
	\M^3_{jk\ell}(t) = \int_{\S} z_jz_kz_\ell \, f(t,dz), \quad j,\, k,\, \ell \in \{1,2\}.
\end{align*}
It is convenient to consider the whole set of third moments as a 3-tensor over $ \R^2 $, which is symmetric in all of its components, i.e. $ \M^3\in (\R^2)^{\otimes 3} $. We refer to the notation section (Section \ref{sec:noation}) for precise definitions. The equation for the third order moments is in fact linear in $ \M^3 $ and is lead by the asymptotics of the second order moments. The study of the third moments allows to determine the tail for large clusters of the distribution. 

The main theorem that we prove in this paper deals with the long-time behaviour of the solution to equation \eqref{sec1:CoagulationEq}. In particular we prove that the solution $f$,  written in the self-similar variables that we will obtain in Section \ref{subsec:scaling} with a heuristic argument, localizes along the direction $w_\theta$ that depends on the initial datum $f_0$. Moreover, we prove that the solution $f$ approaches a self-similar solution along the direction $\omega_\theta $. This self-similar solution belongs to the one-parameter family of continuous distributions (see \cite{menon2004approach}), of the form
\begin{align} \label{eq:sec1:McLeod}
	F_s( r)= \dfrac{1}{ \sqrt{2 \pi K_0} }  r^{-5/2} e^{-r/(2K_0)},
\end{align}
for $K_0>0$. 
\begin{thm}\label{thm:IntrodLocaliSelfSim}
	Let $ \alpha \in \R_+^3 $ be such that $\alpha \neq 0 $. Let  $ f_0\in \Meas_{4,+}(\S) $ and $ T_*=T_{*}(f_0) $ be given as in \eqref{eq:sec1:BlowUpTime}. Let $ f\in C([0,T_*); \Meas_{+,3}(\S))\cap L^\infty_{\loc}([0,T_*);\Meas_{4,+}(\S))) $ be the solution to \eqref{eq:sec1:ContWeakForm} as in Theorem \ref{thm:IntrodWellPos}. Furthermore, assume that $ T_* <\infty $ is finite and define
	\begin{align*}
		F(\tau,\eta ) = (T_*-t(\tau))^{-7}f\left(t(\tau),(T_*-t(\tau))^{-2}\eta \right), \quad t(\tau) = T_*\left( 1-e^{-\tau}\right).
	\end{align*}
    Then, we have the following statements. 
	\begin{enumerate}[(i)]
		\item (Asymptotics of 2nd and 3rd moments) There is $ \theta\in \R_{+}^2 $ and a constant $ C>0 $ such that the matrix $ \M^2(\tau) $ of second moments of $ F(\tau) $ satisfies for all $ \tau\geq0 $
		\begin{align} \label{localization second moment}
			\norm{\M^2(\tau) - \theta\otimes\theta} \leq Ce^{-\tau}.
		\end{align}
		Furthermore, we have for the set of third moments $ \M^3(\tau) $ of $ F(\tau) $
		\begin{align} \label{localization third moment}
			\norm{\M^3(\tau) - c_0 \theta\otimes\theta\otimes\theta}\leq Ce^{-\tau}
		\end{align}
        where $c_0 >0 $. 
	
		Finally, the fourth moments $ \M^4(\tau) $ of $ F(\tau) $ satisfy $ \sup_{ \tau \in [0,\infty )}\norm{\M^4(\tau)}<\infty $.
		
		\item (Localization) the measure $F$ localizes along the line $ \{\eta=\lambda \theta, \lambda\geq0\} $, where $ \theta $ is given in (i). More precisely, there is a constant $ C>0 $ such that for any $ \tau\geq0 $
		\begin{align*}
			\int_{\R_+^2} |\eta|^p \norm{\dfrac{\eta }{|\eta |}-\dfrac{\theta}{|\theta|}}^2\, F(\tau,d\eta ) \leq Ce^{-\tau}, \quad p\in \{2,3\}.
		\end{align*}

		\item (Self-similar asymptotics) the distribution $ F $ converges to a self-similar solution as $ \tau\to \infty $. More precisely, we have
		\begin{align*}
			\lim_{\tau \to \infty}\dfrac{|\eta|^2 F(\tau,d\eta)}{Z(\tau)} = |\eta|F_s(|\eta|)\delta \left( \dfrac{ \eta}{|\eta | } - \dfrac{\theta}{|\theta|} \right) \, d\eta
		\end{align*}
		in the sense of measures. Here, $ Z(\tau) = \int_{\R_+^2} |\eta |^2 F(\tau,d\eta) $ and $ F_s $ is given in \eqref{eq:sec1:McLeod} with the constant $ K_0=c_0|\theta| $ as in (i).
	\end{enumerate}
\end{thm}

\begin{rem}
	Let us give some remarks on the preceding theorem.
	\begin{enumerate}[(i)]
		\item The localization line is defined by the vector $ \theta $. As it turns out it satisfies $ \sum_i\alpha_i\theta^\top K_i\theta=1 $, see Proposition \ref{pro:SecondMomentsLocalization}. However, this does not identify $ \theta $ and solving the second order moment equations is necessary to do so. As a consequence $ \theta $ depends on the second moments of the initial datum. This is in contrast with the result in \cite{ferreira2024asymptotic} for non-gelling kernels where the direction of the localization depends on the initial mass distribution 
        \[ 
        \int_{\mathbb R_+^n} z f_0(dz). 
        \]
        The reason why for the equations studied in this paper the vector $\theta $ depends only on the second moment of the initial datum is that gelation occurs. 
		
		\item Observe that the measure 
		\begin{align*}
			|\eta|F_s(|\eta|)\delta \left( \dfrac{ \eta}{|\eta | } - \dfrac{\theta}{|\theta|} \right) \, d\eta
		\end{align*}
        in point (iii) of Theorem \ref{thm:IntrodLocaliSelfSim}
		is a probability measure. This is due to the change of variables $ \eta\mapsto (|\eta|,\eta/|\eta|)=(r,\omega) $ with Jacobian $ r $, i.e. $ d\eta = rdrd\omega $, and $ \int r^2F_s(r)dr=1 $.
		
		\item Observe that the self-similar asymptotics is formulated for the measure $ |\eta|^2 F(\tau,d\eta)/Z(\tau) $. This is the correct quantity for the behaviour close to the gelation times, since in self-similar variables the second moments converge $ \theta\otimes \theta $, i.e. do not blow up. On the other hand, the measures $ F(\tau,d\eta) $ as well as $ |\eta|F(\tau,d\eta) $ have infinite mass as $ \tau\to \infty $. This is due to the formation of a singularity for clusters of size $ \eta=0 $. Observe that this can also be seen from the profile $ F_s(r) $, which diverges for $ r\to 0 $ like $ r^{-5/2} $. Since via the self-similar change of variables we zoom in around the region of clusters whose mass tends to infinity as $t \to T_* $, the singularity in the self-similar profile corresponds to an excess mass of particles remaining of finite size at the time of gelation.
	\end{enumerate}
\end{rem}

\section{Heuristic arguments} \label{subsec:scaling}
Before going into the details of the proofs of the results stated in Section \ref{sec:main results} we provide some heuristic arguments that aim at motivating these results. We think that these arguments could help the reader to understand the main statements proved in this paper, but they do not substitute the rigorous proofs that can be found in the next sections. 

We start by presenting a heuristic argument for the localization phenomena and later for the choice of the self-similar change of variables that we use in this paper. 

Localization is a phenomenon that occurs when infinitely many collisions take place. This could happen in finite or infinite time depending on whether or not gelation occurs. In this paper, the latter is the case and localization takes place as $t $ approaches the gelation time $T_*$. 

To illustrate the localization phenomenon, let us consider  $ f_0(dz) = \delta_{(1,\beta_1)}+\delta_{(1,\beta_2)} $ for $0 < \beta_1 < \beta_2 $. Denoting $p_{\beta}:= \{  z  \in \R^2_+ : z_2/z_1=\beta \} $, this measure is supported on $ p_{\beta_1} \cup p_{\beta_2} $. We now look at the possible outcomes of coagulation events $ (z,z')\mapsto z+z'$, so we assume for the moment $\xi_i =0$.

If one coagulation event occurs, then the only new cluster is $ (2,\beta_1+\beta_2)\in p_{(\beta_1+\beta_2)/2} $. After another coagulation event also clusters of the form $ (3,2\beta_1+\beta_2)\in p_{(2\beta_1+\beta_3)/3} $ and $ (3,\beta_1+2\beta_2)\in p_{(\beta_1+2\beta_3)/3} $ can be produced. After a third coagulation event we have in addition also the clusters
\begin{align*}
	(4,3\beta_1+\beta_2) &\in p_{\frac{3\beta_1+\beta_2}{4}}, \quad (4,\beta_1+3\beta_2) \in p_{\frac{\beta_1+3\beta_2}{4}}, \quad (5,3\beta_1+2\beta_2) \in p_{\frac{3\beta_1+2\beta_2}{5}},
	\\
	(5,2\beta_1+3\beta_2) &\in p_{\frac{2\beta_1+3\beta_2}{5}}, \quad (6,3\beta_1+3\beta_2) \in p_{\frac{\beta_1+\beta_2}{2}}.
\end{align*}
See Figure \ref{fig4} for an illustration of the possible lines to which a particle can belong. We observe that as the size (or here equivalent $ z_1 $) of the cluster increases the cone of possible directions becomes narrower. In particular, when looking at large clusters (as is the case in self-similar variables) the distribution will localize. 

\begin{figure}
	\centering
	\begin{subfigure}{.5\textwidth}
		\centering
		\includegraphics[width=.49\linewidth]{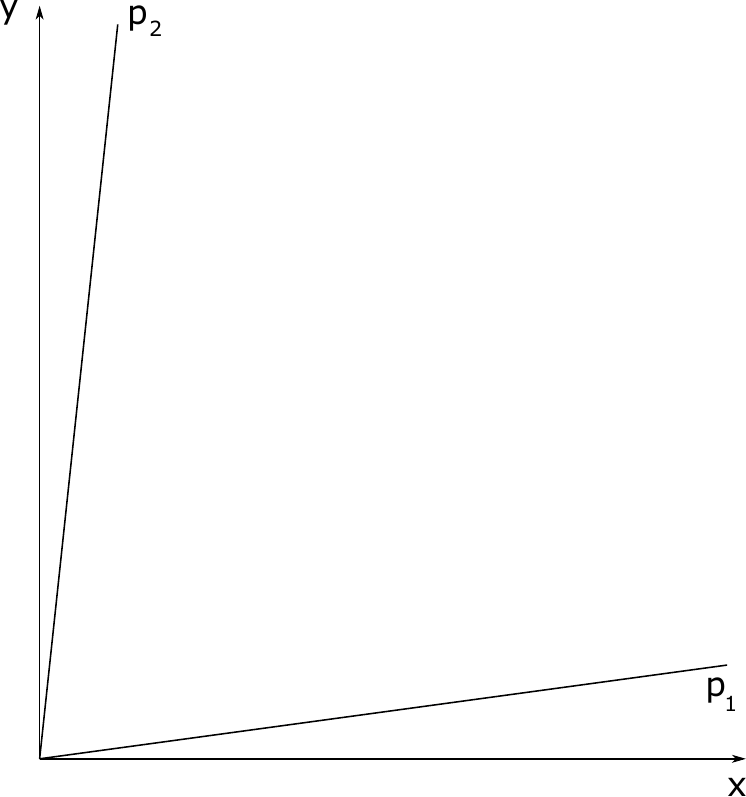}
	\end{subfigure}%
	\begin{subfigure}{.5\textwidth}
		\centering
		\includegraphics[width=.49\linewidth]{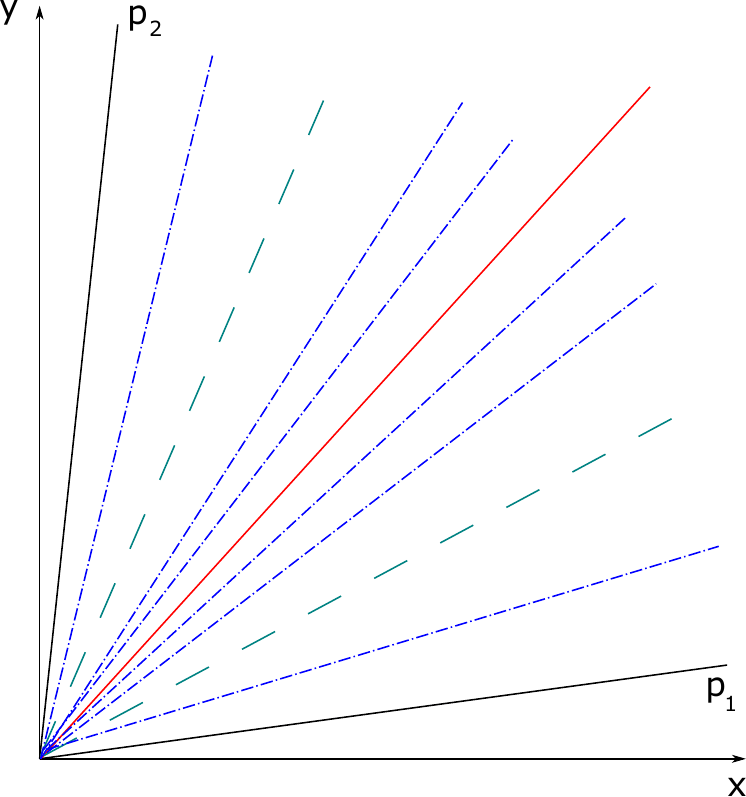}
	\end{subfigure}
	\caption{Left: Initial configuration. the particles are localized in the directions $p_1 = p_{\beta_1} $ and $p_2 = p_{\beta_2 } $. All the clusters will remain between the lines $p_1 $ and $p_2$ for all times. Right: possible directions after three coagulation events. In red (solid line) after one coagulation event. In green (dashed lines) after two coagulation events. In blue (dotted lines) after three coagulation events.}
	\label{fig4}
\end{figure}

The mechanism of this phenomenon is due to the coagulation rule $ (z,z')\mapsto z+z' $ which leads to the effect that the coagulant of two clusters becomes bigger and relative to this its direction become closer. As a result particles on the boundary of the cone of all clusters always move closer to the centre of this cone by a coagulation process. Thus, the bigger the particles the more their directions align. In particular, in self-similar variables (which centres around large clusters) the support of the distribution of clusters localizes.

Let us mention that this heuristics also applies to the coagulation rules $ \{R_i\}_{i=1}^3 $, i.e. $ (z,z')\mapsto z+z'+\xi_i $, with $ \xi_i $ as in \eqref{eq:sec1:DefXi}. Indeed, as the size of the clusters $ z,\, z' $ increases the effect of the term $ \xi_i $ becomes smaller. Furthermore, observe that the above heuristics assumes that all coagulation processes are possible. If certain coagulation events are not possible (as is the case when the coagulation kernel vanishes), localization might not occur.

We now motivate the self-similar change of variables introduced in \eqref{intro:selfsim}. 
To this end we assume that localization occurs and we study the natural self-similar scaling that this assumption suggests.

Let us assume that  $ x \sim s(t) $ as $ t \to T_* $ and that $s(t) \to \infty $ as $ t \to T_* $. 
A change of variable that keeps the coagulation equation invariant must satisfy 
\begin{equation} \label{mathc between coag and der} 
\partial_t [f] = [ \K (f,f) ] 
\end{equation}
where we are denoting with $ [f ] $ the dimension of $f $ and with $[ \K (f,f) ] $ the dimension of $ \K (f,f) $. 
Now notice that for large times, i.e. as $t \to T_*$, it holds that
\[
[\K (f,f) ] \sim [f]^2 [x]^4. 
\]
This follows by the homogeneity of the kernel and by the fact that $s(t) \to \infty $ as $t \to T_* $. Using \eqref{mathc between coag and der} deduce that
\begin{equation} \label{scaling of f}
	[f] \sim \frac{1}{ (t- T_*) s(t)^4 } \text{ as } t \to T_*. 
\end{equation}
The fact that gelation occurs means that we have an escape of mass towards infinity at time $T_* $. As for the classical coagulation equation, this allows us to infer the precise form of the scaling $s(t)$. The following argument is an adaptation of the dimensional analysis performed in \cite{delacour2022mathematical} in order to find the scaling for the self-similar solution to the coagulation equation with product kernel. 

We can define the flux of mass $J_R [f] $ at size $R $ in an analogous way as for the classical coagulation equation. The flux can be obtained multiplying equation \eqref{sec1:CoagulationEq} against the test function $\varphi (z ) = \ind_{ \{ |z| > R \}  } (z )   |z| $ and integrating. We deduce that 
\[
\dfrac{d}{dt} \int_{ \{ |z|> R \} } |z| f(t, dz) = J_R^1 [f] ( t)  + J_R^2 [f](t)
\]
where with $ \zeta_i=-\xi_{i, 1}-\xi_{i, 2}\in \{1,2\} $
\begin{align*}
	J^1_R [f](t) &= \sum_{i =1}^3 \alpha_i  \int_{\{|z|+|z'|\geq R+\zeta_i, \, |z|\leq R\}} z^\top K_i z'\, |z| f(t, dz) f(t, dz'),
	\\
	J^2_R [f](t) &= -\sum_{ i =1}^3 \zeta_i\alpha_i  \int_{\{|z|+|z'|\geq R+\zeta_i, \, |z|\leq R\}} z^\top K_i z'\, f(t, dz) f(t, dz').
\end{align*}
Assuming that gelation occurs at time $ t=T_* $ we necessarily have
\begin{align} \label{flux at infinity}
	0 < \lim_{ R \to \infty }  \left( J^1_R [f](T_*)+J^2_R [f](T_*) \right)  < \infty. 
\end{align}
Furthermore, since localization occurs for $t \to T_*$ we have $f (t, z ) \sim F(t, |z|) $ as $ t \to T_* $. Hence, in order to ensure \eqref{flux at infinity}, $F(t, |z|) $ has to decay slow enough, in particular polynomially, i.e. $F(t, |z|)  \sim |z|^{-\kappa} $ when $ |z|\to \infty $ for some $ \kappa>0 $. We now identify the value of $ \kappa $. To this end, we notice that 
\begin{align*}
	J^1_R [f](T_*) &\sim \sum_{i =1}^3 \alpha_i  \int_{\{|z|+|z'|\geq R+\zeta_i, \, |z|\leq R\}} z^\top K_i z'\, |z|^{1-\kappa}|z'|^{-\kappa} \, dzdz' 
	\\
	&\sim  \int_{\{|z|+|z'|\geq R, \, |z|\leq R\}}|z|^{2-\kappa}|z'|^{1-\kappa}\, dzdz' \sim \int_0^Rdr r^{3-\kappa}\int_{R-r}^\infty dr'(r')^{2-\kappa} \sim R^{7-2\kappa}
\end{align*}
as $ R\to \infty $. Similarly, we obtain $ J^2_R [f](T_*)\sim R^{6-2\kappa} $ as $ R\to \infty $. As a consequence 
\[ 
	J^1_R [f](t)+J^2_R [f](t) \sim R^{ 7 - 2 \kappa } \text{ as } R \to \infty.  
\]
But then \eqref{flux at infinity} requires $ \kappa = - 7/2 $. Therefore using \eqref{scaling of f} we deduce that 
\[ 
[f ] \sim \dfrac{1}{(t- T_*) s(t)^4 } \sim \dfrac{1}{ s(t)^{  7 /2 }}
\]
This implies  $s(t) \sim (t- T_*)^{-2 }$ and $[f ] \sim (t- T_*)^{ - 7 } $. which coincides with \eqref{intro:selfsim}. Notice that this scaling suggests that the second moment of $f $ blows up like $ (T_* - t )^{-1 } $ and the third moment like $ (T_* - t )^{-3 }$ as consistent with Theorem \ref{thm:IntrodLocaliSelfSim} (i).

Notice that applying the change of variables \eqref{intro:selfsim} 
in equation \eqref{sec1:CoagulationEq} yields
\begin{align}\label{eq:self_sim}
	\begin{split}
		\partial_{\tau}F (\tau, \eta  ) &= 7F(\tau, \eta )  + 2\eta \cdot\nabla F (\tau, \eta )  +  \frac{1}{2} \sum_{i=1}^3 \alpha_i \int_{ B_z^i }  \K_i( \eta ', \eta - \eta ' - \delta_i (\tau)   ) F(\tau ,  \eta ' ) F(\tau , \eta - \eta' - \delta_i (\tau )   ) d\eta ' \\
		& - F (\tau , \eta )  \int_{ \S } \sum_{i=1}^3 \alpha_i  K_i(\eta,\eta ') F(\tau ,\eta') d\eta'	
	\end{split}
\end{align}
where $\delta_i (\tau ) \rightarrow 0 $ as $ \tau \to \infty$. 

Let us now set $\delta_i(\tau)=0$ and assume that $ F $ is localized along a line. We then expect $F$ to by a self-similar rescaling of a \emph{one-dimensional} distribution $ f $ which satisfies a \emph{one-dimensional} coagulation equation with product kernel of the form
\begin{align*}
    \partial_t f(t,z)=\dfrac{1}{2}\int_0^z z z'  f(t,  y ) f(t,z -z' ) \, dz'  - f(t, z )  \int_0^\infty  z z' f(t, z' ) \, dz'. 
\end{align*}
As is know from \cite{menon2004approach} solutions to this equation have a self-similar behaviour close to the gelation time. More precisely, in self-similar variables solutions approach the one-parameter family of distributions $F_s$ given by \eqref{McLeod}. This argument suggests the long-time asymptotics formulated in Theorem \ref{thm:IntrodLocaliSelfSim} (iii).


\section{Well-posedness of coagulation models}\label{sec:WellPosed}
In this section we prove well-posedness of the coagulation models. To this end, we first give a definition of weak solutions to the coagulation equations.

\begin{defi}[Weak solutions] \label{def:weak solution}
	Let $ \alpha \in \R_+^3 $ and $ f_0\in \Meas_{1,+}(\S) $. We say that 
	\begin{align*}
		f\in C([0,T);\Meas_{+}(\S)))\cap L^\infty_{\loc}([0,T);\Meas_{1,+}(\S)))
	\end{align*}
	is a weak solution to \eqref{eq:sec1:ContWeakForm} if we have for any $ \varphi\in C_b(\S) $ and all $ t\in [0,T) $ 
	\begin{align*}
		\int_{\S}\varphi(z) \, f(t,dz) =\int_{\S}\varphi(z) \, f_0(dz) +  \dfrac{1}{2}  \sum_{i =1}^3\int_0^t\int_{\S}\int_{\S} \alpha_i\K_i(z,z') \, \Delta_i\varphi(z,z') \, f(s,dz)f(s,dz')\, ds.
	\end{align*}
\end{defi}
\begin{rem}
    Let us mention that the assumption $ f\in L^\infty_{\loc}([0,T);\Meas_{1,+}(\S))) $ ensures that the integral on the right hand side is well-defined.
\end{rem}

\begin{pro}\label{pro:MomentEquations}
	Let $ \alpha \in \R_+^3 $ and $ f\in C^1([0,T];\Meas_{+})\cap L^\infty([0,T];\Meas_{3,+}) $
    be a weak solution to \eqref{eq:sec1:ContWeakForm} for some $ T>0 $. Then, the following holds:
	\begin{enumerate}[(i)]
		\item the function $ t\mapsto \M^1(f(t)) $ belongs to $ C^1([0,T];\R_+^2) $, it satisfies the following system of ODEs 
          \begin{align*}
            \frac{d}{dt} \M^1=  \sum_{i =1}^3\frac{\alpha_i}{2}(\M^1)^\top K_i \M^1 \; \xi_i
        \end{align*}
        and $ \norm{\M^1(f(t))}\leq C\norm{\M^1(f(0))} $ for some $ C>0 $.
		
		\item the function $ t\mapsto \M^2(f(t)) $ belongs to $ C^1([0,T];\R_+^{2\times2}) $ and satisfies the Riccati equation
		\begin{align}\label{eq:sec2:SecMomEq}
			\dfrac{d}{dt}\M^2 = \sum_{i =1}^3\alpha_i(\M^2 + \xi_i \otimes \M^1 ) K_i (\M^2+ \M^1 \otimes \xi_i  )  - \sum_{i =1}^3\frac{\alpha_i}{2} (\xi_i \otimes \M^1 ) K_i  ( \xi_i \otimes \M^1 ).  
		\end{align}
	\end{enumerate}
\end{pro}
\begin{proof}
	We prove this result using a cutoff of the test functions $ z $ and $ z\otimes z $. To this end, define for some $ R>0 $ the smooth function $\Psi_R $ such that $ \Psi_R (v) =1 $ for $ v \leq R $ and $\Psi_R (v)=0$ for $v > R+1$.
	
	We now prove \textit{(i)}. We take $\varphi_R(z) = \Psi_R (|z|) z $ as a test function in \eqref{eq:sec1:ContWeakForm}. Define also $ \varphi(z)=z $ and notice that for $ i\in \{1,2,3\} $
	\begin{align*}
		|\Delta_i \varphi_ R (z, z') - \Delta_i \varphi(z, z') | & \leq 
		| ( \Psi_R (| z+ z' + \xi_i |) -1 ) (z+ z' + \xi_i )  -  \left( \Psi_R ( |z | ) -1  \right)  z -  \left( \Psi_R ( |z' | ) -1  \right)  z' | 
		\\
		& \leq \left( \ind_{ |z|  \geq \frac{ R- | \xi_i |}{2 } } (z,z') +  \ind_{ |z'|  \geq \frac{ R- | \xi_i |}{2 } }  (z,z')  \right)  |z+z'+\xi_i |. 
	\end{align*}
	Due to $f \in L^\infty_{\loc} ([0, T);  \M_{2, + })$ and the product structure of the kernels, we obtain
	\begin{align*}
		\int_{\S}\int_{\S} \K_i(z,z') \left| \Delta_i \varphi_ R (z, z') - \Delta_i\varphi(z,z') \right| f(t,dz)f(t,dz') \rightarrow 0 \text{ as } R \to \infty. 
	\end{align*}
	As a consequence the first moments $\M^1 $ satisfy the following system of equations 
    \begin{align*}
    	\M^1 ( t ) = \M^1 (0 ) + \sum_{i =1}^3\frac{\alpha_i}{2} \int_0^t \xi_i  \left( \M^1(s)\right) ^\top K_i \M^1 ( s) ds.
    \end{align*}
	In particular, we deduce that $\M^1 \in C^1 ([0, T ]; \R^2_+) $ and it satisfies the asserted ODE. Summing this vectorial equation yields with $ \xi_{i,1}+\xi_{i,2}\leq0 $ for all $ i=1,2,3 $
	\begin{align*}
		0\leq \M^1_1 (t)+\M^1_2 (t) \leq \M^1_1 (0)+\M^1_2 (0).
	\end{align*}

	We now prove \textit{(ii)}. We now choose $\varphi(z) = z \otimes z $ and $ \varphi_R(z) = \Psi_R (|z|) (z \otimes z) $. We again have
	\begin{align*}
		| \Delta_i  \varphi_R (z,z' ) -  \Delta_i  \varphi (z,z' )  | & \leq | (\Psi_R (|z+z' + \xi_i |) -1 ) ((z+z'+\xi_i ) \otimes  (z+z'+\xi_i )) \\
		&\quad + (\Psi_R (|z|) -1 ) (z \otimes z ) + (\Psi_R (|z'|) -1 ) (z' \otimes z' ) |  \\
		& \leq \ind_{| z+ z' + \xi_i | \in (R, R+1 )} (z,z')((z+z'+\xi_i ) \otimes  (z+z'+\xi_i )) \\
		&\quad +\ind_{| z | \in (R, R+1 )} (z,z') (z \otimes z )  + \ind_{| z' | \in (R, R+1 )} (z,z') (z' \otimes z' ). 
	\end{align*}
	Since $ f \in L^\infty_{\loc}([0, T);\M_{3, + }) $ we can pass to the limit $ R\to \infty $ in the weak formulation. Observe that 
	\begin{align*}
	 	\Delta_i \varphi (z, z') &= (z+ z' + \xi_i ) \otimes  (z+ z' + \xi_i ) - z \otimes z - z' \otimes z' 
	 	\\
	 	& = ( z+ \xi_i ) \otimes (z' + \xi_i )  +( z'+ \xi_i ) \otimes (z + \xi_i )   - \xi_i \otimes \xi_i.
	\end{align*}
	This yields for $ i=1,2,3 $
	\begin{align*}
		&\int_{\S}\int_{\S} \K_i(z,z') \Delta_i\varphi(z,z') f(t, dz) f(t, dz')
		\\
		&=\int_{\S}\int_{\S} ( z^\top K_i z' )  [( z+ \xi_i ) \otimes (z' + \xi_i ) ] f(t, dz) f(t, dz') -  \frac{1}{2 }   \int_{\S}\int_{\S} ( z^\top K_i z' )  
		(\xi_i \otimes \xi_i )  f(t, dz) f(t, dz') \\ 
		&=  \left(\M^2(t) + \xi_i \otimes \M^1(t) \right) K_i \left(\M^2(t) + \xi_i \otimes \M^1(t) \right) - \frac{1}{2 } \left( \xi_i \otimes \M^1 (t) \right) K_i  \left( \xi_i \otimes \M^1 (t)  \right). 
	\end{align*}
	In particular, we obtain the equation
	\begin{align*}
		\M^2 (t) &= \M^2 (0) + \sum_{i=1}^3\alpha_i \int_0^t \left[ \left(\M^2(s) + \xi_i \otimes \M^1(s) \right) K_i \left(\M^2(s) + \xi_i \otimes \M^1(s) \right) \right.
		\\
		&\qquad \left.  - \frac{1}{2 } \left( \xi_i \otimes \M^1 (s) \right) K_i  \left( \xi_i \otimes \M^1 (s)  \right) \right]  ds.
	\end{align*}
	As a consequence $\M^2 \in C^{1} ((0, T ] ; \R_+^{2 \times 2 }) $ and $\M^2$ satisfies the ODE \eqref{eq:sec2:SecMomEq}.
\end{proof}

\begin{defi}\label{def:BlowUpTime}
	Let $ \alpha \in \R_+^3 $. Let $ \M^1, \, \M^2 $ be the solution to the first and second order moment equations in Proposition \ref{pro:MomentEquations}. We define the blow-up time
	\begin{align*}
		T_{\alpha, *}(f_0) := \sup \left\lbrace T\geq0 \, : \, \sup_{t\in[0,T]} |\M^2(t)|<\infty \right\rbrace \in (0,\infty].
	\end{align*}
\end{defi}

\begin{thm}\label{thm:WellPosed}
	Consider $ \alpha \in \R_+^3 $ and $ p\in \N $, $ p\geq 3 $. Let $ f_0\in \Meas_{p,+}(\S) $ and $ T_*=T_{\alpha,*}(f_0) $ be given as in Definition \ref{def:BlowUpTime}. Then, there is a unique weak solution $ f\in C^1([0,T_*); \Meas_{+})\cap L^\infty_{\loc}([0,T_*); \Meas_{p,+}) $ to \eqref{eq:sec1:ContWeakForm}. Furthermore, the solution satisfies the following properties:
	\begin{enumerate}[(i)]
		\item for every $t \in [0, T_*)$ it holds that $ \norm{\M^1(f(t))}\leq C\norm{\M^1(f(0))} $ for some constant $ C>0 $.
		
		\item the function $ t\mapsto \M^2(f(t)) $ belongs to $ C^1([0,T];\R^{2\times2}_{\sym}) $ and satisfies the Riccati equation \eqref{eq:sec2:SecMomEq}.
		
		\item we have for all $ t < T_* $
        \begin{align*}
			\norm[p]{f(t)} \leq C_p\exp\left( C_p\int_0^t\left( 1+\norm{\M^2(s)}\right) \, ds \right),
		\end{align*}
        for a time-independent constant $ C_p>0 $.
	\end{enumerate}
\end{thm}
We will prove this result in several steps:
\begin{enumerate}[(i)]
	\item At first we show that weak solutions in $ C([0,T); \Meas_{+})\cap L^\infty_{\loc}([0,T); \Meas_{3,+}) $ are unique. The following steps then concerns the construction of the weak solution.
	
	\item We prove well-posedness of an appropriately truncated equation for which the coagulation operator is bounded. 
	
	\item We then derive uniform bounds on the $ p $-th moments independent of the truncation parameter. These bounds are a priori only available on some small, non-empty time interval $ [0,T] $. 
	
	\item Based on (ii) we extract a weakly converging subsequence, yielding a weak solution on the time interval $ [0,T] $ by passing to the limit in the weak formulation. Furthermore, we show that the so constructed solution satisfies (i)-(iv) in Theorem \ref{thm:WellPosed}.
	
	\item Finally, we show that the solution can be extended to the interval $ [0,T_*) $. Here, we make use of the fact that the $ p $-th moment remain bounded as long as the second moment does not blow-up.
\end{enumerate}
As a preparation for the proof of Theorem \ref{thm:WellPosed} we provide the following lemma.

\begin{lem}\label{lem:Uniquenss}
	Let $ \alpha \in \R_{+}^3 $, $ T>0 $, $ f_0\in \Meas_{3,+}(\S) $. Any two weak solutions in the class $ C([0,T); \Meas_{+})\cap L^\infty_{\loc}([0,T); \Meas_{3,+}) $  to \eqref{eq:sec1:ContWeakForm} with initial condition $ f_0 $ coincide.
\end{lem}
\begin{proof}
	We make use of the desingularized Laplace transform. In fact, we show that the Laplace transform $ \hat{f} $ satisfies a quasilinear first order PDE, which has a unique classical solution. 
	
	To this end, we use in \eqref{eq:sec1:ContWeakForm} the test function $ \varphi(z)=z(1-e^{-z\cdot \zeta}) $ for $ \zeta\in \R_{+}^2 $. Be the regularity assumption and a similar argumentation as in Proposition \ref{pro:MomentEquations} we obtain $ \hat{f}\in C^1_tC^0_\zeta\cap C^0_tC^1_\zeta $. Furthermore, we obtain for all $ t\in [0,T) $, $ \zeta\in \R_+^2 $
	\begin{align*}
		\partial_t \hat{f}(t,\zeta) &= \sum_{i =1}^3\dfrac{\alpha_i}{2}\int_{\S}\int_{\S} \K_i(z,z')\left[ \left( z+z'+\xi_i \right)\left( 1-e^{-(z+z'+\xi_i)\cdot \zeta} \right) - z(1-e^{-z\cdot\zeta})- z'(1-e^{-z'\cdot\zeta}) \right] \, ff'
		\\
		&= \sum_{i =1}^3\alpha_i\int_{\S}\int_{\S} \K_i(z,z') \, ze^{-z\cdot \zeta}\left( 1-e^{-(z'+\xi_i)\cdot\zeta} \right)  \, ff' 
		\\
		&\quad +\sum_{i =1}^3\alpha_i \dfrac{1}{2}\int_{\S}\int_{\S} \K_i(z,z') \, \xi_i \left( 1-e^{-(z+z'+\xi_i)\cdot\zeta} \right)  \, ff'.
	\end{align*}
	Using the form of the kernel $ \K_i $ we then obtain 
	\begin{align*}
		\partial_t \hat{f}(t,\zeta) &= D_\zeta\hat{f}\left[\sum_{i =1}^3\alpha_i K_i \, \left( \hat{f} + (1-e^{-\xi_i\cdot \zeta})(\M^1-\hat{f}) \right) \right]
		\\
		&\quad + \sum_{i =1}^3\alpha_i\dfrac{1}{2}\left( (\M^1)^\top K_i\M^1 - (\M^1-\hat{f})^\top \, K_i (\M^1-\hat{f}) e^{-\xi_i\cdot \zeta} \right) \xi_i.
	\end{align*}
	This equation is of the from
	\begin{align*}
		\partial_t\hat{f} = \left( b(t,\zeta,\hat{f})\cdot \nabla_{\zeta} \right)  \hat{f} + c(t,\zeta,\hat{f})
	\end{align*}
	for vector-valued functions $ b,\, c\in C^1 $. This equation can be solved using the method of characteristics. In particular, any classical solution provides a solution to the characteristic system: for all $ \zeta\in \R_+^2 $
	\begin{align*}
		\begin{cases}
			Z'(t,\zeta) = b(t,Z(t,\zeta),F(t,\zeta)),
			\\
			F'(t,\zeta) = c(t,Z(t,\zeta),F(t,\zeta))
		\end{cases}, \quad  Z(0,\zeta)=\zeta, \quad F(0,\zeta)=\hat{f}_0(\zeta).
	\end{align*}
	Since solutions to this system are unique, we conclude that the Laplace transforms to any two weak solutions to \eqref{eq:sec1:ContWeakForm} coincide for all $ t\in [0,T) $, $ \zeta\in \R_+^2 $. Hence, the weak solutions coincide as well.
\end{proof}

As mentioned above, in order to construct a weak solution we introduce an appropriate truncated model. To this end, let $ i\in \{1,2,3\} $ and $ R>0 $. The \textit{truncated kernels} $\K^{(R)} _i: \S\times \S\rightarrow \R_+$ are symmetric continuously differentiable bounded functions such that 
\begin{equation} \label{sec2:truncated kernels}
     \begin{cases} & \K_i^{(R)} (z,z')= \K_i (z,z'), \quad (z,z' ) \in ([0, R]^2 \times [0, R ]^2) \cap \S^2  \\
      &  \K_i^{(R)} (z,z')= 0    \quad (z,z' )\in \S^2 \setminus ( [0, 2R]^2 \times [0, 2R ]^2 )  \\
      &   0 \leq  \left(   \K_i (z,z') -  \K_i^{(R)} (z,z') \right)   \leq  e^{- R }  \quad (z,z' ) \in  ( \S^2 \cap ([0, 2R]^2 \times [0, 2R ]^2) ) \setminus ([0, R]^2 \times [0, R ]^2).          \end{cases}
\end{equation}
The weak form to the truncated equation with these kernels is given by
\begin{align}\label{eq:sec2:TruncEquation}
	\int_{\S}\varphi(z)  f_R  (t,dz) = \int_{\S}\varphi(z)  f_0(dz) +  \sum_{i=1}^3\dfrac{\alpha_i}{2} \int_0^t \int_{\S}\int_{\S} \K_i^{(R)}(z,z') \, \Delta_i\varphi(z,z') \, f_R (s,dz)f_R (s,dz') \, ds,
\end{align}
which needs to be satisfied for every $ \varphi\in C_c(\S) $.

\begin{defi}[Truncated coagulation equation]\label{def:sec2:TruncEquation}
	Let $ \alpha\in\R_+^3 $ and let $R>0 $. Assume that $\K^{(R)}_i $ for $ i=1,2,3 $ are truncated kernels satisfying \eqref{sec2:truncated kernels}. A function $ f_R \in  C^1 ([0,\infty); \Meas_{+, b } (\S) )$ is a weak solution to the truncated model if \eqref{eq:sec2:TruncEquation} is satisfied for all $ \varphi\in C_c(\S) $ and $ t\geq 0 $.
\end{defi}
\begin{rem}
	Note that by definition $ \K_i^{(R)} $ is continuous and compactly supported. This has the advantage that also unbounded test functions $\varphi \in C(\S) $ can be used in the weak form, since
	\begin{align*}
		(z, z') \mapsto \K_i^{(R)}(z,z') \, \Delta_i\varphi(z,z')
	\end{align*}
	is continuous with compact support. In particular, this will be useful in order to study the moments of $ f_R $ and no truncation argument as in the proof Proposition \ref{pro:MomentEquations} is needed.
\end{rem}
\begin{lem}\label{lem:WellPosedTruncModel}
	Let $ \alpha\in \R_{+}^3 $, $ p\geq1 $ and $R>0$. Assume that $ f_0\in \Meas_{p,+}(\S) $. Then there is a unique solution $ f_R \in C([0,\infty); \Meas_{+ }) \cap L^\infty_{\loc}([0,\infty);\Meas_{p,+}) $ to \eqref{eq:sec2:TruncEquation} in the sense of Definition \ref{def:sec2:TruncEquation} with initial condition $ f_0 $.
\end{lem}
\begin{proof}
	We rewrite equation \eqref{eq:sec2:TruncEquation} in fixed point form. To this end we introduce the following notation. We define 
	\begin{align*}
		\mathcal{X}_T:= \left\{ g \in C([0,\infty); \Meas_{1,+}(\S)) : \sup_{t\in[0,T]}\norm[p]{g(t)} \leq 1 +\norm[p]{f_0} \right\}.
	\end{align*}
	We equip this set with the metric
	\begin{align*}
		d_T(f,g) := \sup_{t\in[0,T]}\sup_{\norm[\Lip]{\varphi}\leq1}\left| \dualbra{f(t)-g(t)}{\varphi} \right|,
	\end{align*}
	Then, the metric space $ (\mathcal{X}_T,d_T) $ is complete. Note that this metric corresponds to the Kantorovich-Rubinstein distance on $ \Meas_{1,+}(\S) $.
 	
	Given $ g \in \mathcal{X}_T $ we define the bounded and continuous function
	\begin{align*}
		\mu[ g ] (t, z):= \sum_{i=1}^3\alpha_i\int_{\S}  \K^{(R)}_i (z,z')  g(t, dz' ), \quad t \geq 0, \quad z \in \S_i.
	\end{align*}
	Furthermore, we denote by $\mathcal T_0(t) [g] $ the measure on $ \S $ defined by duality for all $ \varphi \in C_b(\S) $,  $ t \geq 0 $
	\begin{align*}
		\langle \varphi, \mathcal T^0 (t)[g] \rangle :=\int_{\S} \varphi (z) e^{- \int_0^t \mu[g](s,z) \, ds} f_0(dz), \quad \varphi \in C_b(\S), \quad t \geq 0.
	\end{align*} 
	In addition, we write $\mathcal T_1 (t)  [g] $ for the measure on $ \S $ defined by duality
	\begin{align*}
		\langle \varphi, \mathcal T^1(t)[g]  \rangle := \sum_{i =1}^3\frac{\alpha_i}{2} \int_0^t \int_{\S \times \S}  \K^{(R)}_i (z,z')  \varphi(z+z' + \xi_i ) e^{- \int_s^t \mu[g](r,z) dr} g(s,dz') g(s, dz) ds.
	\end{align*}
	Let $ \mathcal{T} := \mathcal{T}^0 +\mathcal{T}^1 $. 
	We now prove that the fixed point equation
	\begin{align}\label{eq:sec2.fixed point}
		g (t) = \mathcal{T} (t) [g] ,\ t \in [0, T)
	\end{align}
	has a unique solution $ f_R \in \mathcal{X}_T $ for a small enough $T>0$. Note that this yields a weak solution to \eqref{eq:sec2:TruncEquation}. 
	
	\textit{Step 1.} We first show that $ \mathcal{T} :   \mathcal{X}_T \rightarrow \mathcal{X}_T $ is a self-mapping for small $ T>0 $. From the definition one can see that for any $ g\in \mathcal{X}_T $ the function $ \mathcal{T}[g] $ is in $ C([0,\infty); \Meas_{+}(\S)) $. Furthermore, have the estimates
	\begin{align*}
		\norm[p]{\mathcal{T}^0[g](t)}\leq \norm[p]{f_0}, \quad \norm[p]{\mathcal{T}^1[g](t)}\leq TC_R\sup_{t\in[0,T]}\left( 1+\norm{g(t)}\right) ^2 \leq TC_R \left( 1+\norm[p]{f_0}\right)^2.
	\end{align*}
	In particular, $ \mathcal{T}[g]\in \mathcal{X}_T $ for sufficiently small $ T>0 $.
	
	\textit{Step 2.} We now show that $ \mathcal{T} $ is a contraction. To this end, let $ g_1,\, g_2\in \mathcal{X}_i $. We observe that for every $\varphi \in \Lip(\S) $ with $ \norm[\Lip]{\varphi} \leq 1 $
	\begin{align*}
		\norm[\infty]{\mu[g_1](t)-\mu[g_2](t)} &\leq C_R \,  d_T(g_1(t),g_2(t)),
		\\
		\left|\dualbra{\mathcal{T}^0[g_1](t)-\mathcal{T}^0[g_2](t)}{\varphi}\right| &\leq T C_R\left( 1+\norm[1]{f_0} \right)  \,  d_T(g_1(t),g_2(t)),
		\\
		\left|\dualbra{\mathcal{T}^0[g_1](t)-\mathcal{T}^0[g_2](t)}{\varphi}\right| &\leq T C_R\left( 1+\norm{f_0} \right)^2  \,  d_T(g_1(t),g_2(t)).
	\end{align*}
	Thus, $ \mathcal{T} $ is a contraction for sufficiently small $ T>0 $. Using Banach fixed point theorem we deduce that there is a unique fixed point.
	
	\textit{Step 3.} Let $ f $ be the unique weak solution on some time interval $ [0,T] $. We now show that there is no finite time blow-up. Observe that from the weak form \eqref{eq:sec2:TruncEquation} we obtain $ \norm{f(t)}\leq \norm{f_0} $. Thus, Step 1 yields the estimate
	\begin{align*}
		\sup_{t\in[0,T]}\norm[p]{f(t)}\leq \norm[p]{f_0} + TC_R \sup_{t\in[0,T]}\left( 1+ \norm{f(t)}\right)^2 \leq \norm[p]{f_0} + TC_R\left( 1+\norm{f_0} \right)^2. 
	\end{align*}
	Consequently, the solution can be extended to all of $ [0,\infty) $.
\end{proof}

In the following we prove uniform bounds on the solution to the truncated model.
\begin{lem}\label{lem:TruncModelUnifBound}
	Let $ \alpha\in \R_{+}^3 $, $ p\in \N $ and $ p\geq3 $. Consider $ f_0\in \Meas_{p,+}(\S) $ and $ f_R\in C([0,\infty); \Meas_{+})\cap L^\infty_{\loc}([0,\infty); \Meas_{p,+}) $ be the unique solution to \eqref{eq:sec2:TruncEquation} with initial condition $ f_0 $.
	
	Then, there is $ T_0>0 $ and $ C>0 $ depending only on $ \norm[2]{f_0} $ such that
	\begin{align*}
		\sup_{t\in[0,T_0]}\norm{\M^2(f_R(t))} \leq C.
	\end{align*}

	 Furthermore, there is a constant $ C>0 $ depending only on $ p $ and $ \norm[p]{f_0} $ such that
    \begin{align*}
		\norm[p]{f_R(t)}\leq  C\exp\left( C \int_0^t\left( 1+\norm{\M^2(f_R(s))}\right) \, ds \right).
    \end{align*}
\end{lem}

\begin{rem}
    Notice that the above inequality implies that the $ p $-th moment remains bounded as long as the second moment is finite. In particular, in the limit $ R\to \infty $ finite time blow-up is dictated by the blow-up of the second moments. 
\end{rem}
\begin{proof}[Proof of Lemma \ref{lem:TruncModelUnifBound}]
	Let us recall that we can use unbounded test functions due to the compact support of the kernel in the truncated model. Using this we first show that $ \norm[1]{f_R(t)}\leq \norm[1]{f_0} $ for all $ t\geq0 $. Indeed, we have with \eqref{eq:sec2:TruncEquation}
	\begin{align*}
		\int_{\S}(1+|z|) \, f_R(t,dz) &= \int_{\S}(1+|z|) \, f_0(dz)  
		\\
		&\quad+  \sum_{i=1}^3\dfrac{\alpha_i}{2}  \int_0^t\int_{\S} \int_{\S} \K_i^{(R)}(z,z')(\xi_{i,1}+\xi_{i,2}-1)  \, f_R (s,dz)f_R(s,dz')\, ds \leq \norm[1]{f_0}
	\end{align*}
	by the definition of $ \xi_i $ in \eqref{eq:sec1:DefXi}. For the second moments we have with $ \K_i^{(R)}(z,z')\leq C|z||z'| $
	\begin{align*}
		\norm{\M^2(f_R(t))} &\leq \norm{\M^2(0)} +C\sum_i\int_0^t\K_i^{(R)}(z,z')(|z|+1)(|z'|+1)\, f_R (s,dz)f_R(s,dz')\, ds
		\\
		&\leq \norm{\M^2(0)} + C\int_0^t\left( \norm{\M^2(f_R(s))}^2+1 \right) \, ds
	\end{align*}
	Hence, by a Gronwall argument there is $ T_0>0 $ and $ C>0 $ depending on $ \norm[2]{f_0} $ such that
	\begin{align*}
		\sup_{t\in[0,T_0]}\norm{\M^2(f_R(t))} \leq C.
	\end{align*}
	Concerning the $ p $-th moment we have 
	\begin{align*}
		\int_{\S}|z|^p \, f_R(t,dz) = \int_{\S}|z|^p \, f_0(dz)+ \sum_i\dfrac{\alpha_i}{2}  \int_0^t\int_{\S_i} \int_{\S_i} \K_i^{(R)}(z,z') \, \Delta_{i}\left[ |\cdot|^p \right]  \, f_R (s,dz) f_R(s,dz') \, ds.
	\end{align*}
	Using $ p\in \N $ one can show the inequality for $ i=1,2,3 $
	\begin{align*}
		\Delta_i\left[ |\cdot|^p \right](z,z') \leq C_p |z|^{p-1}\left( |z'|+1 \right) + C_p |z'|^{p-1}\left( |z|+1 \right) +C_p.
	\end{align*}
	For mixed terms $ |z'|^k|z|^j $ with $ 1\leq j,\, k\leq p-1 $ one can split into the cases $ |z|\leq|z'| $ and $ |z'|\leq|z| $. This yields with the bound $ \K_i^{(R)}(z,z')\leq C|z||z'| $
	\begin{align*}
		\norm[p]{f_R(t)} &\leq \norm[p]{f_0} + C_p\int_0^t \left[ \norm[p]{f_R(s)}\left( \norm{\M^2(f_R(s))}+\norm[1]{f_R(s)}  \right) + \norm[1]{f_R(s)}^2  \right] \, ds
		\\
		&\leq \norm[p]{f_0} + C\int_0^t \left[ \norm[p]{f_R(s)}\left( \norm{\M^2(f_R(s))}+1  \right) + 1  \right] \, ds.
	\end{align*}
	Gronwall's lemma then concludes the proof. 
\end{proof}

We can now give the proof to Theorem \ref{thm:WellPosed}.
\begin{proof}[Proof of Theorem \ref{thm:WellPosed}]
	We divide the proof into two steps.
	
	\textit{Step 1. Limit as $R \to \infty $.} Let $ (f_R)_R $ be the unique weak solution to the truncated model for $ R\geq0 $. We first show that there is a converging subsequence $ R_n\to\infty $. To this end, we first observe that by Lemma \ref{lem:TruncModelUnifBound} the set $ (f_R(t))_R $ is uniformly bounded in $ \M_{p,+} $ for all $ t\in [0,T_0] $. Here, $ T_0 $ is given in Lemma \ref{lem:TruncModelUnifBound}. In addition, $ (f_R)_R $ satisfies the following equicontinuity property with respect to time: for any $ \varphi\in C_b(\S) $ and $ t,\, s\in [0,T_0] $
	\begin{align*}
		\left| \dualbra{f_R(t)-f_R(s)}{\varphi}\right| &\leq \sum_i\alpha_i\int_s^t \K_i^{(R)}(z, z') |\Delta_i\varphi(z, z')| \,  f_R(r,dz) f_R(r,dz') \, dr 
		\\
		&\leq C\norm[\infty]{\varphi} \sup_{t\in[0,T_0]}\norm[1]{f(t)}^2|t-s|
		\\
		&\leq C\norm[\infty]{\varphi} \sup_{t\in[0,T_0]}\norm[1]{f_0}^2|t-s|.
	\end{align*}
	All in all, this implies that there is a subsequence $ R_0\to \infty $ such that $ f_{R_n} $ converges weakly to $ f\in C([0,T_0]; \Meas_{+}(\S))\cap L^\infty_{\loc}([0,T_0]; \Meas_{p,+}(\S)) $. Furthermore, by the uniform moment bounds on the time interval $ [0,T_0] $ we can pass to the limit $ R_n\to \infty $ in the weak form \eqref{eq:sec2:TruncEquation}. Consequently, $ f $ provides a weak solution to \eqref{eq:sec1:ContWeakForm}. In addition, by passing into the limit in the estimate in Lemma \ref{lem:TruncModelUnifBound} we obtain
	\begin{align}\label{eq:sec2:ProofWellPosBound}
		\norm[p]{f(t)}\leq  C\exp\left( C \int_0^t\left( 1+\norm{\M^2(s)}\right) \, ds \right).
	\end{align}
    
	\textit{Step 2. Extension.} Recall that by Lemma \ref{lem:Uniquenss} the solution constructed in Step 1 is unique. Furthermore, due to $ p\geq 3 $ and Proposition \ref{pro:MomentEquations}, the equations for the first and second moments are satisfied on $ [0,T_0] $. In particular, this implies the following a priori bounds: the first moment $ \M^1 $ is bounded for all times, while the second moment is finite on the time interval $ [0,T_*) $ by definition of $ T_* $ in \eqref{def:BlowUpTime}. From \eqref{eq:sec2:ProofWellPosBound} we infer that also the $ p $-th moment is finite on the time interval $ [0,T_*) $. This allows to extend the solution to the whole time interval $ [0,T_*) $.
	
	Statements (i)-(iii) in Theorem \ref{thm:WellPosed} follow from the preceding arguments. This concludes the proof.
\end{proof}


\section{Localization and self-similar asymptotics of coagulation models}\label{sec:LocSelfSimAsymp}
In this section we study the long-time behaviour of solutions to the rouleau coagulation equation. A first step is the study of the moments. More precisely, we study the asymptotic behaviour of the second and third moments close to the gelation time. Furthermore, we give a bound on the fourth moment needed for the subsequent arguments. This asymptotics in fact allows to identify the appropriate self-similar scaling to justify rigorously the scaling identified in Section \ref{subsec:scaling} via heuristic arguments. With respect to these self-similar variables we then show that the distribution function localizes along a line and approaches a self-similar profile.

\subsection{Study of moments}\label{subsec:StudyMoments}
In this section we study the first, second, third and fourth order moments for given $ \alpha\in \R_+^3 $. 

\subsubsection{First order moments} 
By Proposition \ref{pro:MomentEquations} the first moments satisfy the system
\begin{align}\label{eq:MomentEq1st}
	\dfrac{d}{dt} \M^1 =  \sum_{i =1}^3\frac{\alpha_i}{2}(\M^1)^\top K_i \M^1 \; \xi_i.
\end{align}
Since $ \xi_{i,1} + \xi_{i,2} \leq 0 $ for all $ i\in \{1,2,3\} $ we obtain the following uniform bound.
\begin{pro}\label{pro:1stOderMomentEquations}
		Let $ \alpha\in \R_{+}^3 $, $ f_0\in \M_{3, + }(\S) $ and $ f\in C^1([0,T_*); \Meas_{+})\cap L^\infty_{\loc}([0,T_*); \Meas_{3,+}) $ be the unique weak solution to \eqref{eq:sec1:ContWeakForm} on the maximal time interval $ [0,T_*) $. The solution $ \M^1 $ to \eqref{eq:MomentEq1st} satisfies
		\begin{align*}
			\norm{\M^1(t)} \leq C\norm{\M^1(0)},
		\end{align*}
		for all $ t\in [0,T_*) $ and some constant $ C>0 $.
\end{pro}

\subsubsection{Second order moments}
We now turn to the study of the second order moment equations, which are given by (see Proposition \ref{pro:MomentEquations})
\begin{align}\label{eq:MomentEq2nd}
	\dfrac{d}{dt}\M^2 = \sum_{i =1}^3\alpha_i(\M^2 + \xi_i \otimes \M^1 ) K_i (\M^2+ \M^1 \otimes \xi_i  )  - \sum_{i =1}^3\dfrac{\alpha_i}{2 } ( \xi_i \otimes \M^1 ) K_i  ( \xi_i \otimes \M^1 ).
\end{align}
This system is a matrix-valued Riccati equation. As is well-known this system can produce blow-ups in finite time. We now characterize the initial conditions $ f_0 $ that produce blow-ups.

\begin{pro}\label{pro:BlowUp}
	Let $\alpha\in \R_+^3 $, $ f_0\in \M_{3, + }(\S) $, $ f_0\neq0 $, and $ f\in C^1([0,T_*); \Meas_{+})\cap L^\infty_{\loc}([0,T_*); \Meas_{3,+}) $ be the unique weak solution to \eqref{eq:sec1:ContWeakForm} on the maximal time interval $ [0,T_*) $. Then, $ T_*<\infty $ if and only if one of the following holds
	\begin{enumerate}[(i)]
		\item $ \alpha_1>0 $ or $ \alpha_2>0 $,
		
		\item $ \alpha_1=\alpha_2=0 $, $ \alpha_3>0 $ and
		\begin{align}\label{eq:sec3:BlowUpCond}
			\int_{ \S } y(y-2) f_0(dz) >0.
		\end{align}
	\end{enumerate}
\end{pro} 
\begin{rem}
	Observe that \eqref{eq:sec3:BlowUpCond} fails if and only if
	\begin{align*}
		f_0(dz) = \mu(dx)\otimes \delta_{2}(dy)
	\end{align*} 
	for some measure $ \mu\in \Meas_+ $.
\end{rem}
\begin{proof}[Proof of Proposition \ref{pro:BlowUp}]
	We consider three cases.
	
	\textit{Case 1. $ \alpha_1>0 $.} We assume that $ \alpha_1>0 $. Consider the test function $ \varphi(z)=x^2 $. Let us first observe that
	\begin{align*}
		\Delta_i\varphi(z,z) = 2xx'+(x+x')\xi_{i,1} + \xi_{i,1}^2, \quad i=1,2,3.
	\end{align*}
	Considering all three cases for $ \xi_{i,1} $, see \eqref{eq:sec1:DefXi}, one can see that the minimum of $ \Delta_i\varphi(z,z) $ is attained for $ x=x'=2 $ and is positive. In particular, there is $ c_0>0 $ such that for all $ z,\, z'\in\S $
	\begin{align*}
		\Delta_i\varphi(z,z') \geq c_0(1+xx').
	\end{align*}
	Thus we obtain
	\begin{align*}
		\dfrac{d}{dt}\M^2_{11}(t) &= \sum_i\dfrac{\alpha_i}{2}\int_{\S}\int_{\S} \K_i(z,z')\, \Delta_i\varphi(z,z)\, f(t,dz)f(t,dz')
		\\
		&\geq \dfrac{c_0\alpha_1}{2}\int_{\S}\int_{\S} \K_1(z,z')\, (1+xx') f(t,dz)f(t,dz') \geq \dfrac{c_0\alpha_1}{2}\M^2_{11}(t)^2.
	\end{align*}
	Since $ \M^2_{11}(0)>0 $ a comparison argument shows that $ \M^2_{11}(t) $ blows up in finite time.
	
	\textit{Case 2. $ \alpha_2>0 $.} We now assume $ \alpha_2>0 $. In this case we look at $ \varphi(z)= xy $. Here, we get
	\begin{align*}
		\Delta_i\varphi(z,z') = xy'+x'y+(x+x')\xi_{i,2} + (y+y')\xi_{i,1} + \xi_{i,1}\xi_{i,2}, \quad i=1,2,3.
	\end{align*}
	The minimum of this is attained at $ (z,z')=(2,2,2,2) $. We obtain 
	\begin{align*}
		\Delta_1\varphi(z,z') \geq 2, \quad \Delta_2\varphi(z,z') \geq 1, \quad \Delta_3\varphi(z,z')\geq 0.
	\end{align*}
	Hence, there is $ c_0>0 $ such that for all $ z,\, z'\in\S $
	\begin{align*}
		\Delta_2\varphi(z,z') \geq c_0(1+xy'+x'y).
	\end{align*}
	This implies
	\begin{align*}
		\dfrac{d}{dt}\M^2_{12}(t) &\geq \dfrac{\alpha_2}{2}\int_{\S}\int_{\S} \K_2(z,z')\, \Delta_2\varphi(z,z)\, f(t,dz)f(t,dz')
		\\
		&\geq \dfrac{c_0\alpha_2}{2}\int_{\S}\int_{\S} (xy'+x'y)(1+xy'+x'y) \, f(t,dz)f(t,dz') \geq c_0\alpha_2\M^2_{12}(t)^2.
	\end{align*}
	Thus, $ \M^2_{12}(t) $ blows up in finite time.
	
	\textit{Case 3.} We now assume $ \alpha_1=\alpha_2=0 $ but $ \alpha_3>0 $. In this case, we compute explicitly the second moments. Using a time-change we can assume $ \alpha_3=1 $.  For the second moments we obtain, see Proposition \ref{pro:MomentEquations} (ii),
	\begin{align}
		\dfrac{d}{dt}\M^2 = (\M^2 + \xi_3 \otimes \M^1 ) K_3 (\M^2+ \M^1 \otimes \xi_3  )  - \frac{1}{2 } ( \xi_3 \otimes \M^1 ) K_3  ( \xi_3 \otimes \M^3 ).
	\end{align}
	and defining $ \Nmom=\M^2 + \M^1\otimes\xi_3 $ yields with the equations for the first moments
	\begin{align*}
		\dfrac{d}{dt}\Nmom = \Nmom^\top K_3\Nmom.
	\end{align*}
	Writing
	\begin{align*}
		\Nmom = \begin{pmatrix} a(t) & b(t) \\ c(t) & d(t) \end{pmatrix}
	\end{align*}
	yields the system
	\begin{align*}
		\dot a = c^2, \quad
		\dot b = cd, \quad
		\dot c = cd, \quad
		\dot d = d^2.
	\end{align*}
	The solution to this system is
	\begin{align*}
		a(t) = a_0+\dfrac{c_0^2\, t}{1-d_0 t},
		\quad b(t) = b_0+\dfrac{c_0 d_0\, t}{1-d_0 t}, \quad c(t) = \dfrac{c_0}{1-d_0 t}, \quad d(t) &= \frac{d_0}{1-d_0 t}.
	\end{align*}
	Thus, $ \Nmom(t) $ and hence $ \M^2(t) $ blows up in finite time if and only if $ d_0>0 $, i.e.
	\begin{align*}
		0<\M^2_{22}(0)-2\M^1_2(0) = \int_{\S} y(y-2) f_0(dz).
	\end{align*}
	This concludes the proof.
\end{proof}

The following lemma will be useful to characterize the behaviour at a blow-up.
\begin{lem}\label{lem:PrepBlowUp}
	Let $\alpha\in \R_+^3 $, $ f_0\in \M_{3, + }(\S) $ and $ f\in C^1([0,T_*); \Meas_{+})\cap L^\infty_{\loc}([0,T_*); \Meas_{3,+}) $ be the unique weak solution to \eqref{eq:sec1:ContWeakForm} on the maximal time interval $ [0,T_*) $. Assume that $ \alpha_1,\, \alpha_2>0 $. Then, one of the following holds:
	\begin{enumerate}[(i)]
		\item either $ \M^2_{11}(f(t)),\,  \M^2_{22}(f(t))\to \infty $ as $ t\to T_* $,
		
		\item or $ \sup_{t \in [0,T_*)}\left\lbrace \M^2_{11}(f(t)),\,  \M^2_{22}(f(t))\right\rbrace  < \infty $.
	\end{enumerate}
\end{lem}
\begin{proof}
	We will show the following differential inequalities
	\begin{align*}
		\dfrac{d}{dt}\M^2_{11}(t) &\geq c\left( \M^2_{11}(t)^2+\M^2_{12}(t)^2 \right),
		\\
		\dfrac{d}{dt}\M^2_{12}(t) &\geq c\M^2_{11}(t) \M^2_{12}(t),
	\end{align*}
	for some constant $ c>0 $. This implies the claim. Indeed, this system shows that whenever $ \M^2_{11} $ blows up then also $ \M^2_{12} $ and vice verse. Note that always $ \M^2_{11}, \, \M^2_{12}>0 $ since $ f $ is supported in $ \S $. Thus, at the blow-up time $ T_*<\infty $ either both  $ \M^2_{11},\,\M_{12}^2 $ tend to infinity or both remain bounded. The latter occurs when $ \M_{22}^2 $ blows up before $ \M^2_{11} $.
	
	To prove the above inequalities, we use a similar argument as in the proof of Proposition \ref{pro:BlowUp}. For $ \varphi(z)=x^2 $ we obtain
	\begin{align*}
		\Delta_i\varphi(z,z') &= 2xx'+(x+x')\xi_{i,1} + \xi_{i,1}^2, \quad i=1,2,3.
		\\
		\Delta_1\varphi(z,z') &\geq 1,\quad \Delta_2\varphi(z,z') \geq 5 \quad \Delta_3\varphi(z,z')=2xx'.
	\end{align*}
	In particular, we get for some $ \delta>0 $
	\begin{align*}
		\Delta_1\varphi(z,z') &\geq \delta xx'.
	\end{align*}
	This yields
	\begin{align*}
		\dfrac{d}{dt}\M^2_{11}\geq \dfrac{1}{2}\delta\alpha_1(\M^2_{11})^2+\alpha_3(\M^2_{12})^2.
	\end{align*}
	On the other hand, for $ \varphi(z)=xy $ we have
	\begin{align*}
			\Delta_i\varphi(z,z') &= xy'+x'y+(x+x')\xi_{i,2} + (y+y')\xi_{i,1} + \xi_{i,1}\xi_{i,2}, \quad i=1,2,3.
			\\
			\Delta_1\varphi(z,z') &\geq 2, \quad \Delta_2\varphi(z,z') \geq 1, \quad \Delta_3\varphi(z,z')\geq 0.
	\end{align*}
	In particular, we have for some $ \delta>0 $
	\begin{align*}
		\Delta_1\varphi(z,z') \geq \delta(xy'+x'y).
	\end{align*}
	This implies
	\begin{align*}
		\dfrac{d}{dt}\M^2_{12}(t) \geq \dfrac{\delta\alpha_1}{2}\int_{\S}\int_{\S}xx'(xy'+x'y)\, f(t,dz)f(t,dz')\geq \delta\alpha_1 \M^2_{11}(t)\M^2_{12}(t).
	\end{align*}
	This concludes the proof.
\end{proof}

In the following we describe the structure of the second moments when a blow-up occurs.
\begin{pro}\label{pro:SecondMomentsLocalization}
	Let $ \alpha\in \R_+^3 $, $ f_0\in \M_{3, + }(\S) $ and $ f\in C^1([0,T_*); \Meas_{+})\cap L^\infty_{\loc}([0,T_*); \Meas_{3,+}) $ be the unique weak solution to \eqref{eq:sec1:ContWeakForm} on the maximal time interval $ [0,T_*) $. Let $ \M^2(t) $ be the matrix of second moments of $ f(t) $. Assume that $ T_*<\infty $, i.e. $ \M^2 $ blows-up in finite time. Then, there is a non-zero $ \theta\in \R_+^2 $ and some constant $ C>0 $ such that
	\begin{align*}
		\norm{\dfrac{d}{dt}\left[ (T_*-t)\M^2(t) \right] } \leq C,
		\quad
		\norm{(T_*-t)\M^2(t) -  \theta \otimes \theta} \leq C(T_*-t)
	\end{align*} 
	for all $ t\in [0,T_*) $. Furthermore, $ \theta $ satisfies $ \sum_{i=1}^3\alpha_i\theta^\top K_i\theta = 1 $.
\end{pro}
\begin{proof}
	We divide the proof into several steps.
	
	\textit{Step 1.} Let us observe that $ \M^2 $ is symmetric and non-negative definite. We now show that $ t\mapsto \M^2(t) $ is meromorphic on a neighbourhood of $ [0,T_*] $ in $ \C $. To this end, recall from Proposition \ref{pro:MomentEquations} the second moments $ \M^2 $ satisfy \eqref{eq:MomentEq2nd}. We can reformulate this in the form
	\begin{align}\label{eq:MomentEq2}
		\dfrac{d}{dt}\M^2 = \M^2K\M^2 + \M^2A+A^\top\M^2+B,
	\end{align}
	where $ A=A(t)\in\R^{2\times2} $ and $ B(t)\in\R^{2\times2} $ are depending on $ \M^1 $ and $ K=\sum_i\alpha_iK_i $. Furthermore, $ B $ is symmetric. Since $ \M^1 $ solves \eqref{eq:MomentEq1st} it can be extended to a analytic function on a neighbourhood of $ [0,T_*] $ in $ \C $, due to the bound in Proposition \ref{pro:1stOderMomentEquations}. In particular, $ A,\, B $ are analytic on a neighbourhood of $ [0,T_*] $ in $ \C $. We now show that $ \M^2 $ is meromorphic on a neighbourhood of $ [0,T_*] $ in $ \C $. To this end, we use the standard ansatz for Riccati type equations, i.e. we write $ \M^2(t) = U(t)V(t)^{-1} $. The matrices $ U, \, V $ satisfy the system
	\begin{align*}
		\frac{d}{dt}
		\begin{pmatrix}
			U \\
			V
		\end{pmatrix}
		=
		\begin{pmatrix}
			A^\top & B
			\\
			-K & -A
		\end{pmatrix}
		\begin{pmatrix}
			U \\
			V
		\end{pmatrix}, \quad \begin{pmatrix}
			U(0) \\
			V(0)
		\end{pmatrix} = \begin{pmatrix}
			\M^2(0) \\
			I
		\end{pmatrix}.
	\end{align*}
	Consequently, $ U,\ V $ are analytic on a complex neighbourhood of $ [0,T_*] $. Moreover, $ U(t)V(t)^{-1} $ solves \eqref{eq:MomentEq2}, hence by uniqueness $ \M^2(t)= U(t)V(t)^{-1} $ as long as $ V(t) $ is invertible. On the other hand, this formula shows that $ \M^2 $ is meromorphic on a complex neighbourhood of $ [0,T_*] $. 
	
	\textit{Step 2.} We now show that the blow-up of $ \M^2 $ at $ T_* $ is a simple pole. We do this by splitting into two cases.
	\begin{enumerate}[(i)]
		\item We first assume that $ K $ is invertible. Note that poles appear exactly when the matrix $ V(t) $ is not invertible and hence at the blow-up time $ T_* $ we have $ \det V(T_*)=0 $. Let $ v\in \R^2 $ be non-zero and such that $ V(T_*)v=0 $. Define now $ (x(t),y(t))= (U(t)v,V(t)v) $. We obtain the system
		\begin{align*}
			\frac{d}{dt}
			\begin{pmatrix}
				x \\
				y
			\end{pmatrix}
			=
			\begin{pmatrix}
				A^\top & B
				\\
				-K & -A
			\end{pmatrix}
			\begin{pmatrix}
				x \\
				y
			\end{pmatrix}, \quad \begin{pmatrix}
				x(0) \\
				y(0)
			\end{pmatrix} = \begin{pmatrix}
				\M^2(0)v \\
				v
			\end{pmatrix}.
		\end{align*}
		First of all notice that $ x(T_*)\neq 0 $, otherwise $ (x(t),y(t)) $ vanishes at $ t=T_* $. Since it solve the above homogeneous system we would have $ (x(t),y(t))\equiv0 $. This is not consistent with the initial condition $ (\M^2(0)v,v)\neq 0 $. We thus obtain at $ t=T_* $
		\begin{align}
			y'(T_*) = -Kx(T_*) \neq 0.
		\end{align}
		The latter follows from the fact that $ K $ has trivial kernel. Thus, $ y(t) = \mathcal{O}(t-T*) $ vanishes linearly in $ T_* $. In particular, we see that $ V(t) $ becomes singular only linearly. Consequently, the pole of $ \M^2(t) = U(t)V^{-1}(t) $ is simple.
		
		\item We now look at the case of non-invertible $ K $. Since $ K $ is symmetric, we can write $ K=\lambda k\otimes k $ for some $ k\in \R^2 $ and $ \lambda\in \R $. Since $ K $ has only non-negative entries we have $ \lambda>0 $ and $ k\in \R_+^2 $. Assume that $ \M^2 $ has a pole of order $ m\geq2 $ at $ t=T_* $. Then, the Laurent series expansion has the form
		\begin{align*}
			\M^2(t) = \dfrac{\M_m}{(T_*-t)^m} + R(t).
		\end{align*}
		Here, $ R(t) $ has a Laurent series with a pole of order at most $ m-1 $ at $ t=T_* $. Since $ \M^2 $ is a symmetric matrix with non-negative entries, this also holds form $ \M_m $. We then obtain by comparing coefficients in \eqref{eq:MomentEq2}, since $ m\geq2 $,
		\begin{align*}
			\M_mK\M_m = 0.
		\end{align*}
		This requires $ \M_m $ to be non-invertible. Thus, $ \M_m=v\otimes v $ for some vector $ v\in \R^2_+ $, since $ \M_m $ has only non-negative entries. Hence, we obtain with the form of $ K $
		\begin{align*}
			0 = v^\top Kv = \lambda (v\cdot k)^2.
		\end{align*}
		However, we have $ v,\, k\in \R_+^2 $ and $ v\cdot k=0 $. Thus, $ v=0 $ and hence also $ \M_m=0 $. This is contradicting that $ \M^2 $ has a pole of order $ m\geq2 $ at $ t=T_* $. Thus, the pole is of order $ m=1 $.
	\end{enumerate}
	
	\textit{Step 3.} We now prove the last assertions. Using again the Laurent series expansion around the blow-up time $ T_* $ we obtain
	\begin{align*}
		\M^2(t) = \dfrac{\M_{*}}{T_*-t}+R(t).
	\end{align*}
	Here, $ R $ is analytic around $ T_* $. Note that $ \M_{-1} $ is necessarily symmetric and non-negative definite, since $ \M^2 $ is. We use again the equation \eqref{eq:MomentEq2} and match terms in the Laurent series yielding
	\begin{align*}
		\M_{*} = \M_{*}K\M_{*}.
	\end{align*}
	If $ \M_{*} $ is not invertible then necessarily $ \M_{*}=\theta\otimes\theta $ for some $ \theta\in \R_{+}^2 $, since $ \M_{*} $ is symmetric with non-negative entries. Furthermore, we obtain from the preceding equation $ \theta^\top K\theta=1 $.
	 
	We now show that $ \M_{*} $ cannot be invertible. If this would be the case then $ K=\M_{*}^{-1} $. In particular, $ K $ has to be invertible too. Observe that then
	\begin{align*}
		\M_{*} =K^{-1}= \dfrac{1}{\alpha_1\alpha_3-\alpha_2^2}\begin{pmatrix}
			\alpha_3 & -\alpha_2 \\ -\alpha_2 & \alpha_1
		\end{pmatrix}.
	\end{align*}
	Since $ \M_{*} $ has non-negative entries we necessarily have $ \alpha_2=0 $ and $ \alpha_1,\, \alpha_2>0 $. However, this asymptotics then implies that
	\begin{align*}
		\M_{11}^2(t)\to \infty, \quad \M_{22}^2(t)\to \infty
	\end{align*}
	as $ t\to T_* $, but $ \sup_{t \in [0,T_*)} \M_{12}^2(t)<\infty $. This is a contradiction to Lemma \ref{lem:PrepBlowUp}. Thus, $ \M_{*} $ is not invertible. 
	
	The asymptotics
	\begin{align*}
		\M^2(t) = \dfrac{\theta\otimes\theta}{T_*-t}+R(t).
	\end{align*}
	then implies also that for some constant $ C>0 $ it holds
	\begin{align*}
		\norm{\dfrac{d}{dt}\left[ (T_*-t)\M^2(t) \right] } \leq C,
		\quad
		\norm{(T_*-t)\M^2(t) -  \theta \otimes \theta} \leq C(T_*-t)
	\end{align*} 
	for all $ t< T_* $. This concludes the proof.
\end{proof}

\subsubsection{Third order moments}
In the following result we also give the asymptotics of the third moments close to the gelation time.
\begin{pro}\label{pro:ThirdMomentsLocalization}
	Let $ \alpha\in \R_+^3 $, $ f_0\in \M_{4, + }(\S) $ and $ f\in C^1([0,T_*); \Meas_{+})\cap L^\infty_{\loc}([0,T_*); \Meas_{4,+}) $ be the unique weak solution to \eqref{eq:sec1:ContWeakForm} on the maximal time interval $ [0,T_*) $. Let $ \M^3(t) $ be the $ 3 $-tensor of third moments of $ f(t) $. Assume that $ T_*<\infty $ and let $ \theta\in \R^2_+ $ as in Proposition \ref{pro:SecondMomentsLocalization}. Then, there is are constants $ C>0 $ and $ c_0>0 $ such that for all $ t\in [0,T_*) $
	\begin{align*}
		\norm{(T_*-t)^3\M^3(t) -  c_0\theta \otimes \theta\otimes\theta} \leq C(T_*-t).
	\end{align*} 
\end{pro}
\begin{proof}
	We split the proof into two steps.
	
	\textit{Step 1. ODE system.} We first derive the ODE satisfied by the third moments. To this end, we can use as a test function $ \varphi_R(z)=\psi_R(|z|)(z\otimes z\otimes z) $ and let $ R\to \infty $. Here, we make use of the finite fourth moments, similarly as in the proof of Proposition \ref{pro:MomentEquations}. This allows to show that $ t\mapsto \M^3(f(t)) $ belongs to $ C^1([0,T]; (\R^2)^{\otimes 3} ) $. In order to identify the equation we make use of the tensor formulation. We observe that for $ \varphi(z)= z\otimes z\otimes z $ we have
	\begin{align*}
		\Delta_i\varphi(z,z') &= 3 \P_3(z' \otimes z \otimes z) + 3 \P_3(z \otimes z' \otimes z' ) + 3 \P_3(z \otimes z \otimes \xi_i ) + 3  \P_3(z' \otimes z' \otimes \xi_i ) 
		\\
		&\quad +  6 \P_3(z \otimes z' \otimes \xi_i ) + 3  \P_3(z \otimes \xi_i \otimes \xi_i ) + 3 \P_3( z' \otimes \xi_i \otimes \xi_i )  + \xi_i \otimes \xi \otimes \xi_i
	\end{align*}
	Making use of the quadratic form of the kernel we obtain an equation of the form
	\begin{align}\label{eq:sec3:ThirdMomEq}
		\dfrac{d}{dt}\M^3 = \mathcal B (\M^3, \M^2 ) + \mathcal R_1 (\M^3, \M^1 ) +  \mathcal R_2 (\M^2, \M^2 ) + \mathcal R_3 (\M^2, \M^1 ) + \mathcal R_4 (\M^1, \M^1 )
	\end{align}
	Here, $ \mathcal B $ as well as $ \mathcal R_j $ for $ j=1,2,3,4 $ are bilinear operators with values in $ (\R^2)^{\otimes3}_{\sym} $. In the following only the precise form of the first operator is relevant. It has the form
	\begin{align*}
		\mathcal B (\M^3, \M^2 )_{k\ell m} = \sum_{ a,b=1}^2 K_{ab} \left(  \M^2_{ka } \M^3_{b\ell m} + \M^{2}_{\ell a } \M^3_{bmk} + \M^2_{ma} \M^3_{b\ell k } \right), \quad k,\ell,m \in \{ 1, 2\},
	\end{align*}
	where $ K=\sum_{i =1}^3\alpha_iK_i $. We can rewrite this as
	\begin{align*}
		\mathcal{B} (\M^3, \M^2 ) = 3\P_3 \mathcal{A}(\M^3, \M^2),
	\end{align*}
	where $ \mathcal{A} : (\R^2)^{\otimes3} \times (\R^2)^{\otimes2} \to (\R^2)^{\otimes3}  $ is given by
	\begin{align*}
		\mathcal{A}(J,T)_{k\ell m} = \sum_{ a,b=1}^2 K_{ab} J_{ka } T_{b\ell m}.
	\end{align*}

	The most relevant contribution on the right hand side in \eqref{eq:sec3:ThirdMomEq} is the first term when $ \M^2 $ is replaced by $ (T_*-t)^{-1}\theta\otimes\theta $. Recall that $ \mathcal{N}^2(t) := \M^2(t)-(T_*-t)^{-1}\theta\otimes\theta $ is uniformly bounded on $ [0,T_*) $ by Proposition \ref{pro:SecondMomentsLocalization}. We then define
	\begin{align*}
		\Lop\M^3 = B(\M^3, \theta\otimes\theta).
	\end{align*}
	For the remaining terms in \ref{pro:SecondMomentsLocalization} we obtain the bounds
	\begin{align*}
		\norm{\mathcal{B}\left(\M^3,\M^2-(T_*-t)^{-1}\theta\otimes\theta\right)} &\leq C \norm{\M^3}, \quad \norm{\mathcal R_1 (\M^3, \M^1 )} \leq C \norm{\M^3},
		\\
		\norm{\mathcal R_2 (\M^2, \M^2 )} &\leq \dfrac{C}{(T_*-t)^2},
		\norm{\mathcal R_3 (\M^2, \M^1 )} \leq \dfrac{C}{T_*-t}, \quad \norm{\mathcal R_4 (\M^1, \M^1 )}\leq C
	\end{align*}
	for all $ t\in [0,T_*) $.

	We now compute a coordinate representation of the operator $ \mathcal{L} $. To this end, we use the following basis for $ (\R^2)^{\otimes 3} $
	\begin{align*}
		\B_1:=\theta^{\otimes 3}, \quad \B_2:=\P_3\left(\theta\otimes \theta \otimes \theta^\perp \right), \quad \B_3:=\P_3\left(\theta\otimes \theta^\perp \otimes \theta^\perp \right), \quad \B_4:=(\theta^\perp)^{\otimes 3}.
	\end{align*}
	Let us also define $ \beta:=\theta^\top K \theta^\perp $. Recall that $ \theta^\top K \theta=1 $ by Proposition \ref{pro:SecondMomentsLocalization}. With this one can see that
	\begin{align*}
		\Lop \B_1 = 3\B_1, \quad \Lop \B_2 = \beta \B_1 + 2\B_2, \quad \Lop \B_3 = 2\beta \B_2 + \B_3, \quad \Lop \B_4 = 3\beta \B_3.
	\end{align*}
	Thus, with respect to this basis we can identify $ \Lop $ with the matrix
	\begin{align*}
		\hat{\Lop} := \begin{pmatrix}
			3 & \beta & 0 & 0
			\\
			0 & 2 & 2\beta & 0
			\\
			0 & 0 & 1 & 3\beta
			\\
			0 & 0 & 0 & 0
		\end{pmatrix}
	\end{align*}
	Note that this matrix has the eigenvalues $ 0,1,2,3 $.
	
	Let us denote by $ V(t)\in \R^4 $ the coordinate representation of $ \M^3(t) $ with respect to the above basis. The above arguments yields the following equation
	\begin{align}\label{eq:sec3:ProofThirdMoments}
		\dfrac{d}{dt}V = \dfrac{1}{T_*-t}\hat{\Lop}V + \hat{B}V + \hat{R},
	\end{align}
	where $ \hat{B}(t),\,  \hat{R}(t)\in \R^{4\times 4} $ satisfy
	\begin{align*}
		\norm{\hat{B}(t)}\leq C, \quad \norm{\hat{R}(t)}\leq \dfrac{C}{(T_*-t)^2}.
	\end{align*}
	Using now the time-change $ t(\tau)=T_*(1-e^{-\tau}) $ and writing $ V(t(\tau))e^{-3\tau} $ in the basis formed by the eigenvectors of $ \hat{\Lop} $ yields the equation
	\begin{align}\label{eq:sec3:Proof3rdMomODE}
		\dfrac{d}{d\tau}W = \bar{\Lop}W + \bar{B}(\tau)W + \bar{R}(\tau).
	\end{align} 
	Here, $ W(\tau) $ are the coordinates of $ V(t(\tau))e^{-3\tau} $ in this basis, $ \bar{\Lop} = \operatorname{diag}\{0,-1,-2,-3\} $ and the matrices $ \bar{B},\, \bar{R} $ satisfy the bounds
	\begin{align*}
		\norm{\bar{B}(\tau)} \leq Ce^{-\tau},\quad \norm{\bar{R}(\tau)} \leq Ce^{-2\tau}.
	\end{align*}
	
	\textit{Step 2. Study of Asymptotics.}  We now show that $ \M^3(t) $ blows up like $ c_0(T_*-t)^{-3}\B_1 $ at $ t=T_* $. Note that $ \B_1 $ corresponds to the eigenvector to the eigenvalue $ 0 $ of $ \bar{\Lop} $. Here, $ c_0>0 $ is some constant.
	
	To this end, we first show that $ \M^3(t) $ necessarily blows up at $ t=T_* $. This follows from the following inequality relating moments of different order
	\begin{align*}
		\norm{\M^2(t)} \leq C \norm{\M^1(t)}^{1/2}\norm{\M^3(t)}^{1/2}.
	\end{align*}
	Since $ \norm{\M^1(t)}>0 $ for $ t\in [0,T_*] $ we conclude from the asymptotics of the second moments that 
	\begin{align*}
		\norm{\M^3(t)} \geq \dfrac{c}{(T_*-t)^2}.
	\end{align*}
	Writing this in terms of $ W(\tau) $ yields $ \norm{W(\tau)}\geq ce^{-\tau} $ for some constant $ c>0 $. 
	
	We now study the solution $ W $ to the ODE system \eqref{eq:sec3:Proof3rdMomODE}. Due to the bounds on $ \bar{B},\, \bar{R} $ we obtain
	\begin{align*}
		\sup_{ \tau \in [0,\infty )} \norm{W(\tau)}<\infty.
	\end{align*}
	Using Duhamel's formula we obtain
	\begin{align*}
		W(\tau) = W^\infty + \bar{W}(\tau), \quad \norm{\bar{W}(\tau)} \leq Ce^{-\tau}.
	\end{align*}
	Using this expression in \eqref{eq:sec3:Proof3rdMomODE} implies that $ \bar{\Lop}W^\infty=0 $, i.e.  $ W^\infty=c_0(1,0,0,0)^\top $.
	
	If $ c_0\neq 0 $, we conclude that 
	\begin{align*}
		\norm{(T_*-t)^3\M^3(t)-c_0\B_1} \leq C(T_*-t).
	\end{align*} 
	This in particular shows that $ c_0>0 $ since $ \M_{3} $ has only non-negative components $ \M^3_{k\ell m}\geq0 $ for all $ k,\, \ell,\, m \in \{1,2\} $.
	
	We now assume for contradiction that $ c_0=0 $. In particular, we obtain $ W(\tau) =\bar{W}(\tau) $. We now consider only the first component in \eqref{eq:sec3:Proof3rdMomODE} yielding upon integration
	\begin{align*}
		W_1(\tau) = W_1(0) + \int_0^\tau \left( (\bar{B}(\sigma)W(\sigma))_1 + \bar{R}_1(\sigma)\right) \, d\sigma = -\int_\tau^\infty \left( (\bar{B}(\sigma)\bar{W}(\sigma))_1 + \bar{R}_1(\sigma)\right) \, d\sigma.
	\end{align*}
	Here, we used that $ W^\infty_1=0 $. Due to the estimate available for $ \bar{W} $ we conclude that
	\begin{align*}
		|W_1(\tau) | \leq Ce^{-2\tau}.
	\end{align*}
	Writing $ U(\tau) = e^\tau(W_2(\tau),W_3(\tau),W_4(\tau)) $ yields then the equation
	\begin{align*}
		\dfrac{d}{d\tau} U = \tilde{\Lop}U + \tilde{B}(\tau) U + \tilde{R}(\tau).
	\end{align*}
	Here, $ \tilde{\Lop} = \operatorname{diag}\{0,-1,-2\} $,  $ \tilde{B} $ is restricted to the last three components and $ \norm{\tilde{R}(\tau)}\leq Ce^{-\tau} $.  Note that the terms containing $ W_1 $ are absorbed in $ \tilde{R} $. As before this yields
	\begin{align*}
		U(\tau) = U^\infty + \bar{U}(\tau), \quad \norm{\bar{U}(\tau)} \leq Ce^{-\tau}
	\end{align*}
	with $ U^\infty = c_1(1,0,0) $. In particular, we obtain
	\begin{align*}
		W(\tau) = c_1e^{-\tau}(0,1,0,0)^\top + \tilde{W}(\tau), \quad \norm{\tilde{W}(\tau)} \leq Ce^{-2\tau}.
	\end{align*}
	Since we know a priori that $ \norm{W(\tau)}\geq ce^{-\tau} $, we concluded that $ c_1\neq 0 $. Formulating this in terms of the second moments yields
	\begin{align}\label{eq:sec3:ProofThirdMoments2}
		(T_*-t)^2\M^3(t) \to c_1 \, \Xi, \quad \Xi := \beta\B_1-3\B_2.
	\end{align}
	Note that $ \Xi $ is the eigenvector with eigenvalue $ 2 $ to $ \hat{\Lop} $ which corresponds to the eigenvector with eigenvalue $ -1 $ to $ \bar{\Lop} $.
	
	We now show that \eqref{eq:sec3:ProofThirdMoments2} cannot hold, since $ \M^3 $ satisfies some positive definiteness properties. Indeed, the matrix 
	\begin{align*}
		(k,\ell)\mapsto \sum_{ a,b=1} \M^3(t)_{k\ell a} K_{ab}\theta_b = \int_{ \S_i} z_kz_\ell \left( z^\top K\theta \right)  \, f(t,dz)
	\end{align*}
	is non-negative definite for all $ t\in [0,T_*) $, since $ z^\top K\theta\geq0 $. Recall that $ \theta\in \R_{+}^2 $ and $ K $ has only non-negative entries. But then \eqref{eq:sec3:ProofThirdMoments2} would imply the same for
	\begin{align*}
		(k,\ell)\mapsto c\sum_{ a,b=1} \Xi_{k\ell a} K_{ab}\theta_b &= c\left[\beta (\theta\otimes\theta)_{k\ell} - \dfrac{1}{2}\left( \beta (\theta\otimes\theta)_{k\ell} + (\theta\otimes\theta^\perp)_{k\ell} +(\theta^\perp\otimes\theta)_{k\ell} \right)  \right]  
		\\
		&= \dfrac{c}{2}\left[\beta(\theta\otimes\theta)_{k\ell} - (\theta\otimes\theta^\perp)_{k\ell} -(\theta^\perp\otimes\theta)_{k\ell}  \right].
	\end{align*}
	This matrix has the following form with respect to the basis $ \{\theta,\theta^\perp\} $
	\begin{align*}
		\dfrac{c}{2}\begin{pmatrix}
			\beta & - 1
			\\
			-1 & 0
		\end{pmatrix}.
	\end{align*}
	This matrix has determinant $ -1 $ and hence is not positive definite for any choice $ c\neq 0 $. This yields a contradiction.
	
	We hence infer our assumption $ c_0=0 $ cannot hold so that the conclusion of $ c_0\neq0 $ holds. This concludes the proof.
\end{proof}

\subsubsection{Fourth order moments}
In the following we study the fourth order moments. In this case only an upper bound is proven rather than an asymptotics (as was done for the second and third order moments).

\begin{pro}\label{pro:ForthMomentsBound}
	Let $ \alpha\in \R_+^3 $, $ f_0\in \M_{4, + }(\S) $ and $ f\in C^1([0,T_*); \Meas_{+})\cap L^\infty_{\loc}([0,T_*); \Meas_{4,+}) $ be the unique weak solution to \eqref{eq:sec1:ContWeakForm} on the maximal time interval $ [0,T_*) $. Let $ \M^4(t) $ be the $ 4 $-tensor of fourth moments of $ f(t) $. Assume that $ T_*<\infty $. Then, there is a constant $ C>0 $ such that for all $ t\in [0,T_*) $
	\begin{align*}
		\norm{(T_*-t)^5\M^4(t)} \leq C.
	\end{align*} 
\end{pro}
\begin{proof}
	We split the proof into two steps. In the first two steps we consider a solution to with initial condition $ f_0\in \M_{5,+}(\S) $. By Theorem \ref{thm:WellPosed} we know that there is a unique solution $ f\in C([0,T_*); \Meas_{+})\cap L^\infty_{\loc}([0,T_*); \Meas_{5,+}) $. In Step 3 we show how to reduce the general case to this situation.
	
	\textit{Step 1.} We first derive the ODEs satisfied by the $ 4 $-tensor of fourth moments. Since the solution has finite $ 5 $-th moments, we can use the test functions $ \varphi_R(z)=\psi_R(|z|)z^{\otimes 4} $ and let $ R\to \infty $. We observe that for the test function $ \varphi(z) = z^{\otimes 4} $ we have
	\begin{align*}
		\Delta_{i}\varphi(z,z') &= 4\P_4\left( z^{\otimes 3}\otimes z' \right) + 4\P_4\left( z\otimes (z')^{\otimes 3} \right) + 6 \P_4\left( z\otimes z\otimes z'\otimes z' \right)
		\\
		&\quad +4\P_4\left( z^{\otimes 3}\otimes \xi_i \right) +4\P_4\left( (z')^{\otimes 3}\otimes \xi_i \right) + 12\P_4\left( z^{\otimes 2}\otimes z'\otimes \xi_i \right) + 12\P_4\left( (z')^{\otimes 2}\otimes z\otimes \xi_i \right) 
		\\
		&\quad + 6 \P_4\left( z^{\otimes2} \otimes \xi_i^{\otimes2} \right)+6 \left( (z')^{\otimes2} \otimes \xi_i^{\otimes2} \right) + 12\P_4\left( z\otimes z' \otimes \xi_i^{\otimes2} \right)
		\\
		&\quad + 4\P_4\left( z \otimes \xi_i^{\otimes3}\right) +4\P_4\left( z' \otimes \xi_i^{\otimes3} \right) + \xi_i^{\otimes4}.
	\end{align*}
	With the quadratic form of the kernel this yields an ODE system for $ \M^4 $ of the form
	\begin{align}\label{eq:sec3:ProofFourthMoment}
		\begin{split}
			\dfrac{d}{dt}\M^4 &= \mathcal{B}_1(\M^4,\M^2) + \mathcal{B}_2(\M^3,\M^3) + \mathcal{R}_1(\M^4,\M^1)+\mathcal{R}_2(\M^3,\M^2) 
			\\
			&\quad + \mathcal{R}_3(\M^3,\M^1)+\mathcal{R}_4(\M^2,\M^2) + \mathcal{R}_5(\M^2,\M^1)+\mathcal{R}_6(\M^1,\M^1).
		\end{split}
	\end{align}
	Here, the operators $ \mathcal{B}_1,\, \mathcal{B}_2 $ and $ \mathcal{R}_j $, $ j=1,\ldots,6 $, are bilinear with values in $ (\R^2)^{\otimes 4}_{\sym} $. One can see that
	\begin{align*}
		\mathcal{B}_1(\M^3, \M^2 ) = 4\P_4 \mathcal{A}_2(\M^4, \M^2),
	\end{align*}
	where $ \mathcal{A}_2 : (\R^2)^{\otimes4} \times (\R^2)^{\otimes2} \to (\R^2)^{\otimes4}  $ is given by
	\begin{align*}
		\mathcal{A}_2(J,T)_{k\ell m n} = \sum_{ a,b=1}^2 K_{ab} J_{ka } T_{b\ell m n}, \quad K=\sum_{i=1}^3\alpha_iK_i.
	\end{align*}

	\textit{Step 2.} We now study the ODE \eqref{eq:sec3:ProofFourthMoment}. In the first term in \eqref{eq:sec3:ProofFourthMoment} we use similarly as in the proof or Proposition \ref{pro:ThirdMomentsLocalization} the asymptotics of $ \M^2 $ according to Proposition \ref{pro:SecondMomentsLocalization}. We henceforth define
	\begin{align*}
		\Lop\M^4 = 3\P_4 \mathcal{A}_2(\M^4,\theta\otimes\theta).
	\end{align*}
	We then obtain an equation of the form
	\begin{align}\label{eq:sec3:ProofFourthMoment2}
		\dfrac{d}{dt}\M^4 = \dfrac{1}{T_*-t}\Lop \M^3 + \mathcal{B}(t)\M^4 + \mathcal{R}(t),
	\end{align} 
	where $ \mathcal{B}(t) $ is a linear operator $ (\R^2)^{\otimes 4}_{\sym}\to (\R^2)^{\otimes 4}_{\sym} $ and $ \mathcal{R}(t)\in (\R^2)^{\otimes 4}_{\sym} $. Using Proposition \ref{pro:SecondMomentsLocalization} and \ref{pro:ThirdMomentsLocalization} we obtain the bounds
	\begin{align*}
		\norm{\mathcal{B}\M^4} \leq C\norm{\M^4}, \quad \norm{\mathcal{R}(t)}\leq \dfrac{C}{(T_*-t)^6}.
	\end{align*}
	In order to analyse $ \Lop $ we use the following basis on $ (\R^2)^{\otimes4}_{\sym} $
	\begin{align*}
		\B_1:=&\theta^{\otimes4}, \quad \B_2:=\P_4\left( \theta^{\otimes3}\otimes\theta^\perp \right), \quad \B_3:=\P_4\left( \theta^{\otimes2}\otimes(\theta^\perp)^{\otimes 2} \right),
		\\
		\B_4:=&\P_4\left( \theta\otimes(\theta^\perp)^{\otimes3} \right), \quad \B_5:=(\theta^\perp)^{\otimes4}.
	\end{align*}
	Recall that $ \theta^\top K\theta =1 $ and let us write again $ \beta=\theta^\top K\theta^\perp $. The operator $ \Lop $ can then be written in matrix form as
	\begin{align*}
		\hat{\Lop} = \begin{pmatrix}
			4 & \beta & 0 & 0 & 0
			\\
			0 & 3 &  2\beta & 0 & 0
			\\
			0 & 0 & 2 &  3\beta & 0
			\\
			0 & 0 & 0 & 1 &  4\beta
			\\
			0 & 0 & 0 & 0 & 0
		\end{pmatrix}.
	\end{align*}
	In particular, we obtain $ \norm{e^{\hat{\Lop} t}}\leq Ce^{4t}  $ for all $ t\geq0 $. Let us write $ V(t) $ for the coordinate representation of $ \M^4(t) $ in the above basis. We then obtain from \eqref{eq:sec3:ProofFourthMoment2} the equation
	\begin{align*}
		\dfrac{d}{dt}V = \dfrac{1}{T_*-t}\hat{\Lop}V + \hat{\mathcal{B}}(t)V + \hat{\mathcal{R}}(t).
	\end{align*}
	We now introduce the time-change $ t(\tau) = T_*(1-e^{-\tau}) $ yielding for $ W(\tau) = V(t(\tau)) $
	\begin{align*}
		\dfrac{d}{d\tau}W = \hat{\Lop}W + T_*e^{-\tau}\hat{\mathcal{B}}(t(\tau))W + T_*e^{-\tau}\hat{\mathcal{R}}(t(\tau)).
	\end{align*}
	Observe that
	\begin{align*}
		\norm{\hat{\mathcal{B}}(t(\tau))W}\leq C\norm{W}, \quad \norm{\hat{\mathcal{R}}(t(\tau))}\leq Ce^{6\tau}.
	\end{align*}
	By Gronwall's lemma we obtain
	\begin{align*}
		\norm{W(\tau)} \leq Ce^{5\tau} \implies \norm{V(t)}\leq \dfrac{C}{(T_*-t)^5}.
	\end{align*}
	This yields the asserted bound. Let us mention that $ C>0 $ only depends on $ \M^j(f_0) $ for $ j=1,2,3,4 $.

	\textit{Step 3.} We now prove the bound when $ f_0\in \M_{4,+}(\S) $ and hence the solution satisfies merely $ f\in  L^\infty_{\loc}([0,T_*); \Meas_{4,+}) $. We now choose a sequence $ (f^n_0)_n $ in $ \Meas_{5,+}(\S) $ such that for all $ n\in \N $
	\begin{align}\label{eq:sec3:ProofFifthMomentApprox}
		\M^j(f^n_0) = \M^j(f_0), \quad j\in \{1,2,3,4\},
	\end{align}
	and $ f^n_0\to f_0 $ weakly as $ n\to \infty $. A way to construct these measures is to take the restriction $ f_0\ind_{[0,n]^2} $ and add a finite number of Diracs at specific points with specific weights to ensure \eqref{eq:sec3:ProofFifthMomentApprox}. This yields a set of equations for the weights which can be solved. In fact, this is related to the Caratheodory theorem in convex geometry (see Theorem B.12 in \cite{lasserre2009moments}). 
	
	From Step 2 we then obtain for the solutions $ f^n $ to $ f_0^n $
	\begin{align*}
		\norm{\M^4(f^n(t))} \leq \dfrac{C}{(T_*-t)^5}.
	\end{align*}
	Note that $ T_* $ as well as $ C $ is independent of $ n\in \N $ by \eqref{eq:sec3:ProofFifthMomentApprox}. By a compactness argument, similar to the one in the proof of Theorem \ref{thm:WellPosed}, we obtain $ f^n(t)\to f(t) $ weakly for all $ t\in [0,T_*) $. Thus, we obtain the asserted bound also for $ \M^4(f(t)) $. This concludes the proof.
\end{proof}

\subsection{Self-similar variables and localization} \label{sec:localization}
In this subsection we define the self-similar change of variables, which was already introduced in the introduction. We then reformulate the results concerning the moments proved in the preceding subsection. Furthermore, we then show that the distribution localizes along a line. These results correspond to statements (i) and (ii) in Theorem \ref{thm:IntrodLocaliSelfSim}. 

Given the solution $ f $ as in Theorem \ref{thm:WellPosed} with initial condition $ f_0\in \M_{4, + } $ we introduce the self-similar change of variables via
\begin{align} \label{eq:sec3:self-sim}
	F(\tau,\eta ) = (T_*-t(\tau))^{-7}f\left(t(\tau),(T_*-t(\tau))^{-2} \eta \right), \quad t(\tau) = T_*\left( 1-e^{-\tau}\right).
\end{align}
Recall that $ T_* $ is the blow-up time of the second moments, see Definition \ref{def:BlowUpTime}. We assume here that $ T_*<\infty $. Let us note that \eqref{eq:sec3:self-sim} has to be defined by duality, since $ f $ is merely a measure.

\begin{defi}\label{def:SelfSimVar}
	Consider $ \alpha\in \R_{+}^3 $, $ f_0\in \Meas_{4,+}(\S) $ and assume $ T_{\alpha,*}<\infty $. Let $ f $ be the solution to \eqref{eq:sec1:ContWeakForm} with initial condition $ f_0 $ according to Theorem \ref{thm:WellPosed}. We define the measure $ F\in C^1([0,\infty); \Meas_{+}(\R^2_+))\cap L^\infty_{\loc}([0,\infty); \Meas_{4,+}(\R^2_+)) $ by duality via
	\begin{align*}
		\int_{\R^2_+} \varphi(\eta) F(\tau,  d\eta  )= (T_* - t(\tau))^{- 3 } \int_{\S} \varphi(z (T_*- t(\tau))^2 ) f(t(\tau),  dz ) \quad t(\tau) = T_*\left( 1-e^{-\tau}\right)
	\end{align*}
	for every $\varphi \in C_b ( \R_+^2 ) $. 
\end{defi}
\begin{rem}
	Note that the support of the measure on the right hand side in \eqref{eq:sec3:self-sim} is given by $ e^{-2\tau}\S $. In Definition \ref{def:SelfSimVar} we implicitly extend the measure $ F(\tau) $ by zero outside of $ e^{-2\tau}\S $ to all of $ \R_+^2 $.
\end{rem}
We now derive the equation satisfied by $ F $ as in Definition \ref{def:SelfSimVar}.
\begin{lem}\label{lem:SelfSimEq}
	Consider $ \alpha\in \R_{+}^3 $, $ f_0\in \Meas_{4,+}(\S) $ and assume $ T_*=T_{\alpha,*}<\infty $. Let $ F $ be defined as in Definition \ref{def:SelfSimVar}. Then, $ F $ satisfies for all $ \varphi\in C_b(\R^2_+) $ and $ \tau\geq0 $
	\begin{align}\label{eq:ssec3:LemSelfSimEq}
		\begin{split}
			\dfrac{d}{d\tau } \int_{\R_+^2} \varphi (\eta) F(\tau, d\eta) &=
			3\int_{\R_+^2} \varphi (\eta ) F(\tau, d\eta) - 2\int_{\R_+^2}  (\eta \cdot\nabla\varphi(\eta)) F(\tau, dz) 
			\\
			&\quad + \sum_{i=1}^3\dfrac{\alpha_i}{2} \int_{\R_+^2}\int_{\R_+^2} \K_i(\eta,\eta') \, \Delta^*_i\varphi(\tau; \eta ,\eta ') \, F(\tau,d\eta)F(\tau,d\eta').
		\end{split}
	\end{align}
	where 
	\begin{equation} \label{eq:sec3:LemSelfSimEq Delta}
		\Delta^*_i\varphi(\tau; \eta ,\eta ') = \varphi \left( z+ z' + \xi_i \frac{ e^{ - 2 \tau } }{ T_*^2 } \right) - \varphi (z') - \varphi ( z ), \quad i \in \{ 1, 2, 3 \}. 
	\end{equation}
\end{lem}
\begin{proof}
	The assertion follows from \eqref{eq:sec1:ContWeakForm} and the definition of $ F $.
\end{proof}

The study of the moments of the solution $ f $ now yields the following corollary for the moments of $ F $.
\begin{cor}\label{cor:MomentEst}
	Consider $ \alpha\in \R_{+}^3 $, $ f_0\in \Meas_{4,+}(\S) $ and assume $ T_*=T_{\alpha,*}<\infty $. Let $ F $ be defined as in Definition \ref{def:SelfSimVar}. Then, we have for some constant $ C>0 $ and all $ \tau\geq0 $
	\begin{align*}
		\norm{\M^1(F(\tau))} &\leq Ce^{\tau},
		\\
		\norm{\dfrac{d}{d\tau}\M^2(F(\tau)) } \leq Ce^{-\tau}, \quad \norm{\M^2(F(\tau)) - \theta\otimes\theta} &\leq Ce^{-\tau},
		\\
		\norm{\M^3(F(\tau)) - c_0 \theta\otimes\theta\otimes\theta} &\leq Ce^{-\tau},
		\\
		\norm{\M^4(F(\tau))} &\leq C.
	\end{align*}
	Here, $ \theta\in \R^2_+ $ is given as in Proposition \ref{pro:SecondMomentsLocalization} and $ c_0>0 $ is defined in Proposition \ref{pro:ThirdMomentsLocalization}.
\end{cor}
\begin{proof}
	The proof follows from the following scaling principle relating the moments of $ f(t) $ and $ F(\tau) $
	\begin{align*}
		\M^k(F(\tau)) = (T_*-t(\tau))^{2k-3} \M^k(f(t(\tau))), \quad k\in \N_0.
	\end{align*}
	We can then employ the results of Propositions \ref{pro:1stOderMomentEquations}, \ref{pro:SecondMomentsLocalization}, \ref{pro:ThirdMomentsLocalization} and \ref{pro:ForthMomentsBound}. To this end, observe that $ T_*-t(\tau) = T_*e^{-\tau} $. For the estimate of the derivative of the second moments we use Proposition \ref{pro:SecondMomentsLocalization} yielding
	\begin{align*}
		\norm{\dfrac{d}{d\tau} \M^2(F(\tau))} = \norm{\dfrac{d}{dt}\left[ (T_*-t)\M^2(f(t)) \right]\mid_{t=t(\tau)}} \left| \dfrac{dt(\tau)}{d\tau} \right| \leq Ce^{-\tau}.
	\end{align*} 
	This concludes the proof.
\end{proof}

The study of the moments yield in fact the localization of the whole distribution function.
\begin{pro}\label{pro:Localization}
	Let $ \alpha\in \R_{+}^3 $, $ f_0\in \Meas_{4,+}(\S) $ and assume $ T_*=T_{\alpha,*}<\infty $. Consider $ F $ as in Definition \ref{def:SelfSimVar} and $\theta $ as in Proposition \ref{cor:MomentEst}. Then, we have the following estimate
	\begin{align*}
		\int_{\R_+^2} |\eta|^p \norm{\dfrac{\eta}{|\eta|}-\dfrac{\theta}{|\theta|}}^2\, F(\tau,d\eta) \leq Ce^{-\tau}, \quad p\in \{2,3\}
	\end{align*}
	for all $ \tau\geq0 $ and some constant $ C>0 $.
\end{pro}
\begin{proof}
	Observe that for $p \in \{ 2, 3 \} $ 
	\begin{align*}
		|\eta|^p \norm{\dfrac{\eta}{|\eta|}-\dfrac{\theta}{|\theta|}}^2 =     |\eta|^{p-2} \norm{\eta}^2 + \dfrac{ |\eta|^p }{|\theta|^2} \norm{\theta}^2 - 2 \dfrac{ |\eta|^{p-1} }{ |\theta | } \eta \cdot \theta.  
	\end{align*}
	For $ p=2 $ we then have
	 \begin{align*}
		\int_{\R_+^2} |\eta|^2 \norm{\dfrac{\eta}{|\eta|}-\dfrac{\theta}{|\theta|}}^2\, F(\tau,d\eta) &= 	\int_{\R_+^2} \left( \norm{\eta}^2 + \dfrac{ |\eta|^2 }{|\theta|^2} \norm{\theta}^2 -  \dfrac{ 2 |\eta| }{ |\theta | } \eta \cdot \theta \right) \, F(\tau,d\eta) 
		\\
		&=\sum_{k=1}^2 \M^2_{kk}(\tau) + \dfrac{\norm{\theta}^2}{|\theta|^2}\sum_{j,k=1}^2 \M^2_{kj}(\tau) - \dfrac{2}{|\theta|}\sum_{ j,k=1}^2 \M^2_{kj}(\tau) \theta_j.
	\end{align*}
	We now use Corollary \ref{cor:MomentEst}. Observe that replacing $ \M_{kj}(\tau) $ by $ \theta_k\theta_j $ in the preceding formula yields zero. In particular, we have
	\begin{align*}
		&\int_{\R_+^2} |\eta|^2 \norm{\dfrac{\eta}{|\eta|}-\dfrac{\theta}{|\theta|}}^2\, F(\tau,d\eta) 
		\\
		&\quad = \sum_{k=1}^2 \left( \M^2_{kk}(\tau)-\theta_k^2\right)  + \dfrac{\norm{\theta}^2}{|\theta|^2}\sum_{j,k=1}^2 \left( \M^2_{kj}(\tau)-\theta_k\theta_j\right)  - \dfrac{2}{|\theta|}\sum_{ j,k=1}^2 \left( \M^2_{kj}(\tau)-\theta_k\theta_j\right)  \theta_j.
		\\
		&\quad \leq Ce^{-\tau}.
	\end{align*}
	In the last inequality we used the estimate in Corollary \ref{cor:MomentEst}. The proof for $ p=3 $ is analogous using the estimate for the third order moments in Corollary \ref{cor:MomentEst}.
\end{proof}

\subsection{Convergence towards self-similar solution}\label{subsec:SelfSimilarConvergence}
In the previous subsection we prove that the distribution in self-similar variables $ F(\tau) $ concentrates on the line $ \{\lambda\theta, \lambda\geq0\} $. In this subsection we study the long-time asymptotics of the profile on the localization line itself, see (iii) in Theorem \ref{thm:IntrodLocaliSelfSim}. To this end, it is convenient to introduce a variant of polar coordinates as follows. Let us denote $ \Delta = \{ \eta \in \R^2 : |\eta | =1 \} $. We define the mapping
\begin{align*}
	T : (0,\infty)^2 \to (0,\infty)\times \Delta : \eta \mapsto (r,\omega) = \left( |\eta|, \dfrac{\eta}{|\eta|} \right).
\end{align*}
Note that $ T $ is a diffeomorphism with Jacobian $ |\det DT^{-1}|=r $. Furthermore, let us also define the projection
\begin{align*}
	P : (0,\infty)\times \Delta \to (0,\infty): (r,\omega)\to r.
\end{align*}
We then define the following distributions.
\begin{defi}\label{def:PolarCoordDistr}
	Let $ \alpha\in \R_{+}^3 $, $ f_0\in \Meas_{4,+}(\S) $ and assume $ T_{\alpha,*}<\infty $. Consider $ F $ as in Definition \ref{def:SelfSimVar}. We then define 
	\begin{align*}
		G(\tau) = T_{\#} F(\tau) \in \Meas_+(\R_+\times \Delta), \quad g(\tau) = \dfrac{1}{Z(\tau)}P_{\#}G(\tau) \in \Meas_+(\R_+),
	\end{align*}
	where $ Z(\tau)=\sum_{j,k=1}^2\M^2_{jk}(F(\tau)) $.
\end{defi}
The measure $ g(\tau) $  can be regarded as a measure on the line, indeed the $ \omega $-component has been integrated. An advantage of choosing the polar coordinates $ r=|\eta| $ and $\omega= \dfrac{\eta }{|\eta |}$ as the radial variable, is the fact that $ g $ will satisfy a one-dimensional coagulation equation. Indeed, the coagulation rule is $ (\eta,\eta')\mapsto \eta+\eta'+e^{-2\tau}\xi_i/T_*^2 $, see \eqref{eq:sec3:LemSelfSimEq Delta}, so that the coagulation rule for the radial variable is
\begin{align*}
	(r,r')\mapsto r+r'+\dfrac{e^{-2\tau}}{T_*^2}\left( \xi_{i,1}+\xi_{i,2} \right).
\end{align*}
When $ \tau\to \infty $ this is exactly the standard coagulation rule in one dimension.

As mentioned in the introduction the self-similar asymptotics is expected to be valid for the weighted measure $ |\eta|^2F(\tau) $. Thus, on the localization line we are interested in the measure $ r^2g(\tau) $, where the normalization factor $ Z(\tau) $ has been introduced in order to ensure that $ r^2g(\tau) $ is a probability measure:
\begin{align*}
	\int_{\R_+}r^2 g(\tau, dr) = \dfrac{1}{Z(\tau)}\int_{\R_+\times \Delta}r^2 G(\tau, dr,d\omega) = \dfrac{1}{Z(\tau)}\int_{\R^2_+}|\eta|^2F(\tau,d\eta) =\dfrac{1}{Z(\tau)}\sum_{j,k=1}^2\M^2_{jk}(F(\tau)) =1.  
\end{align*}
Also observe that by Corollary \ref{cor:MomentEst} we have $ Z(\tau) \to |\theta|^2 $. Let us mention that $ Z(\tau)>0 $ for all times. Otherwise $ F\equiv 0 $ by the uniqueness of the solution to the coagulation equation.

The main result in this subsection is then the following theorem.
\begin{thm}\label{thm:ConvSelfSim}
	Let $ \alpha\in \R_{+}^3 $, $ f_0\in \Meas_{4,+}(\S) $ and assume $ T_*=T_{\alpha,*}<\infty $. Consider $ g $ as in Definition \ref{def:PolarCoordDistr} and the probability measure on $ \R_+ $
	\begin{align*}
		g_\infty(dr) = \dfrac{1}{\sqrt{2\pi K_0r}} e^{-r/2K_0}dr,
	\end{align*}
	where $ K_0 = c_0|\theta| $. Here, $ \theta $ and $ c_0 $ is given in Corollary \ref{cor:MomentEst}. Then, we have
	\begin{align*}
		\lim_{ \tau \to \infty } r^2g(\tau) = g_\infty,
	\end{align*}
	with respect to the weak convergence of measures. 
\end{thm}
The proof of this result will be achieved via the following steps.
\begin{enumerate}[(i)]
	\item We first derive a coagulation equation satisfied by $ g $. This equation will be one-dimensional. However, the defect of $ G $ being not localized for $tr < T_*$ will introduce forcing terms, that can be shown to decay as $ \tau\to \infty $.
	
	\item We then study the desingularized Laplace transform of $ g $, given by
	\begin{align}\label{eq:sec3:DesingLapTrans}
		\hat{g}(\tau, \rho) = \int_{\R_+}r(1-e^{-r\rho}) \, g(\tau, dr).
	\end{align}
	Using (i) we can derive an equation satisfied by $ \hat{g} $. This equation has the structure of a Burger's equation with forcing. Again the forcing term depends on $ G $ and can be shown to decay as $ \tau\to \infty $.
	
	\item Finally, using the moment bounds in Corollary \ref{cor:MomentEst} we obtain a priori knowledge on the regularity of $ \hat{g} $. In fact, we have $ \hat{g}\in C^1_b(\R_+;C_b(\R_+))\cap C_b(\R_{+}; C^3_b(\R_+)) $. This and the study of the characteristics to the Burger's equation allows to prove for all $ \rho\geq0 $
	\begin{align*}
		\lim_{\tau \to \infty} \partial_\rho\hat{g}(\tau,\rho) = \dfrac{1}{\sqrt{1+2K_0\rho}}.
	\end{align*}
	Observe that $ \partial_\rho\hat{g}(\tau,\rho) $ is the standard Laplace transform of $ r^2g(\tau,dr) $, see \cite{menon2004approach}. Recall that the Laplace transform to $ g_\infty $ in Theorem \ref{thm:ConvSelfSim} is given by $ 1/\sqrt{1+2K_0\rho} $. Furthermore, recall that weak convergence in the sense of measures is equivalent to the pointwise convergence of the Laplace transform.
\end{enumerate}
Following the preceding plan we have the following lemma.
\begin{lem}\label{lem:ProjCoagEq}
	Let $ \alpha\in \R_{+}^3 $, $ f_0\in \Meas_{4,+}(\S) $ and assume $ T_*=T_{\alpha,*}<\infty $.  Consider $ g $ and $ G $ as in Definition \ref{def:SelfSimVar}. Let $ \theta\in \R^2_+ $ be given as in Corollary \ref{cor:MomentEst} and define $ \omega_\theta = \theta/|\theta| $. Then, we have the following equation for all $\varphi \in C^1_b (\R_+) $
	\begin{align*}
		&\dfrac{d}{d\tau}\int_{\R_+} \varphi(r) g(\tau,dr) =
		3\int_{\R_+} \varphi (r) g(\tau,dr) - 2\int_{\R_+^2} r\partial_r\varphi(r) g(\tau, dr) - \dfrac{\dot{Z}(\tau)}{ Z(\tau) }  \int_{\R_+} \varphi (r) g(\tau, dr) 
		\\
		&\quad + \sum_{i=1}^3\dfrac{\alpha_iZ(\tau)}{2}\dfrac{\theta^\top K_i\theta}{|\theta|^2} \int_{\R_+}\int_{\R_+} rr' \, \bar{\Delta}_i\varphi(\tau;r,r ') \, g(\tau,dr) g (\tau,dr') 
		\\
		& \quad + \sum_{i=1}^3\dfrac{\alpha_i}{2Z(\tau)}\int_{\R_+ \times \Delta}\int_{\R_+ \times \Delta}  r r' \left[\K_i(\omega,\omega') - \K_i(\omega_\theta,\omega_\theta) \right] \, \bar{\Delta}_i\varphi(\tau; r,r ') \, G(\tau,dr,d\omega) G (\tau,dr',d\omega).
	\end{align*}
	Here, we defined
	\begin{align*}
		\bar{\Delta}_i \varphi (\tau,r,r') = \ind_{\{r+r'-\delta_i(\tau)\geq0\}}\left[ \varphi\left( r+r'-\delta_i(\tau) \right) -\varphi(r)-\varphi(r')\right] , \quad \delta_i(\tau)=-(\xi_{i,1}+\xi_{i,2})\dfrac{e^{-2\tau}}{T_*^2}\geq0.
	\end{align*}
\end{lem}
\begin{proof}
	The equation follows using Lemma \ref{lem:SelfSimEq} and the definition of $ g $ and $ G $. Recall that $ F(\tau) $ defined in Definition \ref{def:SelfSimVar} is supported on the set $ e^{-2\tau}\S_i $. As a consequence $ G(\tau) $ is supported on the set $ \{r+r'-\delta_i(\tau)\geq0\} $. The coagulation operator has then the form
	\begin{align*}
		\sum_{i=1}^3\dfrac{\alpha_i}{2Z(\tau)}\int_{\R_+\times\Delta} \int_{\R_+\times\Delta} rr'\K_i(\omega,\omega') \, \bar{\Delta}_i\varphi(\tau; r  ,r ') \, G(\tau,dr,d\omega) G(\tau,dr',d\omega).
	\end{align*}
	We then add and subtract the kernel $ \K_i(\omega_\theta,\omega_\theta)=\theta^\top K_i\theta/|\theta|^2 $. This yields then the last two terms in the asserted equation. 
\end{proof}

In order to estimate the contribution of the coagulation kernel containing $ G $ we make use of the following lemma, which immediately follows from Proposition \ref{pro:Localization}.

\begin{lem}\label{lem:EstLocPolarCoord}
	Let $ G $ be given as in Definition \ref{def:PolarCoordDistr}. Then, we have the estimates for all $ \tau\geq0 $
	\begin{align*}
		\int_{\R_+\times\Delta}r^p \norm{\omega-\omega_\theta}^2G(\tau, dr,d\omega)\leq Ce^{-\tau}, \quad p\in \{2,3\}.
	\end{align*}
\end{lem}

Using Lemma \ref{lem:ProjCoagEq} we now derive the equation satisfied by the desingularized Laplace transform of $ g $, see \eqref{eq:sec3:DesingLapTrans}.

\begin{lem}\label{lem:ProjEqLapTransf}
	Let $ \alpha\in \R_{+}^3 $, $ f_0\in \Meas_{4,+}(\S) $ and assume $ T_*=T_{\alpha,*}<\infty $.  Consider $ g $ and $ G $ as in Definition \ref{def:SelfSimVar}. Then, the desingularized Laplace transform $ \hat{g} \in C^1(\R_+;C_b(\R_+)) \cap C_b(\R_+;C^3_b(\R_+)) $ satisfies the equation
	\begin{align}\label{eq:sec3:ProjEqLapTransf}
		\partial_\tau \hat{g} = \hat{g} + \left(\hat{g} - 2\rho \right) \partial_\rho \hat{g} + \Rem.
	\end{align}
	Here, we defined
	\begin{align*}
		&\Rem(\tau,\rho) = -\frac{\dot{Z}(\tau)}{ Z(\tau)}\hat{g} + \left( \dfrac{Z(\tau)}{|\theta|^2}-1 \right) \hat{g}\partial_\rho\hat{g} + \sum_{i=1}^3\alpha_iZ(\tau)\dfrac{\theta^\top K_i\theta}{|\theta|^2}\int_{\R_+}\int_{\R_+} \, \left[ h_1^i+h_2^i \right]  \, g(\tau,dr) g(\tau,dr')
		\\
		&\quad + \sum_{i=1}^3 \dfrac{\alpha_i}{Z(\tau)} \int_{\R_+\times\Delta}\int_{\R_+ \times \Delta} \left[\K_i(\omega,\omega') - \K_i(\omega_\theta  ,\omega_\theta) \right] \, \left[ h_1^i+h_2^i+h_3 \right] \, G(\tau,dr,d\omega) G(\tau,dr',d\omega'),
	\end{align*}
	as well as
	\begin{align*}
		h_1^i(\tau,\rho,r,r ') &=  \ind_{\{r+r'-\delta_i(\tau)\geq0\}}  r'r^2e^{-\rho (r+ r')} \left(1-e^{\delta_i(\tau)\rho}\right), \quad i\in \{1,2,3\},
		\\
		h_2^i(\tau,\rho,r,r ') &=  -\ind_{\{r+r'-\delta_i(\tau)\geq0\}} \dfrac{1}{2}rr'\delta_i(\tau) \left(1-e^{-(r+r'-\delta_i(\tau))\rho}\right), \quad i\in \{1,2,3\},
		\\
		h_3(\tau,\rho,r,r ') &= r'r^2e^{-r\rho} \left(1- e^{-r'\rho}\right).
	\end{align*}
\end{lem}
\begin{proof}
	First of all the regularity $ \hat{g} \in C^1(\R_+;C_b(\R_+)) \cap C_b(\R_+;C^3_b(\R_+)) $ follows from the moment bounds given in Corollary \ref{cor:MomentEst} together with the definition of $ g $ and its desingularized Laplace transform. The asserted equation can be derived using Lemma \ref{lem:ProjCoagEq} and the test function $ \varphi(r) = r(1-e^{-r\rho}) $ for given $ \rho\geq0 $. To this end, we observe that 
	\begin{align*}
		\bar{\Delta}_{i}\varphi(\tau,r,r') &= \ind_{\{r+r'-\delta_i(\tau)\geq0\}}\left[ r' r^2 e^{ - \rho r } (1- e^{- r'\rho}) + r'r^2e^{-\rho (r+ r')} (1-e^{\delta_i(\tau)\rho}) \right]  +
		\\
		&\quad + \ind_{\{r+r'-\delta_i(\tau)\geq0\}} \left[ r (r')^2 e^{ - \rho r' } (1- e^{- r\rho}) + r (r')^2e^{-\rho (r+ r')} (1-e^{\delta_i(\tau)\rho}) \right] 
		\\
		&\quad - \ind_{\{r+r'-\delta_i(\tau)\geq0\}} \left[ rr'\delta_i(\tau) (1-e^{-(r+r'-\delta_i(\tau))\rho}) \right] 
		\\
		&= \ind_{\{r+r'-\delta_i(\tau)\geq0\}} \sum_{j=1}^2 \left[ h_j^i(\tau,r,r')+ h_j^i(\tau,r',r) \right] +\ind_{\{r+r'-\delta_i(\tau)\geq0\}}\left[ h_3(\tau,r,r')+ h_3(\tau,r',r) \right].
	\end{align*}
	We can reduce this in the integral by making use of the symmetry $ r\leftrightarrow r' $. This yields
	\begin{align*}
		&\sum_{i=1}^3\dfrac{\alpha_iZ(\tau)}{2}\dfrac{\theta^\top K_i\theta}{|\theta|^2} \int_{\R_+}\int_{\R_+} rr' \, \bar{\Delta}_i\varphi(\tau;r,r ') \, g(\tau,dr) g (\tau,dr') 
		\\
		&\quad= \dfrac{Z(\tau)}{|\theta|^2}\hat{g}\partial_\rho\hat{g} + \sum_{i=1}^3\alpha_iZ(\tau)\dfrac{\theta^\top K_i\theta}{|\theta|^2}\int_{\R_+}\int_{\R_+} \, \left[h_1^i+h_2^i\right]  \, g(\tau,dr) g(\tau,dr').
	\end{align*}
	Here, we used for the first term containing $ h_3 $ that $ \sum_{i=1}^3\alpha_i\theta^\top K_i\theta =1 $, see Proposition \ref{pro:SecondMomentsLocalization}. Furthermore, we have
	\begin{align*}
		3\int_{\R_+} \varphi (r) g(\tau,dr) - 2\int_{\R_+^2} r\partial_r\varphi(r) g(\tau, dr) - \dfrac{\dot{Z}(\tau)}{ Z(\tau) }  \int_{\R_+} \varphi (r) g(\tau, dr) = \hat{g} - 2\rho\partial_\rho\hat{g} - \dfrac{\dot{Z}(\tau)}{Z(\tau)}\hat{g}.
	\end{align*}
	Putting all terms together yields the equation \eqref{eq:sec3:ProjEqLapTransf} and the form of $ \Rem $.
\end{proof}

Next we provide estimates on the reminder term $ \Rem $ and its derivative $ \partial_\rho\Rem $.
\begin{lem}\label{lem:EstRem1}
	Under the assumptions of Theorem \ref{thm:ConvSelfSim} we have
	\begin{align*}
		|\Rem(\tau,\rho)| \leq C\rho e^{-\tau/2},
	\end{align*}
	for all $ \tau\geq0 $, $ \rho\geq0 $ and some constant $ C>0 $.
\end{lem}
\begin{proof}
	We first consider the following splitting of the reminder $ \Rem = \sum_{j=1}^6\Rem_k $,
	where we have
	\begin{align*}
		\Rem_1(\tau,\rho) &= -\frac{\dot{Z}(\tau)}{ Z(\tau)}\hat{g} + \left( \dfrac{Z(\tau)}{|\theta|^2}-1 \right) \hat{g}\partial_\rho\hat{g},
		\\
		\Rem_2(\tau,\rho) &= \sum_{i=1}^3\alpha_iZ(\tau)\dfrac{\theta^\top K_i\theta}{|\theta|^2}\int_{\R_+}\int_{\R_+} \, h_1^i(\tau,\rho, r,r ') \, g(\tau,dr) g(\tau,dr'),
		\\
		\Rem_3(\tau,\rho) &=\sum_{i=1}^3\alpha_iZ(\tau)\dfrac{\theta^\top K_i\theta}{|\theta|^2}\int_{\R_+}\int_{\R_+} \, h_2^i(\tau,\rho,r,r ') \, g(\tau,dr) g(\tau,dr'),
		\\
		\Rem_4(\tau,\rho) &=\sum_{i=1}^3 \dfrac{\alpha_i}{Z(\tau)}\int_{\R_+\times\Delta}\int_{\R_+ \times \Delta} \left[\K_i(\omega,\omega') - \K_i(\omega_\theta  ,\omega_\theta) \right] \, h_1^i(\tau,\rho,r,r ')  \, G(\tau,dr,d\omega) G(\tau,dr',d\omega'),
		\\
		\Rem_5(\tau,\rho) &=\sum_{i=1}^3 \dfrac{\alpha_i}{Z(\tau)}\int_{\R_+\times\Delta}\int_{\R_+ \times \Delta} \left[\K_i(\omega,\omega') - \K_i(\omega_\theta  ,\omega_\theta) \right] \, h_2^i(\tau,\rho,r,r ')  \, G(\tau,dr,d\omega) G(\tau,dr',d\omega'),
		\\
		\Rem_6(\tau,\rho) &=\sum_{i=1}^3 \dfrac{\alpha_i}{Z(\tau)}\int_{\R_+\times\Delta}\int_{\R_+ \times \Delta} \left[\K_i(\omega,\omega') - \K_i(\omega_\theta  ,\omega_\theta) \right] \, h_3(\tau,\rho,r,r ')  \, G(\tau,dr,d\omega) G(\tau,dr',d\omega').
	\end{align*}
	We now estimate each term separately.
	
	First of all, we observe for $ \Rem_1 $ that $ \hat{g}(\tau,\rho) \leq C\rho $, $ \norm{\partial_\rho\hat{g}}\leq C $ and hence
	\begin{align*}
		|\Rem_1| \leq C\rho \left| \dfrac{\dot{Z}(\tau)}{Z(\tau)} \right| + \rho\left| \dfrac{Z(\tau)}{|\theta|^2}-1 \right| \leq  C\rho e^{-\tau}.
	\end{align*}
	Here, we used the estimates in Corollary \eqref{cor:MomentEst} for $ Z(\tau)=\sum_{j,k=1}^2\M^2_{jk}(F(\tau)) $.
	
	For $ \Rem_2 $ we make use of
	\begin{align*}
		|h_1^i(\tau,\rho,r,r')| &= \ind_{\{r+r'-\delta_i(\tau)\geq0\}}  r'r^2e^{-\rho (r+ r')} \left|e^{\delta_i(\tau)\rho}-1\right| \leq \ind_{\{r+r'-\delta_i(\tau)\geq0\}}  r'r^2e^{-\rho (r+ r')} \, \rho\delta_i(\tau) e^{\delta_i(\tau)\rho}
		\\
		&\leq r'r^2\delta_i(\tau) \rho.
	\end{align*}
	This yields then
	\begin{align*}
		|\Rem_2(\tau,\rho)|\leq C\rho\M^1(g(\tau)) \M^2(g(\tau))\sum_i \delta_i(\tau) \leq C\rho e^{-\tau}.
	\end{align*}
	Here, we used the fact that $ \M^1(g(\tau))\leq Ce^{\tau} $ and $ \M^2(g(\tau))=1 $ as follows from Corollary \ref{cor:MomentEst} and the definition of $ g(\tau) $.
	
	Considering $ \Rem_3 $ we first observe that
	\begin{align*}
		|h_2^i(\tau,\rho,r,r')| \leq \ind_{\{r+r'-\delta_i(\tau)\geq0\}} \dfrac{1}{2}rr'\delta_i(\tau) \left(r+r'-\delta_i(\tau)\right)\rho.
	\end{align*}
	This yields
	\begin{align*}
		|\Rem_3(\tau,\rho)| \leq C\rho \sum_i\left[ \delta_i(\tau)\M^1(g(\tau))+\delta_i(\tau)^2\M^1(g(\tau))^2  \right]\leq C\rho e^{-\tau}. 
	\end{align*}

	We now turn to $ \Rem_4 $ and use the previous estimate for $ h_1^i $. This yields
	\begin{align*}
		|\Rem_4(\tau,\rho)| &\leq C\rho \sum_i\delta_i(\tau) \int_{\R_+\times\Delta}\int_{\R_+ \times \Delta} r(r')^2 \, G(\tau,dr,d\omega) G(\tau,dr',d\omega')
		\\
		&\leq C\rho \sum_i\delta_i(\tau)\M^1(g(\tau)) \leq C\rho e^{-\tau}.
	\end{align*}
	
	For $ \Rem_5 $ we use the previous estimate for $ h_2^i $ to get
	\begin{align*}
		|\Rem_5(\tau,\rho)| &\leq C\rho \sum_i\delta_i(\tau) \int_{\R_+\times\Delta}\int_{\R_+ \times \Delta} \left|rr'\delta_i(\tau) \left(r+r'-\delta_i(\tau)\right)\right| \, G(\tau,dr,d\omega) G(\tau,dr',d\omega')
		\\
		&\leq C\rho\sum_i\left[ \delta_i(\tau)\M^1(g(\tau))+\delta_i(\tau)^2\M^1(g(\tau))^2  \right]\leq C\rho e^{-\tau}. 
	\end{align*}

	Finally, for $ \Rem_6 $ we first observe that
	\begin{align*}
		|h_3(\tau,\rho,r,r')|\leq C (rr')^2 \rho
	\end{align*}
	and 
	\begin{align*}
		\left| \K_i(\omega,\omega') - \K_i(\omega_\theta  ,\omega_\theta)\right| \leq C \norm{\omega-\omega_\theta}+C \norm{\omega'-\omega_\theta}.
	\end{align*}
	We thus obtain, making use of the symmetry $ (r,\omega)\leftrightarrow(r',\omega') $,
	\begin{align*}
		|\Rem_6(\tau,\rho)|&\leq C\rho\int_{\R_+\times\Delta}\int_{\R_+ \times \Delta} r^2(r')^2 \norm{\omega-\omega_\theta}  \, G(\tau,dr,d\omega) G(\tau,dr',d\omega')
		\\
		&\leq C\rho\left( \int_{\R_+\times\Delta}r^2G(\tau,dr,d\omega) \right)^{3/2}\left( \int_{\R_+\times\Delta}r^2\norm{\omega-\omega_\theta}  \, G(\tau,dr,d\omega) \right)^{1/2} 
		\\
		&\leq C\rho e^{-\tau/2}.
	\end{align*}
	In the last inequality we used Lemma \ref{lem:EstLocPolarCoord}. This concludes the proof.
\end{proof}

\begin{lem}\label{lem:EstRem2}
	Under the assumptions of Theorem \ref{thm:ConvSelfSim} we have
	\begin{align*}
		|\partial_\rho\Rem(\tau,\rho)| \leq C(1+\rho) e^{-\tau/2},
	\end{align*}
	for all $ \tau\geq0 $, $ \rho\geq0 $ and some constant $ C>0 $.
\end{lem}
\begin{proof}
	We use the same splitting $ \Rem = \sum_{j=1}^6\Rem_k $ as in the proof of Lemma \ref{lem:EstRem1}. We estimate term by term.
	
	First of all, we have
	\begin{align*}
		\partial_\rho\Rem_1 = -\frac{\dot{Z}(\tau)}{ Z(\tau)}\partial_\rho\hat{g} + \left( \dfrac{Z(\tau)}{|\theta|^2}-1 \right) \left( (\partial_\rho\hat{g})^2+\hat{g}\partial_\rho^2\hat{g} \right).
	\end{align*}
	Hence we obtain
	\begin{align*}
		|\partial_\rho\Rem_1| \leq Ce^{-\tau}.
	\end{align*}
	
	For the remaining terms we first observe
	\begin{align*}
		\partial_\rho h_1^i &=  -\ind_{\{r+r'-\delta_i\geq0\}}  \left[ r'r^2(r+ r')e^{-\rho (r+ r')} \left(1-e^{\delta_i\rho}\right)+r'r^2\delta_ie^{-\rho (r+ r'-\delta_i)}  \right],
		\\
		\partial_\rho h_2^i &=  \ind_{\{r+r'-\delta_i\geq0\}} \dfrac{1}{2}rr'\delta_i\left(r+r'-\delta_i \right)  e^{-(r+r'-\delta_i)\rho},
		\\
		\partial_\rho h_3 &= -r'r^3e^{-r\rho} \left(1- e^{-r'\rho}\right)-(r'r)^2e^{-(r+r')\rho}.
	\end{align*}
	We then obtain the estimates
	\begin{align*}
		|\partial_\rho h_1^i| &\leq C\rho\delta_i \left( r'r^3 + (r'r)^2 + r'r^2 \right),
		\\
		|\partial_\rho h_2^i| &\leq \dfrac{1}{2}rr'\delta_i \left(r+r'-\delta_i \right),
		\\
		|\partial_\rho h_3| &\leq \rho(r')^2r^3 + (r'r)^2. 
	\end{align*}
	This yields then 
	\begin{align*}
		|\Rem_2| &\leq C\rho \sum_i\delta_i \left( \M^1(g(\tau))\M^3(g(\tau)) + \M^2(g(\tau))^2 + \M^1(g(\tau))\M^2(g(\tau))  \right)  \leq C\rho e^{-\tau},
		\\
		|\Rem_3| &\leq C\sum_i\delta_i  \M^1(g(\tau))\M^2(g(\tau))+ C\sum_i\delta_i^2  \M^1(g(\tau))^2\leq Ce^{-\tau}.
	\end{align*}
	Here, we used the estimates on the moments, see Corollary \ref{cor:MomentEst} together with the definition of $ g $. The above estimates yield in the same way
	\begin{align*}
		|\Rem_4| &\leq C\rho \sum_i\delta_i \left( \M^1(g(\tau))\M^3(g(\tau)) + \M^2(g(\tau))^2 + \M^1(g(\tau))\M^2(g(\tau))  \right)  \leq C\rho e^{-\tau},
		\\
		|\Rem_5| &\leq C \sum_i\delta_i  \M^1(g(\tau))\M^2(g(\tau))+ C\sum_i \delta_i^2  \M^1(g(\tau))^2\leq Ce^{-\tau}.
	\end{align*}
	Finally, we have
	\begin{align*}
		|\Rem_6| &\leq C\int_{\R_+\times\Delta}\int_{\R_+ \times \Delta} \left( \rho(r')^2r^3 + (r'r)^2\right) \left(  \norm{\omega-\omega_\theta}+\norm{\omega'-\omega_\theta}\right)   \, G(\tau,dr,d\omega) G(\tau,dr',d\omega')
		\\
		&\leq C\rho\int_{\R_+\times\Delta}(r^2+r^3)\norm{\omega-\omega_\theta}G(\tau,dr,d\omega)+C\int_{\R_+\times\Delta}r^2\norm{\omega-\omega_\theta}G(\tau,dr,d\omega)
		\\
		&\leq C\rho e^{-\tau/2}+Ce^{-\tau/2}.
	\end{align*}
	Here, we used Lemma \ref{lem:EstLocPolarCoord}. Putting all estimates together concludes the proof.
\end{proof}

We now give the proof of Theorem \ref{thm:ConvSelfSim}.
\begin{proof}[Proof of Theorm \ref{thm:ConvSelfSim}]
	As mentioned before it suffices to show that $ \partial_\rho\hat{g}(\tau,\rho) $ converges pointwise to $ 1/\sqrt{1+2K_0\rho} $. In fact, below we prove the following. There are constants $ C,\, \kappa>0 $  such that for all $ M\geq1 $ and all $ \tau\geq0 $ we have
	\begin{align*}
		\sup_{\rho\in[0,M]}\left| \partial_\rho\hat{g}(\tau,\rho) - \dfrac{1}{\sqrt{1+2K_0\rho}}\right| \leq CM^{\kappa}e^{-\tau/10}.
	\end{align*}
	Let us give a short overview of the proof. The idea is to use the equation satisfied by $ \hat{g} $ in Lemma \ref{lem:ProjEqLapTransf}. Since $ \hat{g} $ is a classical solution to this equation, $ \hat{g} $ can be quantified using the method of characteristics. In order to turn this into a perturbative analysis, we do not study the dynamics starting at $ \tau=0 $ but at $ \tau=T $ large. This ensures that the formal initial condition $ \hat{g}(T,\rho) $ is arsing from a distribution that is almost localized on the concentration line $ \{\lambda\theta \, :\, \lambda\geq0\} $ with controlled values of the third moments. On the level of the desingularized Laplace transform this means that
	\begin{align*}
		\partial^2_\rho \hat{g}(T,0) = \dfrac{1}{Z(\tau)}\int_{\R_+ \times \Delta} r^3\, G(T,dr,d\omega) = \dfrac{1}{Z(\tau)}\int_{\R^2_+}|\eta|^3 \, F(T,d\eta) \approx c_0|\theta| = K_0.
	\end{align*} 
	Recall that by Corollary \ref{cor:MomentEst} we have $ \int_{\R^2_+}|\eta|^3 \, F(T,d\eta) \to c_0|\theta|^3 $ and $ Z(\tau)\to |\theta|^2 $. Observe that the value $ K_0 $ characterizes the limiting self-similar distribution, see the formula for $ g_\infty $ in Theorem \ref{thm:ConvSelfSim}. Note that $ \partial_\rho\hat{g}(T,0)=1 $ for all $ T\geq0 $ by definition of $ g $. 
	
	In addition, using the starting point $ \tau=T $ also ensures that the reminder $ \Rem(T,\rho) $ is small, i.e.\ of order $ e^{-T/2} $. This allows to prove that the characteristics of the Burger's type equation and hence also the solution can be controlled on a time interval of size $ c T $, $ c>0 $. And we obtain for all $ \tau \leq c T $
	\begin{align*}
		\partial_\rho\hat{g}(T+\tau,\rho) - \dfrac{1}{\sqrt{1+2K_0\rho}} = \mathcal{O}\left( e^{-\tau} \right) + \mathcal{O}\left( e^{-T/10} \right).
	\end{align*}
	Setting $ \tau = cT $ will then provide the above estimate.
	
	We now provide the full details and to this end we divide the proof into several steps.
	
	\textit{Step 1. Setup.} Let us fix $ M\geq1 $ and $ T\geq1 $. We define 
	\begin{align*}
		H_T(\tau,\rho) = \hat{g}(T+\tau,\rho), \quad \Rem_T(\tau,\rho) = \Rem(T+\tau,\rho).
	\end{align*}
	We know that $ H_T $ satisfies the equation
	\begin{align}\label{eq:sec3:ProofConvEq}
		\partial_\tau H_T = H_T  +  \left( H_T - 2 \rho \right) \partial_\rho H_T  + \Rem_T.
	\end{align}
	We also have by Taylor's theorem
	\begin{align*}
		H_T(0,\rho) = \rho - \dfrac{1}{2}K_0\rho^2 + h_T(\rho).
	\end{align*}
	Here, we defined the function
	\begin{align}\label{eq:sec3:ProofConvInitialProfil}
		h_T(\rho) = \dfrac{1}{2}\left( \partial_\rho^2H_T(0,0)-K_0 \right) \rho^2 + \dfrac{\rho^3}{2}\int_0^1 \partial_\rho^3H_T(0,s\rho)(1-s)^2\, ds.
	\end{align} 
	This satisfies the estimates
	\begin{align}\label{eq:sec3:ProofConvInitialEst}
		|h_T(\rho)| \leq Ce^{-T}\rho^2+C\rho^3, \quad |\partial_\rho h_T(\rho)| \leq Ce^{-T}\rho+C\rho^2.
	\end{align}

	\textit{Step 2. Characteristics.} The characteristics of the equation \eqref{eq:sec3:ProofConvEq} are defined by the following ODE
	\begin{align}\label{eq:sec3:ProofConvCharac}
		\dfrac{d}{d\tau}P_T(\tau,\rho) = 2P_T(\tau,\rho) - H_T(\tau, P_T(\tau,\rho)), \quad P_T(0,\rho)=\rho.
	\end{align}
	Since the right-hand side is a continuously differentiable function of $ P_T(\tau,\rho) $, the function $ P_T(\tau,\rho) $ is well-defined and differentiable in both arguments. Since $ 0\leq \partial_\rho H_T(\tau,\rho)\leq1 $ we have $ 0\leq H_T(\tau,\rho)\leq \rho $. This implies
	\begin{align}\label{eq:sec3:ProofConvCharEst1}
		e^\tau\rho \leq P_T(\tau, \rho)\leq e^{2\tau} \rho
	\end{align}
	for all $ \tau, \, \rho \geq0 $. We differentiate equation \eqref{eq:sec3:ProofConvCharac} and obtain that
	\begin{align*}
		\dfrac{d}{d\tau}\partial_\rho P_T(\tau,\rho) = 2\partial_\rho P_T(\tau,\rho)-\partial_\rho H_T(\tau,P_T(\tau,\rho))\partial_\rho P_T(\tau,\rho). 
	\end{align*}
	Notice that since $\partial_\rho H_T(\tau,P_T(\tau,\rho)) = \partial_\rho \hat{g} (\tau+ T ,P_T(\tau,\rho)) $, we have that $\partial_\rho H_T(\tau,P_T(\tau,\rho))  \in (0,1) $, hence 
	\begin{align}\label{eq:sec3:ProofConvCharEst2}
		e^\tau\leq \partial_\rho P_T(\tau,\rho) \leq e^{2\tau}.
	\end{align}
	The equation \eqref{eq:sec3:ProofConvEq} implies 
	\begin{align*}
		\dfrac{d}{d\tau} \left[ H_T(\tau,P_T(\tau,\rho)) \right]  = H_T(\tau,P_T(\tau,\rho))+\Rem_T\left( \tau, P_T(\tau,\rho) \right). 
	\end{align*}
	Hence, we obtain
	\begin{align}\label{eq:sec3:ProofConvProfil}
		H_T(\tau,P_T(\tau,\rho)) = e^{\tau} H_T(0,\rho) + e^\tau\int_0^\tau e^{-\sigma}\Rem_T\left( \sigma, P_T(\sigma,\rho) \right) \, d\sigma.
	\end{align}
	From \eqref{eq:sec3:ProofConvCharac} we obtain
	\begin{align*}
		P_T(\tau,\rho) &= e^{2\tau}\left( \rho-H_T(0,\rho) \right) + e^\tau H_T(0,\rho) + e^{2\tau}\int_0^\tau\left( e^{-\sigma}-e^{-\tau} \right) e^{-\sigma}  \Rem_T\left( \sigma, P_T(\sigma,\rho) \right)\, d\sigma
		\\
		&=  \dfrac{K_0}{2}e^{2\tau}\rho^2 +e^\tau\rho + \bar{\Err}_T(\tau,\rho),
		\\
		\bar{\Err}_T(\tau,\rho) &=-e^{2\tau}h_T(\rho) - \dfrac{K_0}{2}e^\tau\rho^2+ e^\tau h_T(\rho) + e^{2\tau}\int_0^\tau\left( e^{-\sigma}-e^{-\tau} \right) e^{-\sigma}  \Rem_T\left( \sigma, P_T(\sigma,\rho) \right)\, d\sigma.
	\end{align*}
	Here, we used \eqref{eq:sec3:ProofConvInitialProfil}. The goal now is to give a asymptotic formula for inverse $ \rho\mapsto P_T(\tau,\rho) $. Note that by \eqref{eq:sec3:ProofConvCharEst2} we already know that this function is invertible on $ \R_+ $. To this end, it is convenient to define
	\begin{align*}
		Q_T(\tau,\rho) &:= P_T(\tau,e^{-\tau}\rho) = Q_0(\rho) + \Err_T(\tau, \rho),
		\\
		Q_0(\rho) &:= \dfrac{K_0}{2}\rho^2 + \rho,
		\\
		\Err_T(\tau, \rho) &:= -e^{2\tau}h_T(e^{-\tau}\rho) - \dfrac{K_0}{2}e^{-\tau}\rho^2+ e^\tau h_T(e^{-\tau}\rho) + e^{2\tau}\int_0^\tau\left( e^{-\sigma}-e^{-\tau} \right) e^{-\sigma}  \Rem_T\left( \sigma, P_T(\sigma,e^{-\tau}\rho) \right)\, d\sigma.
	\end{align*}
	Observe that $ Q_0:\R_{+}\to \R_{+} $ is bijective and
	\begin{align*}
		Q_0^{-1}(\rho) = \dfrac{2\rho}{1+\sqrt{1+2K_0\rho}}, \quad \partial_\rho Q_0^{-1}(\rho) = \dfrac{1}{\sqrt{1+2K_0\rho}}.
	\end{align*}
	
	\textit{Step 3. Estimates on $ \Err_T $ and $ Q_T $.} We now give estimates on $ \Err_T $ and $ \partial_\rho\Err_T $. For the first we make use of Lemma \ref{lem:EstRem1} and \eqref{eq:sec3:ProofConvInitialEst}, \eqref{eq:sec3:ProofConvCharEst1} to get
	\begin{align*}
		|\Err_T(\tau,\rho)| &\leq C\left( e^{-T}\rho^2+e^{-\tau}\rho^3+e^{-\tau}\rho^2\right) + Ce^{2\tau}\int_0^\tau e^{-2\sigma} e^{2\sigma}e^{-\tau}e^{-T/2}e^{-\sigma/2}\rho \, d\sigma
		\\
		&\leq C\left( e^{-T}\rho^2+e^{-\tau}\rho^3+e^{-\tau}\rho^2 + e^\tau e^{-T/2} \rho \right).
	\end{align*}
	For the derivative we note that
	\begin{align*}
		\partial_\rho \Err_T(\tau,\rho) &= -e^\tau \partial_\rho h_T(e^{-\tau}\rho) - K_0e^{-\tau}\rho + \partial_\rho h_T(e^{-\tau}\rho) 
		\\
		&\quad + e^\tau \int_0^\tau \left( e^{-\sigma}-e^{-\tau} \right)e^{-\sigma}\partial_{\rho} \Rem_T(\sigma,P(\sigma,e^{-\tau}\rho)) \partial_{\rho}P_T(\sigma,e^{-\tau}\rho)\, d\sigma.
	\end{align*}
	We then use Lemma \ref{lem:EstRem2} and \eqref{eq:sec3:ProofConvInitialEst}, \eqref{eq:sec3:ProofConvCharEst2}
	\begin{align*}
		|\partial_\rho \Err_T(\tau,\rho)| &\leq C\left( e^{-T}\rho + e^{-\tau}\rho^2+e^{-\tau}\rho \right) + Ce^\tau\int_0^\tau e^{-2\sigma}\left( 1+e^{2\sigma}e^{-\tau}\rho \right) e^{-T/2}e^{-\sigma/2}e^{2\sigma}\, d\sigma  
		\\
		&\leq C\left( e^{-T}\rho + e^{-\tau}\rho^2+e^{-\tau}\rho + e^{\tau}e^{-T/2}+e^{3\tau/2}e^{-T/2}\rho \right).
	\end{align*}
	In particular, for $ \tau \in [0,T/8] $ we obtain for all $ \rho\geq0 $
	\begin{align}\label{eq:sec3:ProofConvErrEst}
		|\Err_T(\tau,\rho)|\leq C\left( 1+\rho^2\right) \rho e^{-T/4}, \quad |\partial_{\rho}\Err_T(\tau,\rho)|\leq C\left( 1+\rho^2\right) e^{-T/4}.
	\end{align}
	This implies in particular
	\begin{align}\label{eq:sec3:ProofConvEstChar}
		|Q_T(\tau,\rho)| \leq C(1+\rho^3), \quad |\partial_{\rho}Q_T(\tau,\rho)| \leq C(1+\rho^2)
	\end{align}
	
	\textit{Step 4. Inverting characteristics.} We now study the inverse of $ Q_T(\tau,\rho) = Q_0(\rho) + \Err_T(\tau,\rho) $ for $ \tau \in [0,T/8] $ and $ \rho\in [0,M] $. To this end, we make use of the estimates in Step 3, in particular $ T $ will be chosen larger than $ T_0(M) $. We show below that in this case
	\begin{align}\label{eq:sec3:ProofConvEstInvChar}
		\begin{split}
			\sup_{\rho\in[0,M]}\sup_{ \tau \in [0,T/8]} \left| Q_T^{-1}(\tau,\rho) - Q_0^{-1}(\tau,\rho) \right| &\leq CM^{3/2}e^{-T/4},
			\\
			\sup_{\rho\in[0,M]}\sup_{ \tau \in [0,T/8]} \left| \partial_{\rho}Q_T^{-1}(\tau,\rho) - \partial_{\rho}Q_0^{-1}(\tau,\rho) \right| &\leq CM^{3/2}e^{-T/4}.
		\end{split}
	\end{align}
	To this end, let use define
	\begin{align*}
		U_T(\tau,\rho) = Q_T\left(\tau,Q_0^{-1}(\rho)\right) = \rho + \Err_T\left(\tau,Q_0^{-1}(\rho)\right).
	\end{align*}
	Observe that
	\begin{align*}
		Q_0^{-1}(\rho) \leq \dfrac{C}{\sqrt{1+\rho}}, \quad \partial_{\rho}Q_0^{-1}(\rho) \leq \dfrac{C}{\sqrt{1+\rho}}.
	\end{align*}
	We hence obtain with \eqref{eq:sec3:ProofConvErrEst}
	\begin{align*}
		\sup_{\rho\in[0,2M]}\sup_{ \tau \in [0,T/8]} \left| U_T(\tau,\rho) - \rho \right| &\leq CM^{3/2}e^{-T/4},
		\\
		\sup_{\rho\in[0,2M]}\sup_{ \tau \in [0,T/8]} \left| \partial_{\rho} U_T(\tau,\rho) - 1 \right| &\leq CM^{3/2}e^{-T/4}.
	\end{align*}
	In particular, for $ T\geq T_0(M)=c\ln(1+M) $, where $ c>0 $ is some large constant, we get
	\begin{align*}
		\dfrac{1}{2}\leq\partial_{\rho} U_T(\tau,\rho) \leq \dfrac{3}{2}
	\end{align*}
	and thus for $ \rho\in [0,M] $
	\begin{align*}
		\dfrac{3}{2}\rho\leq U_T^{-1}(\tau,\rho) \leq 2\rho\leq 2M.
	\end{align*}
	This then yields
	\begin{align}\label{eq:sec3:ProofConvEstInvCharPrep}
		\begin{split}
			\sup_{\rho\in[0,M]}\sup_{ \tau \in [0,T/8]} \left| U_T^{-1}(\tau,\rho) - \rho \right| &\leq CM^{3/2}e^{-T/4}.
			\\
			\sup_{\rho\in[0,M]}\sup_{ \tau \in [0,T/8]} \left| \partial_{\rho} U_T^{-1}(\tau,\rho) - 1 \right| &\leq CM^{3/2}e^{-T/4}.
		\end{split}
	\end{align}
	We hence obtain 
	\begin{align*}
		Q_T^{-1}(\tau,\rho) - Q_0^{-1}(\tau,\rho) &= Q_0^{-1}(\tau,U_T^{-1}(\tau,\rho))-Q_0^{-1}(\tau,\rho) = \int_\rho^{U_T^{-1}(\tau,\rho)}\partial_{\rho}Q_0^{-1}(\tau,\xi)\, d\xi,
		\\
		\partial_{\rho}Q_T^{-1}(\tau,\rho) - \partial_{\rho}Q_0^{-1}(\tau,\rho) &= \partial_{\rho}Q_0^{-1}(\tau,U_T^{-1}(\tau,\rho))\partial_{\rho}U_T^{-1}(\tau,\rho)-\partial_{\rho}Q_0^{-1}(\tau,\rho).
	\end{align*}
	Using the estimates \eqref{eq:sec3:ProofConvEstInvCharPrep} and $ \partial_{\rho}^2Q_0^{-1}(\rho)\leq C $ yields \eqref{eq:sec3:ProofConvEstInvChar} for $ T\geq T_0(M) $.
	
	Recalling that $ P_T(\tau,\rho) = Q_T(\tau,e^\tau\rho) $ yields with \eqref{eq:sec3:ProofConvEstInvChar}
	\begin{align}\label{eq:sec3:ProofConvEstFullInvChar}
		\begin{split}
			\sup_{\rho\in[0,M]}\sup_{ \tau \in [0,T/8]} \left| P_T^{-1}(\tau,\rho) - e^{-\tau}Q_0^{-1}(\tau,\rho) \right| &\leq CM^{3/2}e^{-T/4},
			\\
			\sup_{\rho\in[0,M]}\sup_{ \tau \in [0,T/8]} \left| \partial_{\rho}P_T^{-1}(\tau,\rho) - e^{-\tau}\partial_{\rho}Q_0^{-1}(\tau,\rho) \right| &\leq CM^{3/2}e^{-T/4}
		\end{split}
	\end{align}
	for $ T\geq T_0(M) $.
	
	\textit{Step 5. Conclusion.} We now obtain from \eqref{eq:sec3:ProofConvProfil}
	\begin{align*}
		\partial_{\rho}H_T(\tau,\rho) &= \partial_{\rho}H_T\left(0,P_T^{-1}(\tau,\rho)\right)e^{\tau}\partial_{\rho}P_T^{-1}(\tau,\rho) 
		\\
		&\quad + e^\tau\int_0^\tau e^{-\sigma}\partial_{\rho}\Rem_T\left( \sigma, P_T\left(\sigma,P_T^{-1}(\tau,\rho)\right) \right)\partial_{\rho}P_T\left(\sigma,P_T^{-1}(\tau,\rho)\right) \partial_{\rho}P_T^{-1}(\tau,\rho) \, d\sigma
		\\
		&= \partial_{\rho}H_T\left(0,P_T^{-1}(\tau,\rho)\right)e^{\tau}\partial_{\rho}P_T^{-1}(\tau,\rho) 
		\\
		&\quad + \int_0^\tau \partial_{\rho}\Rem_T\left( \sigma, Q_T(\sigma,e^{\sigma}e^{-\tau}Q_T^{-1}(\tau,\rho)) \right) \partial_{\rho}Q_T(\sigma,e^{\sigma}e^{-\tau}Q_T^{-1}(\tau,\rho)) \partial_{\rho}Q_T^{-1}(\tau,\rho) \, d\sigma
	\end{align*}
	Note that
	\begin{align*}
		\dfrac{1}{\sqrt{1+2K_0\rho}} = \partial_{\rho}H_T(0,0)\partial_{\rho}Q_0^{-1}(\tau,\rho).
	\end{align*}
	Thus, using Lemma \ref{lem:EstRem2} and \eqref{eq:sec3:ProofConvEstFullInvChar} we get for $ \rho\in [0,M] $ and $ T\geq T_0(M) $
	\begin{align*}
		&\left| \partial_{\rho}H_T(\tau,\rho) - 	\dfrac{1}{\sqrt{1+2K_0\rho}} \right| 
		\\
		& \leq \norm[\infty]{\partial_{\rho}^2H_T}P_T^{-1}(\tau,\rho)\partial_{\rho}Q_0^{-1}(\tau,\rho)  + \norm[\infty]{\partial_{\rho}H_T}\left| e^\tau \partial_{\rho}P_T^{-1}(\tau,\rho)-\partial_{\rho}Q_0^{-1}(\tau,\rho) \right| 
		\\
		&\quad + e^{-T/2} \int_0^\tau \left( 1+ Q_T(\sigma,e^{\sigma}e^{-\tau}Q_T^{-1}(\tau,\rho))^2 \right) \partial_{\rho}Q_T(\sigma,e^{\sigma}e^{-\tau}Q_T^{-1}(\tau,\rho)) \partial_{\rho}Q_T^{-1}(\tau,\rho) \, d\sigma.
	\end{align*}
	In order to estimate the first two terms we use \eqref{eq:sec3:ProofConvEstFullInvChar}. To estimate the last term we use \eqref{eq:sec3:ProofConvEstChar} and \eqref{eq:sec3:ProofConvEstInvChar} for $ \rho\in [0,M] $
	\begin{align*}
		Q_T(\sigma,e^{\sigma}e^{-\tau}Q_T^{-1}(\tau,\rho)) &\leq C\left( 1+ \left( e^{\sigma}e^{-\tau}Q_T^{-1}(\tau,\rho) \right)^3 \right) \leq CM^{9/2},
		\\
		|\partial_{\rho}Q_T(\sigma,e^{\sigma}e^{-\tau}Q_T^{-1}(\tau,\rho))| &\leq C\left( 1+ \left( e^{\sigma}e^{-\tau}Q_T^{-1}(\tau,\rho) \right)^2\right) \leq CM^{3},
		\\
		|\partial_{\rho}Q_T^{-1}(\tau,\rho)| &\leq CM^{3/2}.
	\end{align*}
	We thus obtain for $ \tau\in [0,T/8] $
	\begin{align*}
		\sup_{\rho\in[0,M]}\left| \partial_{\rho}H_T(\tau,\rho) - 	\dfrac{1}{\sqrt{1+2K_0\rho}} \right|  \leq C\left( e^{-\tau}+M^{3/2}e^{-T/4}+M^{14}\tau e^{-T/2}\right).
	\end{align*}
	This implies when setting $ \tau = T/8 $
	\begin{align*}
		\sup_{\rho\in[0,M]} \left| \partial_{\rho}\hat{g}(T+T/8),\rho) - 	\dfrac{1}{\sqrt{1+2K_0\rho}} \right| \leq CM^{14} e^{-T/8},
	\end{align*}
	for $ T\geq T_0(M)=c\ln(1+M) $. Since the left-hand side is bounded we can change constants to obtain for all $ \tau\geq0 $
	\begin{align*}
		\sup_{\rho\in[0,M]}\left| \partial_{\rho}\hat{g}(\tau,\rho) - 	\dfrac{1}{\sqrt{1+2K_0\rho}} \right|  \leq CM^{\kappa}e^{-\tau/10},
	\end{align*}
	where $ \kappa=14+c/8 $. This concludes the proof.
\end{proof}

We now give the proof of Theorem \ref{thm:IntrodLocaliSelfSim}.

\begin{proof}[Proof of Theorem \ref{thm:IntrodLocaliSelfSim}]
	The statements in (i) and (ii) are exactly as in Corollary \ref{cor:MomentEst} and Proposition \ref{pro:Localization}, respectively.
	
	For (iii) we argue as follows. Let
	\begin{align*}
		F_\infty(d\eta) = |\eta|F_s(|\eta|) \delta\left( \dfrac{\eta}{|\eta|}-\dfrac{\theta}{|\theta|} \right) \, d\eta,
		\quad
		\tilde{F}(\tau,d\eta) = \dfrac{|\eta|^2 \, F(\tau,d \eta) }{Z(\tau)}.
 	\end{align*} 
 	Here, $ F_s $ is given in \eqref{eq:sec1:McLeod} with $ K_0=c_0|\theta| $. Note that $ F_s(r) = g_\infty(r)/r^2 $ with $ g_\infty $ as in Theorem \ref{thm:ConvSelfSim}. Consequently, $ F_\infty $ is a probability measure. Recall that $ d\eta = r drd\omega $ using the change of variables $ T(\eta)=(r,\omega)$.
 	
 	For $ \varphi\in C_c^1(\R_{+}^2) $ we have
	\begin{align*}
		\left| \int_{\R^2_+} \varphi(\eta) \, \tilde{F}(\tau,d\eta) - \int_{\R^2_+} \varphi\left(|\eta|\dfrac{\theta}{|\theta|}\right) \, \tilde{F}(\tau,d\eta)\right|  \leq \norm[\infty]{\nabla\varphi}\int_{\R^2_+} |\eta|\norm{\dfrac{\eta}{|\eta|}-\dfrac{\theta}{|\theta|}} \, \tilde{F}(\tau,d\eta).
	\end{align*}
	The latter goes to zero as $ \tau\to \infty $ by Proposition \ref{pro:Localization} and the moment bounds in Corollary \ref{cor:MomentEst}. Furthermore, we have
	\begin{align*}
		\int_{\R^2_+} \varphi\left(|\eta|\dfrac{\theta}{|\theta|}\right) \, \tilde{F}(\tau,d\eta) - \int_{\R^2_+} \varphi\left(|\eta|\dfrac{\theta}{|\theta|}\right) \, F_\infty(\tau,d\eta) = \int_0^\infty \varphi\left(r\dfrac{\theta}{|\theta|}\right) r^2g(\tau,dr) - \int_0^\infty \varphi\left(r\dfrac{\theta}{|\theta|}\right) g_\infty(dr),
	\end{align*}
	with $ g $ as defined in Definition \ref{def:PolarCoordDistr}. The latter expression converges to zero by Theorem \ref{thm:ConvSelfSim}. Using the fact that the family of probability measures $ \tilde{F}(\tau) $ is tight by Corollary \ref{cor:MomentEst}, we hence obtain $ \tilde{F}(\tau) \to F_\infty $ as $ \tau \to \infty $ in the sense of measures. This concludes the proof.
\end{proof}

We now collect our results to yields Theorem \ref{thm:IntrodLocaliSelfSim}.

\begin{proof}[Proof of Theorem \ref{thm:IntrodLocaliSelfSim}]
	Part (i) is a result of Corollary \ref{cor:MomentEst}, while part (ii) is just Proposition \ref{pro:Localization}. Finally, statement (iii) is a consequence of Theorem \ref{thm:ConvSelfSim}. 
\end{proof}

\section*{Acknowledgments}
B. Kepka gratefully acknowledges support of the SNSF through grant PCEFP2\_203059 and the NCCR SwissMAP.
E. Franco gratefully acknowledge the support of the grant CRC 1720 "Analysis of Criticality: from Complex Phenomena to Models and Estimates" (Project-ID
539309657) of the University of Bonn funded through the Deutsche Forschungsgemeinschaft (DFG, German Research Foundation)  and Germany's Excellence Strategy-EXC-2047/2-390685813.
The funders had no role in study design, analysis, decision to publish, or preparation of the manuscript.

\addtocontents{toc}{\protect\setcounter{tocdepth}{1}}

\bibliographystyle{habbrv}
\bibliography{References.bib}

\begin{thebibliography}{10}

\bibitem{bertoin2009two}
J.~Bertoin.
\newblock Two solvable systems of coagulation equations with limited
  aggregations.
\newblock In {\em Annales de l'Institut Henri Poincar{\'e} C, Analyse non
  Lin{\'e}aire}, volume~26, pages 2073--2089. Elsevier, 2009.

\bibitem{blatz1945note}
P.~Blatz and A.~Tobolsky.
\newblock Note on the kinetics of systems manifesting simultaneous
  polymerization-depolymerization phenomena.
\newblock {\em The journal of physical chemistry}, 49(2):77--80, 1945.

\bibitem{cristian2024coagulation}
I.~Cristian and J.~J.~L. Vel{\'a}zquez.
\newblock Coagulation equations for non-spherical clusters.
\newblock {\em Archive for Rational Mechanics and Analysis}, 248(6):113, 2024.

\bibitem{delacour2022mathematical}
J.~Delacour, M.~Doumic, S.~Martens, C.~Schmeiser, and G.~Zaffagnini.
\newblock A mathematical model of p62-ubiquitin aggregates in autophagy.
\newblock {\em Journal of Mathematical Biology}, 84(1):3, 2022.

\bibitem{ferreira2021localization}
M.~A. Ferreira, J.~Lukkarinen, A.~Nota, and J.~J.~L. Vel{\'a}zquez.
\newblock Localization in stationary non-equilibrium solutions for
  multicomponent coagulation systems.
\newblock {\em Communications in {M}athematical {P}hysics}, 388(1):479--506,
  2021.

\bibitem{ferreira2021stationary}
M.~A. Ferreira, J.~Lukkarinen, A.~Nota, and J.~J.~L. Vel{\'a}zquez.
\newblock Stationary non-equilibrium solutions for coagulation systems.
\newblock {\em Archive for {R}ational {M}echanics and {A}nalysis},
  240(2):809--875, 2021.

\bibitem{ferreira2024asymptotic}
M.~A. Ferreira, J.~Lukkarinen, A.~Nota, and J.~J.~L. Vel{\'a}zquez.
\newblock Asymptotic localization in multicomponent mass conserving coagulation
  equations.
\newblock {\em Pure and Applied Analysis}, 6(3):731--764, 2024.

\bibitem{ferreira2025global}
M.~A. Ferreira and S.~Pirnes.
\newblock Global existence of measure-valued solutions to the multicomponent
  {S}moluchowski coagulation equation.
\newblock {\em arXiv preprint arXiv:2504.10306}, 2025.

\bibitem{friendlander2000smoke}
S.~Friendlander.
\newblock Smoke, dust and haze: Fundamentals of aerosol dynamics.
\newblock {\em Oxford University Press, New York, USA}, 20(0):0, 2000.

\bibitem{gueron1995dynamics}
S.~Gueron and S.~A. Levin.
\newblock The dynamics of group formation.
\newblock {\em Mathematical {B}iosciences}, 128(1-2):243--264, 1995.

\bibitem{hoogendijk2024gelation}
J.~Hoogendijk, I.~Kryven, and C.~Schenone.
\newblock Gelation and localization in multicomponent coagulation with
  multiplicative kernel through branching processes.
\newblock {\em Journal of {S}tatistical {P}hysics}, 191(7):91, 2024.

\bibitem{krapivsky1996aggregation}
P.~Krapivsky and E.~Ben-Naim.
\newblock Aggregation with multiple conservation laws.
\newblock {\em Physical Review E}, 53(1):291, 1996.

\bibitem{lasserre2009moments}
J.~B. Lasserre.
\newblock {\em Moments, positive polynomials and their applications}, volume~1.
\newblock World Scientific, 2009.

\bibitem{leyvraz1981singularities}
F.~Leyvraz and H.~R. Tschudi.
\newblock Singularities in the kinetics of coagulation processes.
\newblock {\em Journal of Physics A: Mathematical and General},
  14(12):3389--3405, 1981.

\bibitem{menon2004approach}
G.~Menon and R.~L. Pego.
\newblock Approach to self-similarity in {S}moluchowski's coagulation
  equations.
\newblock {\em Communications on Pure and Applied Mathematics},
  57(9):1197--1232, 2004.

\bibitem{norris2000cluster}
J.~R. Norris.
\newblock Cluster coagulation.
\newblock {\em Communications in Mathematical Physics}, 209(2):407--435, 2000.

\bibitem{PerelsonSamsel1982}
R.~Samsel and A.~Perelson.
\newblock Kinetics of rouleau formation. {I}. {A} mass action approach with
  geometric features.
\newblock {\em Biophysical Journal}, 37(2):493--514, 1982.

\bibitem{PerelsonSamsel1984}
R.~Samsel and A.~Perelson.
\newblock Kinetics of rouleau formation. {II}. {R}eversible reactions.
\newblock {\em Biophysical Journal}, 45(4):805--824, 1984.

\end{thebibliography}

\end{document}